\documentclass[11pt, reqno]{amsart}
\usepackage[T1]{fontenc}
\usepackage[english]{babel}

\usepackage{amsmath}
\usepackage{amssymb}
\usepackage{amsthm}
\usepackage{bbm}
\usepackage{mathrsfs}
\usepackage{epsfig}
\usepackage{enumerate}
\usepackage{mathtools}
\usepackage{graphicx}
\usepackage[usenames,dvipsnames]{xcolor}
\usepackage{appendix}
\usepackage[font=small]{caption}
\usepackage[colorlinks=true,linkcolor=NavyBlue,urlcolor=RoyalBlue,citecolor=PineGreen,bookmarks=true,bookmarksopen=true,bookmarksopenlevel=2,unicode=true,linktocpage]
{hyperref}

\newtheorem{theoremi}{Theorem}[]

\newtheorem{theorem}{Theorem}[section]
\newtheorem{lemma}[theorem]{Lemma}

\newtheorem{proposition}[theorem]{Proposition}
\newtheorem*{proposition*}{Proposition}
\newtheorem*{lemma*}{Lemma}
\newtheorem*{theorem*}{Theorem}
\newtheorem{corollary}[theorem]{Corollary}
\newtheorem{definition}[theorem]{Definition}
\newtheorem{remark}[theorem]{Remark}

\newcommand{\M}{\mathsf{M}}

\newcommand{\nor}[1]{|\! |\! |#1 |\! |\! |_0}
\newcommand{\nord}[1]{|\! |\! |#1 |\! |\! |_2}

\DeclareMathOperator{\Range}{Range}

\def\d{\;{\rm d}}
\def\nsd{{\rm d}}
\def\R{\mathbb{R}}

\def\P{\mathbb{P}}
\def\Pmu{\mathbb{P}^{\M}}
\def\Emu{\mathbb{E}^{\M}}
\def\Omu{(\Omega^\M,\mathcal{F}^\M,\mathbb{P}^{\M})}
\def\C2{{\sf{C}_2}}
\def\PX{\mathbb{P}^{X}}
\def\EX{\mathbb{E}^{X}}
\def\OX{(\Omega^X,\mathcal{F}^X,\mathbb{P}^{X})}

\def\epsilon{\varepsilon}

\def\build#1_#2^#3{\mathrel{\mathop{\kern 0pt#1}\limits_{#2}^{#3}}}

\def\XXint#1#2#3{{\setbox0=\hbox{$#1{#2#3}{\int}$}
     \vcenter{\hbox{$#2#3$}}\kern-.5\wd0}}

\title[Stochastic integration in Gaussian multiplicative chaos]{Integration and stochastic integration \\ in Gaussian multiplicative chaos}
\author{Isao Sauzedde}

\address{Isao Sauzedde -- LPSM, Sorbonne Universit\'e, Paris}
\email{isao.sauzedde@lpsm.paris}


\begin{document}

\begin{abstract}
We show that for $\gamma<\sqrt{4/3}$, it is possible to define the Lévy area of a planar Brownian motion with the Liouville measure of intermittency parameter $\gamma$ as the underlying area measure. We also consider the case of smoother curves, and study some properties of the integration map thus defined.
%
%
%
%
\end{abstract}

\maketitle
\setcounter{tocdepth}{1}
\tableofcontents

\section*{Introduction}
In a previous paper \cite{LAWA}, we showed a Green formula for $\alpha$-H\"older continuous planar curves (with $\alpha>\frac{1}{2}$) and for the planar Brownian motion. It allows us to compute the Young or Stratonovich integral of a smooth differential $1$-form $\eta$ along the curve $X$ as the surface integral of the winding function $\theta_X$ against the $2$-form $\nsd \eta$. 
In the Brownian case, the surface integral has be interpreted through a principal value method.

In the present paper, we look at the same integrals but with the smooth $2$-form $\nsd \eta$ replaced with a random measure, the Liouville measure of parameter $\gamma$. For a planar Brownian motion independent of the Liouville measure, we show that, for $\gamma<\sqrt{4/3}$, the principal value is still well defined and lies in $L^2$ of the product probability space. We also prove continuity, H\"older continuity, and additivity (`Chen's relation') of the integral thus defined. For the case of $\alpha$-H\"older continuous curves, we show similar results, provided $\gamma$ is small enough and $\alpha$ is large enough.



%
%
%

\medskip

For a continuous curve $X:[s,t]\to \R^2$, we denote by $\bar{X}$ the concatenation of $X$ with the straight line segment from $X_t$ to $X_s$. This is an oriented loop, and for all point $z\in \R^2\setminus \Range(\bar{X})$, we denote by $\theta_X(z)\in \mathbb{Z}$ the winding of this oriented loop $\bar{X}$ around the point $z$. Provided $X$ is a smooth curve, Green's theorem is easily seen to implies that for all smooth $1$-form $\eta$,
\begin{equation}
\label{eq:green}
\int_{\bar{X}} \eta = \int_{\R^2} \theta_X \d \eta.
\end{equation}
In \cite{LAWA}, we showed that this equality extends to all $\alpha$-H\"older continuous curves $X$ with $\alpha>\frac{1}{2}$. For this, the left-hand side must be interpreted as a Young integral. We further proved that a similar equality also holds for the Brownian motion (in the particular case $\eta=x\d y$, but we can deduce the general case with a few extra work), provided we interpret the left-hand side as a Stratonovich integral and the right-hand side as a kind of principal value. To be more specific, and using superscripts to indicates the coordinates in $\R^2$, and denoting the Lebesgue measure by $\lambda$, one has almost surely
\begin{align}
\nonumber
\int_0^1 X^1\circ\nsd X^2 +\int_0^1 ((1-r)X_1+r X_0)^1 (X_0-X_1)^2 \d r &= \lim_{K\to +\infty} \int_{\R^2} \mathbbm{1}_{\{ |\theta_X(z)|\leq K\}}\,  \theta_X(z) \d \lambda(z)\\
&\hspace{-1.5cm}=\lim_{K\to +\infty} \int_{\R^2} \max(-K, \min(\theta_X(z),K))\d \lambda(z). \label{eq:greenStoch}
\end{align}
The goal of this paper is to study the right-hand sides of \eqref{eq:green} and \eqref{eq:greenStoch}, but with $\!\d \eta$ (resp. $\nsd z$) replaced with a Gaussian multiplicative chaos, that we denote by $\M$. 

In this paper, $\M$ will be a random measure heuristically described by the formula
\begin{equation}\label{eq:defmuheuri}
 \forall A\in \mathcal{B}(\R^2), \qquad \M(A)=\int_A \exp\big( \gamma \Phi_z -\tfrac{\gamma^2}{2}\mathbb{E}[\Phi_z^2]\big) \d \lambda(z),
 \end{equation}
where $\Phi$ is a centered Gaussian field, and $\gamma\in [0,2)$.  It is assumed that the covariance kernel $K:\R^2\times\R^2 \to \R_+\cup\{\infty\}$ of $\Phi$ takes the form
\[ K(z,w)= \log_+(|z-w|^{-1}) +g(z,w),\]
where $\log_+$ is the positive part of the logarithm and $g: \R^2\times\R^2 \to \R$ is a bounded and $\mathcal{C}^2$-function
with bounded derivatives up to order $2$. The kernel $K$ is well defined, as well as the random field $\Phi$, even though the logarithmic divergence of $K$ makes $\Phi$ a random distribution rather than a random function. However, precisely because $\Phi$ is a random distribution, \eqref{eq:defmuheuri} does not make sense, and the construction of $\M$ cannot rely on this formula. The first construction of a random measure that is a reasonable candidate to be a mathematical incarnation of \eqref{eq:defmuheuri} was given using the theory of multiplicative chaos by Kahane \cite{Kahane}. For an introduction to Gaussian multiplicative chaos, see also \cite{Rhodes} and \cite{Berestycki2}. In general, the larger $\gamma$, the more irregular the measure. Most of the time in this paper, we will be working in the case where $\gamma<\sqrt{2}$. We will thus be in the so-called `$L^2$-phase', in which it is relatively easy to define and to study $\M$, using martingale methods and Hilbertian techniques.

For the reader more familiar with multiplicative chaos, let us mention that the condition $g\in \mathcal{C}^2$ is not necessary to define the measure $\M$, nor to study the integral along H\"older paths, but it is a practical one when we look at the integral along Brownian paths.

We will denote by $(\Omega^\M, \mathcal{F}^\M,\Pmu)$ the probability space on which $\M$ is defined. We also let $(\Omega^X, \mathcal{F}^X,\PX)$ be a second probability space, on which a planar Brownian motion $X=(X_t)_{t\in [0,1]}$ is defined. Finally, we denote by $(\Omega, \mathcal{F},\mathbb{P})$
the product probability space.

%
%
%

The first part of this paper (Sections 1 to 5) is dedicated to the proof of the following result. Let us emphasize that we stop the Brownian motion at time $1$.

\begin{theoremi}
  \label{th:1}
  Assume $\gamma<\sqrt{4/3}$. Then, $\mathbb{P}$-almost surely, the integral
  \[ \int_{\R^2} \max(-K, \min(\theta_X(z),K)) \d \M(z)\]
  admits a limit $\mathbb{A}^X$ as $K\to +\infty$.

  Moreover, for all $p\in[2,\frac{4}{\gamma^2})$, this limit lies in $L^p(\Omega^X, L^2(\Omega^\M))$.\footnote{
    The reader may find surprising, as we do, that the integrability with respect to $\PX$ depends on $\gamma$, and not the integrability with respect to $\Pmu$. However, the result as it is written is what we mean.
  }
\end{theoremi}

For $(s,t)\in \Delta=\{(s,t)\in[0,1]^2: s\leq t\}$, we denote
by $\mathbb{A}_{s,t}^X$ the almost surely defined limit,
\begin{equation}
\label{eq:def:Abarintro:MB}
\mathbb{A}_{s,t}^X= \lim_{K\to +\infty}
\int_{\R^2} \max(-K, \min(\theta_{X_{|[s,t]}} (z),K)) \d \M(z).
\end{equation}
The second part of the paper (Section 6 to 9) is devoted to the study of the map $(s,t)\mapsto \mathbb{A}_{s,t}^X$, and also, for a curve $Y$ deterministic and more regular than a Brownian motion, of the map
$\mathbb{A}_{s,t}^Y$ given by
\begin{equation}
\label{eq:def:Abarintro:holder}
\mathbb{A}_{s,t}^Y=\int_{\R^2} \theta_{Y_{|[s,t]}} \d \M.
\end{equation}

For a function $\mathbb{A}:\Delta\to\R$ to deserve the name of `algebraic $\M$-area enclosed by the curve $Z$', it should satisfy
a relation that we will call a Chen relation. To write this relation, let us denote, for all $s\leq u\leq t$, by $T_{s,u,t}$ the triangle with vertices $Z_s$, $Z_u$ and $Z_t$. Let us also set $\epsilon_{s,u,t}=1$ if $Z_s$, $Z_u$ and $Z_t$ are found in this cyclic order along the positively oriented boundary of $T_{s,u,t}$, and $\epsilon_{s,u,t}=-1$ otherwise.  We say that $\mathbb{A}$ satisfies the {\em Chen relation} (relative to $Z$) if for all $s\leq u\leq t$,
\[ {\mathbb{A}}_{s,t}={\mathbb{A}}_{s,u}+{\mathbb{A}}_{u,t}+\epsilon_{s,u,t}\M(T_{s,u,t}).
\]

In a nutshell, the results of this second part can be summarized as follows (some additional properties are given below).
\begin{theoremi}
  Let $\gamma<2(\sqrt{2}-1)$, and $Z:[0,1]\to \R^2$ be either a planar Brownian motion or an $\alpha$-H\"older continuous function with $\alpha>\big(2(1-\frac{\gamma^2}{4} )\big)^{-1}$. Then, for all $(s,t)\in\Delta$, almost surely, $\mathbb{A}_{s,t}^Z$ is well-defined. It admits a modification which is continuous, H\"older continuous, and satisfies the Chen relation.
\end{theoremi}

Let us also say that $\mathbb{A}$ is $\beta$-regular if there exists a constant $C$ such that for all $(s,t)\in \Delta$,
\begin{equation}
\label{eq:defbetareg}
\mathbb{A}_{s,t}\leq C(t-s)^{\beta}.\footnote{The relation between this notion of regularity and H\"older continuity is given by Lemma \ref{le:RegToHold}.}
\end{equation}


%

From now on, we set $\nu=2(1-\frac{\gamma^2}{4})$.\footnote{
This number $\nu$ is the Hausdorff dimension of a Borel set of full $\M$-measure. However, the reason why it matters to us is that it is equal to $\frac{\xi(q)}{q}|_{q=2}=\frac{\xi(2)}{2}$, where $\xi$ the so-called structure exponent of $\M$. This exponent is defined by the fact that for $r>0$, the expectation of $\M(B(0,r))^q$ is of the order of $r^{\xi(q)}$.}

\begin{theoremi}
Let $\gamma<\sqrt{4/3}$. Then,  $(\mathbb{A}_{s,t}^X)_{(s,t)\in\Delta}$  admits a modification which satisfies the Chen relation.

Let $\gamma<2(\sqrt{2}-1)\simeq 0.81$. Then, $(\mathbb{A}_{s,t}^X)_{(s,t)\in\Delta}$  admits a modification which satisfies the Chen relation, is $\beta$-regular for all $\beta<\min (\frac{1}{2}-\frac{\gamma^2}{4}, 1+\frac{\gamma^2}{4}-\sqrt{2}\gamma)$, and
$\beta$-H\"older continuous for all
  $\beta<\frac{1}{2}( 1+\frac{\gamma^2}{4}-\sqrt{2}\gamma)$.
\end{theoremi}

\begin{theoremi}
1.  For all $\gamma<2$ and $\alpha>\frac{1}{2}$, for all $\alpha$-H\"older continuous function $Y:[0,1]\to\R^2$, for all $s<t$, $\Pmu$-almost surely, the integral \eqref{eq:def:Abarintro:holder} is well-defined. The function $\mathbb{A}^Y$ admits a modification which satisfies the Chen relation,

2.1.  For all $\gamma<\sqrt{2}\simeq 1.4$ and $\alpha>\frac{1}{\nu}$, the function $\mathbb{A}^Y$ admits a modification which satisfies the Chen relation, is almost surely $\beta$-regular for all
\[
\beta<\beta_0=\left\{
\begin{array}{ll}
\alpha\nu-1 &\mbox{if } \alpha\geq \gamma^{-2},\\
2\alpha(1+\frac{\gamma^2}{4}) -2\gamma\sqrt{\alpha} &\mbox{if } \alpha\in [\frac{3-2\sqrt{2}}{2} \gamma^{-2}, \gamma^{-2}]\\
\alpha\nu-\frac{1}{2} &\mbox{if } \alpha\leq\frac{3-2\sqrt{2}}{2} \gamma^{-2}.
\end{array}
\right.
\]

2.2.  For all $\gamma<2(\sqrt{2}-1)\simeq 0.81$ and $\alpha>\frac{1}{\nu}$, this modification is $\beta'$-H\"older continuous for  $\beta'=\min((1-\sqrt{2}\gamma+\frac{\gamma^2}{4})\alpha , \beta)$, for all $\beta<\beta_0$.
\end{theoremi}
Figure \ref{fig:phase} below represents the different ranges that appear in these results.
\begin{figure}[h!]
    \includegraphics[width=.4\linewidth]{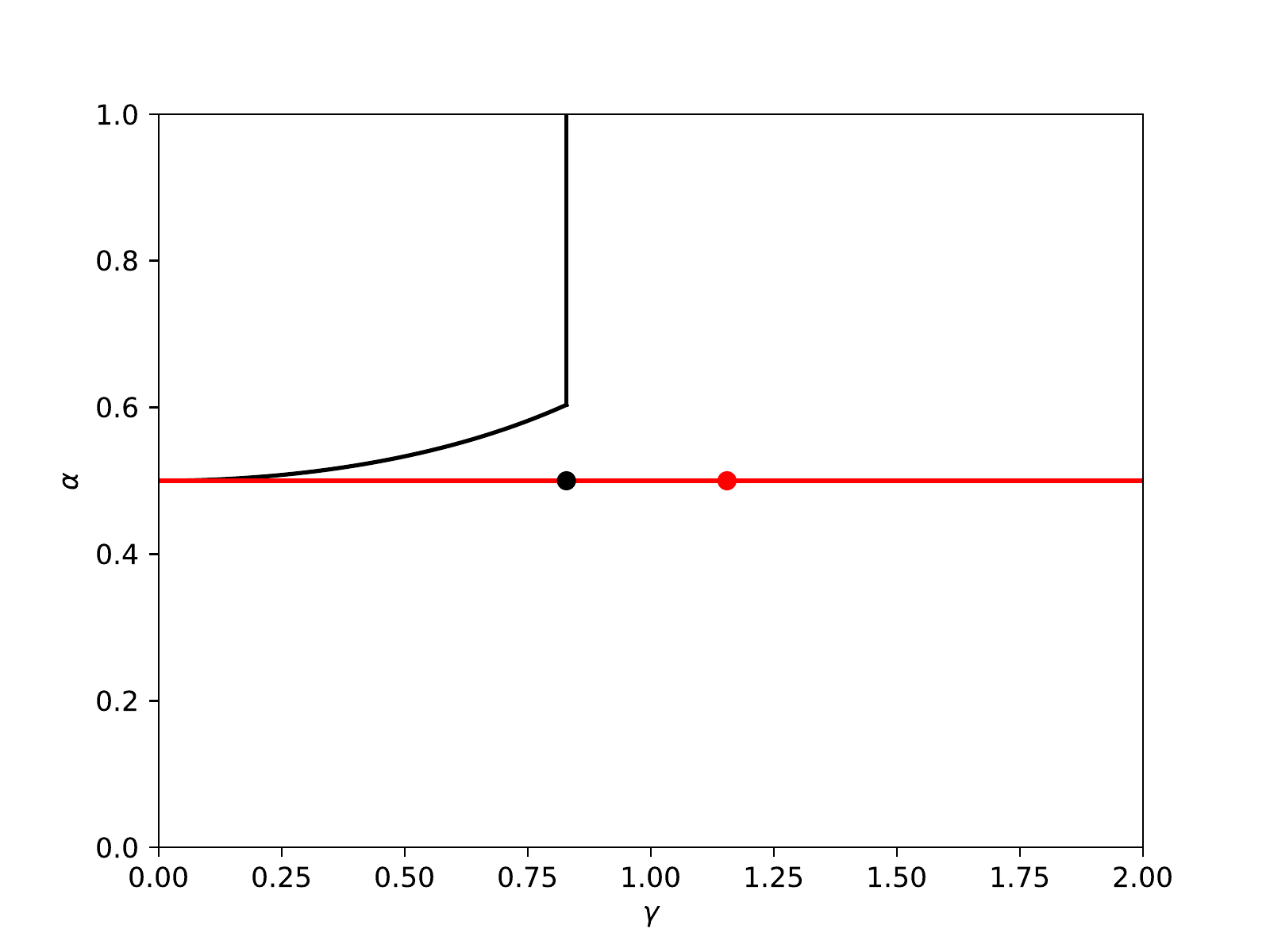}
  \hspace{0.3cm}
    \includegraphics[width=.4\linewidth]{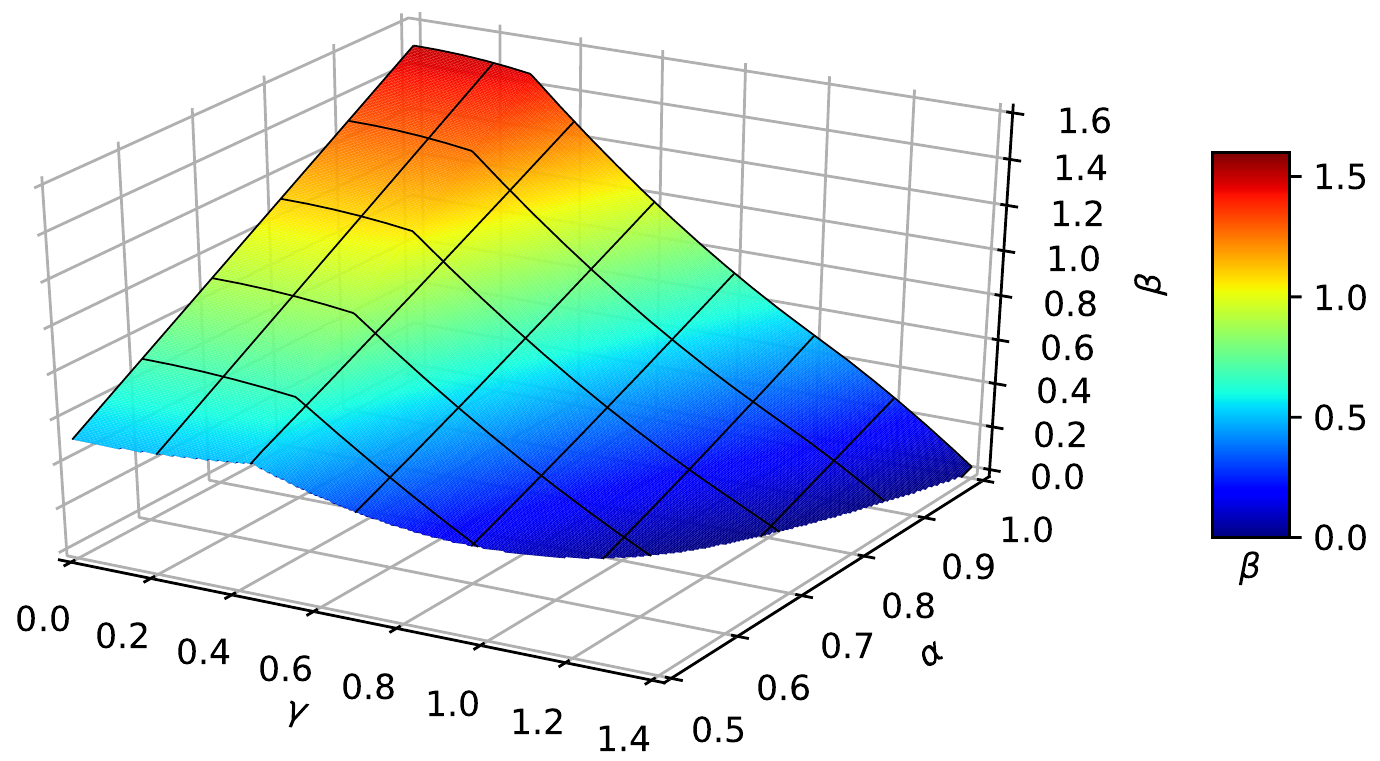}

\begin{center}
\mbox{
\includegraphics[width=.4\linewidth]{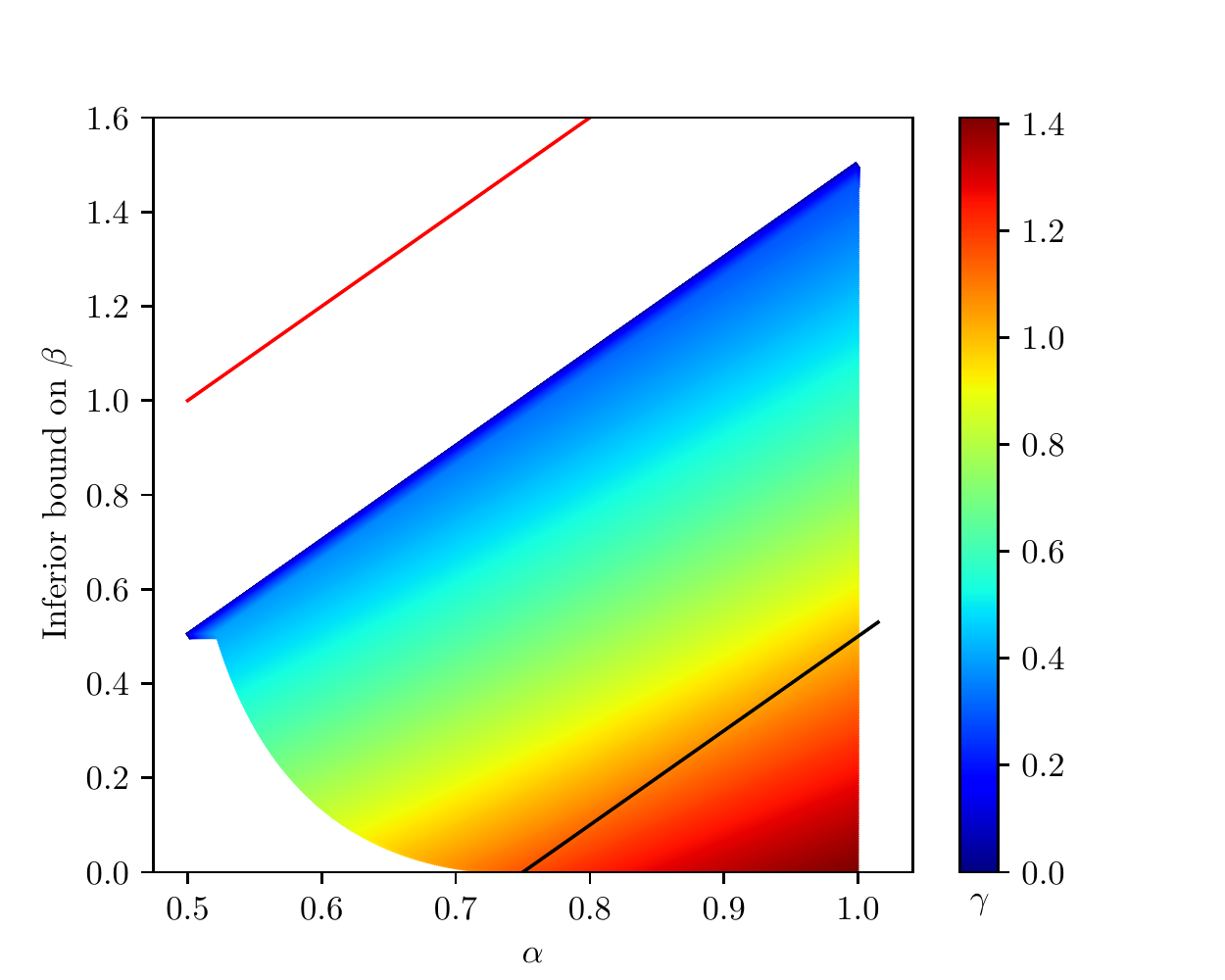}   \hspace{-0cm}
\includegraphics[width=.4\linewidth]{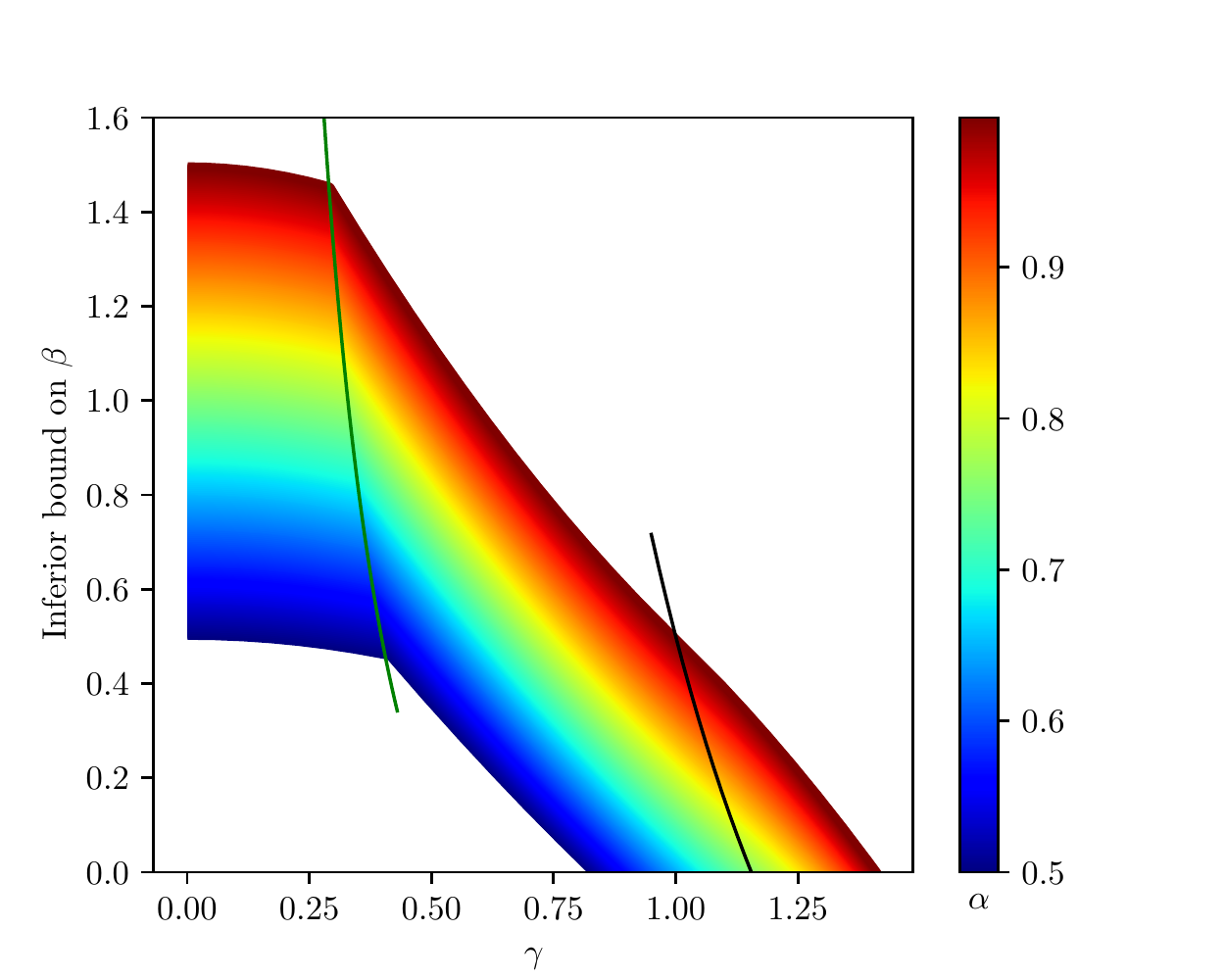}
}
\caption{
\label{fig:phase}
  \textbf{Top left:}  Above the red line, $\mathbb{A}_{s,t}$ is well-defined for given $s$ and $t$, and satisfies the Chen relation (up to modification). Above the black curve, it admits a continuous modification.
  On the left of the red point, $\mathbb{A}_{s,t}$ is well-defined for a Brownian motion, and satisfies the Chen relation.
   On the left of the black point, it admits a continuous modification.
  Remark the black dot is below the black curve.
  \textbf{Top right, bottom left, bottom right:} Lower bound on the exponent $\beta=\beta(\alpha, \gamma)$ of regularity.
  The black curves on the bottom figures correspond to the values $\alpha=\gamma^{-2}$ for which the expression of our lower bounds changes.
  The green curve in the bottom right figure corresponds to the values $\alpha=\frac{3-2\sqrt{2}}{2\gamma^2}$, for which the expression of our lower bounds changes a second time.
     The red line on the bottom left figure is the correct value of $\beta$  in the classical case $\M=\lambda$ (`$\gamma=0$')
  }
\end{center}
\end{figure}

During the proof, we will also show that $\mathbb{A}^Y_{s,t} $ is `stochastically $2\alpha$-regular', in the following sense.
\begin{theoremi}
Assume that either $\gamma\in[0,2)$, $\alpha>\frac{1}{2}$ and $\rho=1$, or
that $\gamma<\sqrt{2}$, $\alpha>\frac{1}{\nu}$ and $\rho\in[1,\alpha \nu)$. Then, for all $(s,t)\in \Delta$, $\mathbb{A}^Y_{s,t}$ lies in $L^\rho(\Omega^\M)$, and there exists a constant $C$ which does not depend on $(s,t)$ and such that
$\| \mathbb{A}^Y_{s,t} \|_{L^\rho(\Omega^\M)  }\leq C|t-s|^{2\alpha}$.
\end{theoremi}

The two parts of the paper (Sections 1 to 5,  Sections 6 to 10) are mostly independent, but reading Proposition 1.1 and Lemmas 2.1 and 2.2 is necessary to understand the second part.

\section{Strategy of the proof of Theorem 1}
\label{sec:1}
Let us recall that $X:[0,1]\to \R^2$ is a planar Brownian motion defined on $\OX$, and $\theta_X$ is the winding function associated with $X$. The random measure $\M$ is defined on $\Omu$, and is formally given by $\M=e^{\gamma\phi-\frac{\gamma^2}{2}\mathbb{E}[\phi^2]}\lambda$, where $\lambda$ is the Lebesgue measure and $\phi$ is a centered Gaussian field on $\R^2$ with covariance kernel $K:(z,w)\mapsto \log_+(|z-w|)+g(z,w)$,
for a bounded $\mathcal{C}^2$ function $g$ with bounded derivatives up to order $2$. From now on and until Section 5 included, we will assume that $\gamma<\sqrt{4/3}$.
 We are looking at the existence of a limit, as $K\to +\infty$, of the  integral with `cut-off'
\[ \int_{\R^2} \max(-K, \min(\theta_X(z),K)) \d z.\]
Since $\theta_X$ takes its values in $\mathbb{Z}$, we can rewrite this integral
as
\[ \sum_{n=-\infty}^{+\infty} \max(-K,\min(n,K)) \M( \{ z\in\R^2: \theta_X(z)=n \} ),\]
which is more conveniently written
\[ \sum_{n=1}^{+\infty} \min(n,K) \big(\M( \{ z\in\R^2: \theta_X(z)=n \} ) -\M( \{ z\in\R^2: \theta_X(z)=-n \} ) \big).\]
Performing a summation by parts, this is finally equal to
\[
\sum_{N=1}^{K} \big( \M( \{ z\in\R^2: \theta_X(z)\geq N \} )-\M( \{ z\in\R^2: \theta_X(z)\leq -N \} ) \big).
\]
This is the expression we will study. We will show that the general term goes to zero sufficiently fast for the sum to be convergent. Consequently, for $N\geq 1$, we define the following sets:
\begin{equation}\label{eq:defDN}
\mathcal{D}_N=\{ z\in\R^2\setminus{\rm Range}(\bar{X}): \theta_X(z)\geq N \}, \ \   \mathcal{D}_{-N}=\{ z\in\R^2\setminus{\rm Range}(\bar{X}): \theta_X(z)\leq -N \}.
\end{equation}
For a random variable $Z$ on $\Omega$, we set
\[\|Z\|_{p,q}=\|Z\|_{L^p(\Omega^X,L^q(\Omega^\M))} =\EX[ \Emu[Z^{q}]^{\frac{p}{q}}]^{\frac{1}{p}}.\]
Theorem \ref{th:1} is then a direct consequence of the following estimation.
\begin{proposition}
\label{le:1}
For all $p\in[2,\frac{4}{\gamma^2})$, there exists $\epsilon>0$ and $C$ such that for all $N$,
\[ \|\M(\mathcal{D}_N)-\M(\mathcal{D}_{-N})\|_{p,2}
\leq C N^{-1-\epsilon}.\]
\end{proposition}
Let us remark that this is only possible thanks to the compensation between $\M(\mathcal{D}_N)$ and $\M(\mathcal{D}_{-N})$. Indeed, each of these sets has a Lebesgue measure which is equivalent as $N$ tends ton infinity (in $L^q$ for all $q$, and in the almost sure sense) to $\frac{1}{2\pi N}$, so that even $\|\M(\mathcal{D}_N)\|_{p,1}$ is two large for the sum to be convergent.
The fact that a compensation occurs can be informally understood as a consequence of the fact that the sets $\mathcal{D}_N$ and $\mathcal{D}_{-N}$ are extremely similar when $N$ is large. 
For a better understanding of the asymptotic behaviour of $\mathcal{D}_N$, we refer to our article \cite{bwe}.

The strategy for the proof of Proposition \ref{le:1} is similar to the one we used in \cite{LAWA} (Sections 3.2 and 3.3), and then in \cite{bwe}. In short, we will decompose the quantity $\M(\mathcal{D}_N)$ into a sum of smaller quantities, corresponding to contributions of pieces of the Brownian trajectory.

To this end, we fix some $t>0$, which should think of as being very small. For a positive integer $N$, we set $T=\lfloor N^t \rfloor$.
For $i\in \{1,\dots, T\}$, we define $X^i$ to be the restriction of $X$ to $[(i-1)T^{-1},iT^{-1}]$, $\bar{X}^i$ the concatenation of $X^i$ with the segment between its endpoints,
and $\theta^i(z)$ the number of times $\bar{X}^i$ winds around $z$. 
We set
\begin{equation}\label{eq:defDiN}
\mathcal{D}^i_N=\{ z\in \R^2\setminus\Range(\bar{X}^i): \theta^i(z)\geq N\}, \ \  \mathcal{D}^i_{-N}=\{ z\in \R^2\setminus\Range(\bar{X}^i): \theta^i(z)\leq - N\}.
\end{equation}
Finally, we define
\begin{equation}\label{eq:defmuN}
\M_N=\sum_{i=1}^T  \M(\mathcal{D}^i_N ).
\end{equation}
This random variable should be thought of as a proxy for $\M(\mathcal{D}_N)$, when $N$ is large. To control the difference between $\M(\mathcal{D}_N)$ and $\M_N$ will be the subject of Section \ref{sec:comparison}, in which we will prove the following bound.
\begin{proposition*}[\ref{le:main2}]
  Let $p\in(1,+\infty)$ and $\epsilon>0$.
  Then, for $t$ small enough, there exists $C$ such that for all  $N\geq 1$,
  \[
  \|\M(\mathcal{D}_N)-\M_N\|_{p,2}
  \leq
CN^{-\frac{3\nu}{4}+2 t+\epsilon}.
   \]
\end{proposition*}

In the case where the random measure $\M$ has a law invariant by translation, the variables $ \M(\mathcal{D}^i_N )$ are identically distributed. If $\M$ also has a nice behaviour under scaling, then these variables are distributed as a scaled version of $\M(\mathcal{D}_N)$, and we end up with the informal relation
\[ \M(\mathcal{D}_N)-\M(\mathcal{D}_{-N})\simeq \sum_{i=1}^T  (\M(\mathcal{D}^i_N )- \M(\mathcal{D}^i_{-N}) )=\sum_{i=1}^T f(T) Z_i, \]
where the $Z_i$ are random variables distributed as $\M(\mathcal{D}_N)-\M(\mathcal{D}_{-N})$. Let us assume that we can bound nicely the scaling function $f(T)$, and show that the variables $Z_i$ are not too strongly correlated. Then, inserting a bound on
$\M(\mathcal{D}_N)-\M(\mathcal{D}_{-N})$ on the right-hand side might lead to a \emph{better} bound on $\M(\mathcal{D}_N)-\M(\mathcal{D}_{-N})$. We refer to this as the \emph{bootstrap}.
The control on the correlation takes the following form.

{\bf Notation.} The notation $f(N)\leq N^{a+o(1)}$ means that for all $h>0$, there exists $C$ such that for all $N\geq 1$, $f(N)\leq CN^{a+h}$.

\begin{lemma*}[\ref{le:outdiagonal}]
  For $t,\epsilon>0$, and $i,j\in \{1,\dots,T\}$, let $
   F_{i,j}=\{ | X_{(i+1)T^{-1}}-X_{jT^{-1}}|\geq T^{-\frac{1}{2}+\epsilon } \}$.
  For all $p\in\big[2,\frac{4}{\gamma^2})$, there exists a constant $C$ such that for all $N\geq 1$,
  \begin{multline}
  \EX \big[ \big|\Emu\big[
   \sum_{i,j=1}^T \mathbbm{1}_{F_{i,j} }  (\M (\mathcal{D}^i_N)-\M (\mathcal{D}^i_{-N})(\M (\mathcal{D}^j_N)-\M (\mathcal{D}^j_{-N})) \big]\big|^{\frac{p}{2}} \big]^{\frac{1}{p}}\\
  \leq C \log(N+1)^\frac{1}{p} T^{\frac{\epsilon}{2}} N^{-1} \big(
  N^{-\frac{1}{4}+o(1)}+ T^{ \frac{\gamma^2}{4}-\frac{1}{p} }
  \big).
  \end{multline}
\end{lemma*}
The bootstrap then takes the following form. In the statement, there appears a set $\mathcal{T}(\M)$ of random measures that will be defined precisely in Section \ref{sec:section2}. For the moment, suffice it to say that all the elements of $\mathcal{T}(\M)$ look like $\M$, in the sense that they are deduced from $\M$ by (possibly random) translations and symmetries. In many cases of interest, namely when $g$ is invariant by translation, $\mathcal O(\M)=\{\M\}$.

\begin{lemma*}[\ref{le:diagBoundBootstrap}]
  Consider $t>0, p\in[2,\frac{4}{\nu}]$ and $\zeta\in \R$. Assume that
  there exists a constant $C$ such that for all $\M'\in \mathcal{T}(\M)$ and all $N\geq 1$,
  \[
  \|\M'(\mathcal{D}_N )-\M'(\mathcal{D}_{-N} )\|_{p,2}\leq C N^\zeta .
  \]
  Then, for all $\epsilon>0$, there exists a constant $C'$ such that for all $N\geq 1$,
  \[\EX\big[\Emu\big[\hspace{-0.2cm} \sum_{1\leq i\leq j\leq T } \hspace{-0.2cm} \mathbbm{1}_{F_{i,j}^c}  |(\M(\mathcal{D}^i_N )-\M(\mathcal{D}^i_{-N} ))(\M(\mathcal{D}^j_N )-\M(\mathcal{D}^j_{-N} ) )| \big]^{\frac{p}{2}}\big]^\frac{1}{p}  \leq C' \log(T+1)^{\frac{1}{p}}  T^{\frac{\gamma^2}{4}-\frac{1}{p}+\epsilon }N^\zeta.
  \]
\end{lemma*}

These propositions, properly combined, will allow us to prove the theorem. During the proof, we will use the following estimation that we proved in \cite[{Theorem 1.5}]{LAWA}. For a measurable set $A$, the notation $|A|$ always denotes the Lebesgue measure of $A$.
\begin{lemma}
\label{le:estimate:LAWA}
For all $p\in[2,+\infty)$ and all $\delta<\frac{1}{2}$, there exists a constant $C$ such that for all $N\in \mathbb{N}$,
\[
\mathbb{E}\Big[\big| N |\mathcal{D}_N| -\frac{1}{2\pi}\big|^p \Big]^{\frac{1}{p}}\leq C N^{-\delta}.
\]
\end{lemma}

The next section contains the material that will allow us to estimate the $\Pmu$-expectation of some functionals of $\M$.

\section{Three estimates about the Liouville measures}
\label{sec:section2}

%
%


In this section, we assume that $\gamma<\sqrt{2}$. Recall that $\nu$ is the positive real number
\[\nu=2-\frac{\gamma^2}{2}.
\]




One of the advantages of working in the $L^2$-phase (that is, under the assumption $\gamma<\sqrt{2}$) is that one has the following explicit formula.
\begin{lemma}
\label{le:moment2}
  Let $A,B$ be two Borel subsets of $\R^2$ with finite Lebesgue measure. Then,
  \[
  \Emu[ \M(A)\M(B)]= \int_{A\times B} \exp(\gamma^2 K(z,w)) \d z\d w.
  \]
\end{lemma}
\begin{proof}
  We first do an informal computation, which is entirely valid when the kernel $K$ is replaced with a continuous kernel $\tilde{K}$ (in which case the centered Gaussian field $\phi$ with kernel $\tilde{K}$ is defined pointwise):
  \begin{align*}
  \Emu[\M(A)\M(B)]
  &= \Emu\Big[ \int_{A\times B} e^{\gamma \phi_z- \frac{\gamma^2}{2} \Emu[ \phi_z^2] } e^{ \gamma \phi_w- \frac{\gamma^2}{2} \Emu[ \phi_w^2] } \d z\d w \Big] \\
  &= \int_{A\times B} \Emu\big[e^{\gamma (\phi_z+\phi_w)- \frac{\gamma^2}{2} \Emu[ (\phi_z+\phi_w)^2] }\big] e^{ \gamma^2 \tilde{K}(z,w) } \d z\d w \\
  &=\int_{A\times B} e^{\gamma^2 \tilde{K}(z,w)} \d z\d w.
  \end{align*}

For a smooth mollifier $\theta$ and a centered Gaussian field $\Phi$ with kernel $K$, we define $\theta_\epsilon$ as the function $\epsilon^{-2} \theta(\tfrac{\cdot}{\epsilon})$, and $\Phi_\epsilon$ as the convolution of $\Phi$ with $\theta_\epsilon$. For a Borel measurable set $A$ with finite Lebesgue measure, set also
\[\M_\epsilon(A)= \int_A e^{\gamma \Phi_\epsilon(z)-\frac{\gamma^2}{2}\mathbb{E}[(\Phi_\epsilon(z))^2]    } \d \lambda(z). \]
Then, when $\epsilon\to 0$, $\M_\epsilon(A)$ converges in $L^2$ toward $\M(A)$
(see for example \cite[Theorem 2.3]{houches}\footnote{
The cited theorem states convergence in probability, without assuming $\gamma<\sqrt{2}$, but the proof is about the $L^2$ convergence provided $\gamma<\sqrt{2}$. Remark that it is taken as a definition for the measure $\M$ (up to some argument to make it not only a collection of random variables, but a measure). The original construction of Kahane is rather different.}).

It follows that for any two Lebesgue measurable sets $A$ and $B$ with finite Lebesgue measures,
\[
\Emu[\M(A)\M(B)]=\lim_{\epsilon\to 0}\Emu[\M_\epsilon(A)\M_\epsilon(B)].
\]

On the other hand, the previous computation shows that
\begin{align*}
\Emu[\M_\epsilon(A)\M_\epsilon(B)]&=\int_{A\times B} e^{\gamma^2 \Emu[X_\epsilon(z)X_\epsilon(w)]} \d z\d w\\
&\underset{\epsilon\to 0}{\longrightarrow}\int_{A\times B} e^{\gamma^2 K(z,w)} \d z\d w.
\end{align*}
The last convergence follows from pointwise convergence of $\mathbb{E}[X_\epsilon(z)X_\epsilon(w)]$ toward $K(z,w)$, and dominated convergence theorem (see \cite[Theorem 2.3]{houches} again).
\end{proof}

According to \cite{}, the scaling properties of the measure $\M$ are such that, in the case where $K(z,w)= \log_+(\ |z+w|^{-1})$, for a set $A\subset B(0,\frac{1}{2})$ and $r<1$, we have
\[\Emu[\M(r A)^2]=r^{2 \nu} \Emu[\M(A)^2],\]
where $\nu$ is the constant defined at the beginning of this section. We deduce that for each such set $A$, there exists a constant $C$ such that for all $r\leq 1$,
\[\Emu[\M(r A)^2]\leq C |r A|^{\nu}.\]
We can actually choose $C$ such that the two terms are equal.
The following lemma states that the constant $C$ can be chosen to be uniform over all measurable sets $A$.
\begin{lemma}
  \label{le:firstBound}
  There exists $C$ such that for all measurable set $A\subset \R^2$,
  \[
  \Emu[\M(A)^2]\leq C  (|A|^{\nu}+|A|^2).
  \]
\end{lemma}
\begin{proof}
  From the previous lemma, we know that
  \[\Emu[\M(A)^2]=\int_{A^2} \exp(\gamma^2 K(z,w)) \d z \d w.\]

  Since $K(z,w)\leq \log_+(|z-w|^{-1}) +c$ for some constant $c$,
  \[\Emu[\M(A)^2]\leq e^{\gamma^2 c} \int_{A^2} \exp(\gamma^2 \log_+(|z-w|^{-1})) \d z \d w= e^{\gamma^2 c}
  \int_{A^2} \max(|z-w|^{-\gamma^2},1)  \d z \d w.\]

  Let $r=|A|^{\frac{1}{2}}\pi^{-\frac{1}{2}}$ be the radius of a disk of the same area as $A$, and
  \[S_r=\sup_{A': |A'|=\pi r^2} \int_{A'} \max(|z|^{-\gamma^2},1) \d z.
  \]
  The supremum is achieved by the ball $B(0,r)$. By considering separately the cases $r\leq 1$ and $r\geq 1$, we find that
  $S_r\leq
  \frac{2\pi}{2-\gamma^2}r^{2-\gamma^2} +\pi r^2$.

  By invariance under translations, $S_r$ is also equal to
  \[
  \sup_{w\in \R^2} \sup_{A': |A'|=\pi r^2} \int_{A'} \max(|z-w|^{-\gamma^2},1) \d z.
  \]

  Thus,
  \begin{align*} \Emu[\M(A)^2 ]&\leq   e^{\gamma^2 c} \int_{A^2} \max(|z-w|^{-\gamma^2},1) \d z\d w \\
  & \leq
  e^{\gamma^2 c} |A| S_r \leq
  e^{\gamma^2 c}\Big(\frac{2\pi^{\frac{\gamma^2}{2} } }{2-\gamma^2}
  |A|^{2-\frac{\gamma^2}{2} }+ |A|^2\Big).
  \end{align*}
  This concludes the proof.
  %
\end{proof}

%
%



The next estimate, should be understood as a quantitative version of the following idea. Let first $A_1$, $A_2$ and $B$ be three sets such that $A_1$ and $A_2$ are very close from each other, have comparable Lebesgue measures, and are far from $B$; then, $\Emu[\M(A_1)\M(B)]$ should be very close to $\Emu[\M(A_2)\M(B)]$.
Let now $A_1$, $A_2$, $B_1$, and $B_2$ be four sets such that
\begin{itemize}
\item $A_1$ and $A_2$ are very close from each other, and have comparable Lebesgue measures,
\item $B_1$ and $B_2$ are very close from each other, and have comparable Lebesgue measures,
\item The $A_i$ are far from the $B_i$.
\end{itemize}
Then, the expectation $\Emu[( \M(A_2)-\M(A_1))(\M(B_2)-\M(B_1))]$ should be even smaller than $\Emu[(\M(A_2)-\M(A_1))\M(B_i)]$. Remark that the control on the `four terms' expectation should depend on the distance between the $A_i$, the distance between the $B_i$, the distance between $A_i$ and $B_i$, the Lebesgue measure of these four sets, and finally the difference between their Lebesgue measure. Hence, we do not expect a very short bound to appear.

For two disjoint sets $A$ and $B$, let us define
\begin{align*}
d_{\sup}(A,B)&=\sup \{d(z,z'):z\in A,z'\in B \},\\
d_{\inf}(A,B)&=\min(\inf \{d(z,z'):z\in A,z'\in B \}, 1),
\end{align*}
as illustrated by Figure \ref{fig:quatre_ens} below.
\begin{figure}[h!]
\begin{center}
\includegraphics[scale=1]{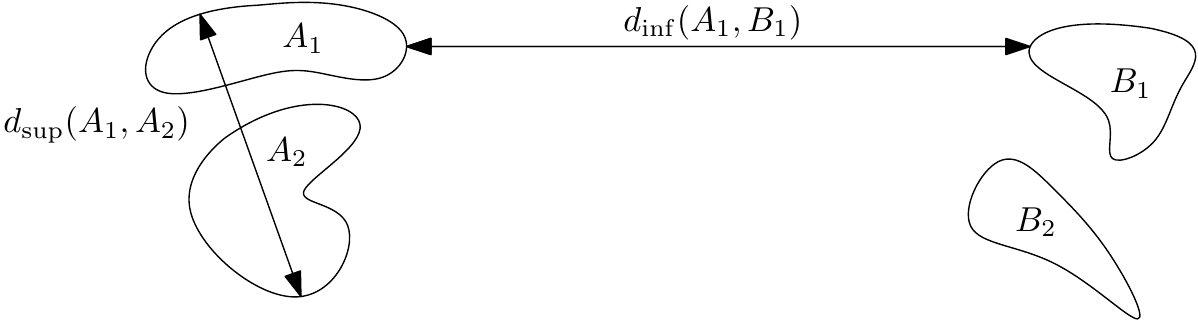}
\caption{\label{fig:quatre_ens}
A typical position for the sets in Lemma \ref{le:4points}, and some of the `distances' between them.
A typical choice for these sets would be $A_1=\mathcal{D}^i_N$, $A_2=\mathcal{D}^i_{-N}$, $B_1=\mathcal{D}^j_N$, and $B_2=\mathcal{D}^j_{-N}$.
}
\end{center}
\end{figure}
\begin{lemma}
  \label{le:4points}
  There exists $C$ such that for all 
  bounded Borel measurable sets $A_1,A_2,B_1,B_2$ with $4 \max(d_{\sup}(A_1,A_2),d_{\sup}(B_1,B_2) )\leq d_{\inf}(A_1,B_1)$ and $|A_1|,|A_2|,|B_1|,|B_2|\leq 1$,   
  \begin{multline}
  |\Emu[ (\M(A_1)-\M(A_2)) (\M(B_1)-\M(B_2)]|\leq C \Big( d_{\inf}(A_1,B_1)^{-\gamma^2}\big||B_1|-|B_2|\big| \big||A_1|-|A_2|\big|
  \label{eq:4points}\\
  +  \big(d_{\sup}(A_1,A_2)+d_{\sup}(B_1,B_2) \big) d_{\inf} (A_1,B_1)^{-1-\gamma^2}
\big(|A_1| \big||B_1|-|B_2|\big|+
  \big||A_2|-|A_1|\big| |B_1| \big)\\
+ d_{\inf}(A_1,B_1)^{-2-\gamma^2} d_{\sup}(A_1,A_2)d_{\sup}(B_1,B_2)|A_1||B_1|
  \Big).
  \end{multline}
\end{lemma}
Let us remark that the exponent $\gamma$, which is carried by $|A|$ when looking at $\Emu[\M(A)^2]$, is here carried instead by  $d_{\inf}(A_1,B_1)$. Remark also that all these terms but the last vanish if $|A_1|=|A_2|$ and $|B_1|=|B_2|$.
\begin{proof}
    It is a simple computation. Using the explicit formula given by Lemma \ref{le:moment2}, we have
    \begin{align*}
    \Emu[ (\M(A_1)-\M(A_2)) (\M(B_1)-\M(B_2)]&=
    \int_{A_1\times B_1} e^{\gamma K(z_1,w_1)} \d z_1 \d w_1
  -\int_{A_1\times B_2} e^{\gamma K(z_1,w_2)} \d z_1 \d w_2\\
  &\hspace{1cm}-\int_{A_2\times B_1} e^{\gamma K(z_2,w_1)} \d z_2 \d w_1
  +\int_{A_2\times B_2} e^{\gamma K(z_2,w_2)} \d z_2 \d w_2
  \end{align*}
  Writing $Q=A_1\times A_2\times B_1\times B_2$, the right-hand side of this equation is equal to
  \begin{multline*}
  \frac{1}{|Q|} \int_{Q} \Big( |A_1||B_1|(e^{\gamma K(z_1,w_1)}-e^{\gamma K(z_1,w_2)}-e^{\gamma K(z_2,w_1)}+e^{\gamma K(z_2,w_2)} ) \\
  +
  |A_1|(|B_2|-|B_1|)( e^{\gamma K(z_2,w_2)}-e^{\gamma K(z_1,w_2)})
  +
   (|A_2|-|A_1|)|B_1| ( e^{\gamma K(z_2,w_2)}-e^{\gamma K(z_2,w_1)})\\
   +
  (|A_2|-|A_1|) (|B_2|-|B_1|)e^{\gamma K(z_2,w_2)}\Big)\d z_1\d z_2\d w_1\d w_2.
  \end{multline*}
  This, in absolute value, is less than
  \begin{align*} \sup_{(z_1,z_2,w_1,w_2)\in Q} \Big( &|A_1||B_1| |e^{\gamma K(z_1,w_1)}-e^{\gamma K(z_1,w_2)}-e^{\gamma K(z_2,w_1)}+e^{\gamma K(z_2,w_2)}|\\
  &+
  |A_1|\big||B_2|-|B_1|\big|  |e^{\gamma K(z_2,w_2)}-e^{\gamma K(z_1,w_2)}|
  +
  \big||A_2|-|A_1|\big||B_1| |e^{\gamma K(z_2,w_2)}-e^{\gamma K(z_2,w_1)}|\\
  &\hspace{6cm}+\big||A_2|-|A_1|\big|\big||B_2|-|B_1|\big|  e^{\gamma K(z_2,w_2)}\Big).
  \end{align*}
  We set $D_A$ a ball of diameter $d_{\sup{}}(A_1,A_2)$ containing $A_1$ and $A_2$, and $D_B$ a ball of diameter $d_{\sup{}}(B_1,B_2)$ containing $B_1$ and $B_2$, so tat the last expression is less than
  \begin{multline*}
  |A_1||B_1| d_{\sup{}}(A_1,A_2)d_{\sup{}}(B_1,B_2)\sup_{z\in D_A,w\in D_B}\|\nabla^2e^{\gamma K(z,w)}\|\\+
  |A_1|\big||B_2|-|B_1|\big| d_{\sup{}}(A_1,A_2)\sup_{z\in D_A,w_2\in B_2} \|\nabla  e^{\gamma K(z,w_2)}\|\\
  +
  \big||A_2|-|A_1|\big||B_1| d_{\sup{}}(B_1,B_2)\sup_{z_2\in A_2,w\in D_B} \|\nabla e^{\gamma K(z_2,w)}\|\\
  +\big|\!|A_2|-|A_1|\!\big| \big|\!|B_2|-|B_1|\!\big|  \sup_{z_2\in A_2,w_2\in B_2} e^{\gamma K(z_2,w_2)}.
  \end{multline*}
  The three suprema appearing in this expression are respectively less than $Cd_{\inf{}}(D_A,D_B)^{-2-\gamma^2}$,  $Cd_{\inf{}}(D_A,B_2)^{-1-\gamma^2}$  and $Cd_{\inf{}}(A_2,B_2)^{-\gamma^2}$, where the constant $C$ depends on the exact expression of the kernel $K$ but not on the four sets $A_1$,$A_2$,$B_1$,$B_2$.

  Since $d_{\inf{}}(D_A,D_B)\geq d_{\inf{}}(A_1,B_1)-d_{\sup{}}(A_1,A_2)-d_{\sup{}}(B_1,B_2)\geq \frac{1}{2} d_{\inf{}}(A_1,B_1)$, we deduce the announced inequality.


\end{proof}

For reasons that will appear only at the end of the proof, we need to consider not only the random measure $\M$, but also measures obtained by possibly random (but independent from $\M$) translation of $\M$.

For $z\in\R^2$, we define  the random measure $\tau_z(\M)$ by setting, for all Borel measurable set $A$,
\[ \tau_z(\M)(A)=\M(A+z).\]
For a probability law $p$ on $\R^2$, we also set $\tau_p(\M)=\int_{\R^2} \tau_z(\M) \d p(z)$. Finally, we define \[\mathcal{T}(\M)=\{ \tau_p(\M): p \mbox{ is a probability law on }\R^2 \}.\]
%
%

For  a random measure $\tilde{\M}$, let us
define
\begin{align*}
\nor{\tilde{\M}}&=\inf\big\{C>0 : \forall A \text{ Borel set with } |A|\leq 1,\  \Emu[\tilde{\M}(A)^2]\leq C |A|^{\nu}\big\}^\frac{1}{2}\\
&=\inf\big\{C>0 : \text{ Lemma \ref{le:firstBound} holds}\}^\frac{1}{2}
\end{align*}
and, similarly\footnote{We chose this notation because it is related to the second derivatives of the kernel $K$.},
\[\nord{\tilde{\M}}=\inf\big\{C>0 : \text{ Lemma \ref{le:4points} holds}\}^\frac{1}{2}.
\]
It is easily shown that, for all
$\tilde\M \in \mathcal{T}(\M)$, $ \nor{\tilde\M}\leq\nor{\M}$ and  $\nord{\tilde\M}\leq\nord{\M}$.

In the following two sections, all the results stated with $\M$ also hold for any measure $\tilde\M\in \mathcal{T}(\M)$, with the same constants. Actually, they hold for all random measure $\tilde\M$ with $\nor{\tilde\M}<+\infty$ and  $\nord{\tilde\M}<+\infty$. Furthermore, they hold uniformly on $\mathcal{T}(\M)$ provided the constants $C$ are replaced with $C(\nor{\tilde\M}+\nord{\tilde\M})$.

\section{Comparing $\M(\mathcal{D}_N)$ and $\M_N$}
\label{sec:comparison}

Let us recall the notation that we introduced in Section \ref{sec:1}. We fix $t>0$, a small positive real. We also fix a (large) positive integer $N$, and define $T=\lfloor N^{t} \rfloor$. In the estimations that follow, $t$ is going to be fixed, and $N$ is going to tend to infinity.

For each $i\in \{1,\dots, T\}$, we define
\begin{itemize}
    \item $X^i$ to be the restriction of $X$ to $[(i-1)T^{-1},iT^{-1}]$,
    \item $\bar{X}^i$ the concatenation of $X^i$ with the segment between its endpoints,
    \item $\theta^i(z)$ the number of times $\bar{X}^i$ winds around $z$,
    \item $\mathcal{D}^i_N=\{ z\in \R^2\setminus\Range(\bar{X}^i): \theta^i(z)\geq N\}$.
\end{itemize}
We then define
\[\M_{N}=\sum_{i=1}^T  \M(\mathcal{D}^i_{N} ).\]
Let us stress the fact that the law of the random sets $\mathcal{D}^i_{N}$ can be deduced from that  of $\mathcal{D}_N$ by a scaling and a random translation. This self-similar behaviour will come into play in a crucial way in Section \ref{sec:boot}.

As explained in Section \ref{sec:1}, one of our main objects of interest is $\M(\mathcal{D}_N)$, the $\M$-measure of the set of points around which the Brownian motion winds at least $N$ times (see \eqref{eq:defDN}), and our strategy to study it is to compare it with $\M_N$.


The goal of this section is to prove that $\M(\mathcal{D}_N)$ and $\M_N$ are close, in the sense of Proposition \ref{le:main2}. This result will allow us, in Section \ref{sec:boot}, to transfer to $\M(\mathcal{D}_N)$ the information about $\M_N$ that we will gather in Section \ref{sec:estim}.
\medskip

To compare $\M(\mathcal{D}_N)$ and $\M_N$, we  introduce several sets. To start with, let us denote by $\mathcal E$ the union of the range of $\bar{X}$ and the $T$ segments joining the endpoints of $X^i$ for $i\in \{1,\ldots,T\}$. It is an important fact for us that this set $\mathcal E$ is Lebesgue-negligible, and this implies that it is $\Pmu$-almost surely $\M$-negligible. Then, given integers $N,M_1,M_2$, indices $i,j\in \{1,\dots, T\}$, and a multi-index $\mathbf{i}\in \mathbf{I}_k=\{(i_1,\dots, i_k)\in \{1,\dots, T\}^k: i_1<\dots <i_k\}$, we define
\begin{align*}
\mathcal{D}^{i,j}_{N,M_1}&=\{z\in \R^2\setminus \mathcal E: |\theta^i(z)|\geq N,|\theta^j(z)|\geq M_1\},\\
\mathcal{D}^{\mathbf{i}}_{M_2}&=\{z\in \R^2\setminus \mathcal E: \forall l\in\{1,\dots, k\},  |\theta^{i_l}(z)|\geq M_2\}.
\end{align*}
The integer $k$ that we will use will only depend on $\gamma$, and can be though of as being fixed.

The following lemma gives us a relation between the set $\mathcal{D}_N$ and the sets that we just defined.

\begin{lemma}
  \label{le:inclusions}
  Assume that $TM_2\leq \frac{N}{k}-T$ and $kM_1+(M_2+1)T<N$. Then,
  \[
  \bigcup_{i=1}^T \mathcal{D}^i_{N+TM_1 } \setminus \bigcup_{i\neq j}\mathcal{D}^{i,j}_{\frac{N}{k},M_1}\  \subseteq
  \ \mathcal{D}_N \setminus \mathcal E\ \subseteq \  \bigcup_{i=1}^T \mathcal{D}^i_{N-kM_1- (M_2+1)T } \cup \bigcup_{i\neq j}\mathcal{D}^{i,j}_{\frac{N}{k},M_1} \cup \bigcup_{\mathbf{i}\in  \mathbf{I}_k } \mathcal{D}^{\mathbf{i}}_{M_2}.
  \]
\end{lemma}
These inclusions are purely deterministic and the statement remains true if replace the Brownian curve with any other curve. The proof is based on a discussion of the highest values taken by the winding functions of the pieces of our curve. It would probably be best done by the reader for himself, but we offer a detailed argument for the second inclusion, which is the less simple one.


\begin{proof}

  For a given point $z\in\R^2\setminus \mathcal E$, let us sort the values $(|\theta^i(z)|)_{i\in \{1,\dots, T\}}$ in non-increasing order and denote them by $\eta_1\geq \dots \geq \eta_T$.
  Let us also denote $\tilde{\eta}_1$ one of the values  $(\theta^i(z))_{i\in \{1,\dots, T\}}$ such that $|\tilde{\eta}_1|=\eta_1$. We have the following implications:
\begin{alignat*}{2}
      &\textstyle  z\notin \bigcup_{i=1}^T \mathcal{D}^i_{N-kM_1- ({M_2}+1)T } &&\ \Longrightarrow\  \tilde{\eta}_1<N-kM_1- ({M_2}+1)T,\\
&\textstyle      z\notin \bigcup_{\mathbf{i}\in \mathbf{I}_k } \mathcal{D}^{\mathbf{i}}_{M_2} &&\ \Longrightarrow \ \eta_k<{M_2},\\
    &\textstyle  z\notin \bigcup_{i\neq j}\mathcal{D}^{i,j}_{\frac{N}{k},M_1} &&\ \Longrightarrow\  \big(\eta_1< \tfrac{N}{k} \mbox{ or }\eta_2< M_1\big).
\end{alignat*}

If $z$ is in none of the sets appearing on the left of these implications, we are in one of two cases, depending on which of the two assertions on the right of the third implication holds.

If $\eta_1< \frac{N}{k}$, then $\eta_2,\ldots, \eta_{k-1}<\frac{N}{k}$ and
  \[ \sum_{i=1}^T \theta^i(z)\leq \sum_{i=1}^T \eta_i < (k-1)\frac{N}{k}+ (T-(k-1)){M_2}\leq N-T.\]
If $\eta_2< M_1$, then $\eta_3,\dots, \eta_{k-1}<M_1$ and
  \[ \sum_{i=1}^T \theta^i(z)\leq \tilde{\eta}_1+ \sum_{i=2}^T \eta_i <(N-kM_1- ({M_2}+1)T)+ (k-2)M_1+ (T-(k-1)){M_2}\leq N-T.\]

In both cases, we conclude that
  $\sum_{i=1}^T \theta^i(z)\leq N-T$. The difference between this sum and $\theta_X(z)$ is the winding at $z$ of a piecewise linear curve with $T+1$ pieces, which cannot exceed $T-1$. Thus $\theta_X(z)< N$, so that $z\notin\mathcal{D}_N$.
 \end{proof}

In order to compare $\M(\mathcal D_N)$ with $\M_N$, we are going to take the $\M$-measures of the sets of which we just proved the inclusion. The $\M$-measures of the first unions appearing in the leftmost and rightmost terms of Lemma \ref{le:inclusions} will be close, but not exactly equal, to $\M_N$.

A first difference is that we are taking the measure of a union instead of the sum of the measures. This problem turns out not to be a serious one, and will be treated in the proof of Proposition \ref{le:main2}.

A second difference is that instead of $\M_N$, there seems to appear $\M_{N'}$ for two integers $N'$ close to $N$. To go around this difficulty, we will in fact apply Lemma \ref{le:inclusions} to several well-chosen values of $N$, and use Lemma \ref{le:boundMomentP} to connect the various estimations that we obtain in this way.

A third difference is that there are correction terms appearing on both sides, and which we need to control: this will be done by the following Lemma \ref{le:LAWA:2} and Corollary \ref{coro:joint}.




The first estimation in the next statement is a mild reformulation of the Lemma 2.4 that we obtained in \cite{LAWA}. The second one is a slight improvement of the Lemma 2.5 in the same paper, which corresponds to the case $k=3$. The extension from $k=3$ to general $k$ is obtain by following the same proof. It is a long, but elementary, computation that involves a decomposition of $(\R^2)^{k-1}$ into a family of products of balls and complementary of balls in $\R^2$.

\begin{lemma}
\label{le:LAWA:2}
  For all positive integer $k$ and all $r\in(0,+\infty)$, there exists $C$ such that for all positive integers $N,M_1,M_2$ and $T$, the following holds.
\begin{itemize}
\item
  For all $i,j \in \{1,\dots, T\}, i\neq j$,
  \[
  \EX[ |\mathcal{D}^{i,j}_{N,M_1}|^r ]\leq C
  \log(TNM_1+1)^{3r+1}
  \frac{(TNM_1)^{-r}}{|j-i|+1} .\hspace{1.5cm}
  \]
\item
  For all $\mathbf{i}\in  \mathbf{I}_k$,
  \[
  \EX[ |\mathcal{D}^{\mathbf{i}}_{M_2}|^r ]\leq C
  \log(M_2T+1)^{(k+1)(r+1)-2} \frac{T^{-r} M_2^{-kr}  }{ \prod_{j=1}^{k-1} (i_{j+1}-i_j+1)}.
   \]
   \end{itemize}
\end{lemma}
From these estimations, Lemma \ref{le:firstBound} allows us to deduce corresponding estimations in the Liouville case. We then sum over $i\neq j$ or over $\mathbf{i}$, and we obtain the following result.

\begin{corollary}
  \label{coro:joint}
  For all positive integer $k$, for all $p\in[1,+\infty)$, there exists a constant $C$ such that for all positive integers $N,M_1,M_2,T$,
  \begin{align*}
  \sum_{i\neq j}  \| \M(\mathcal{D}^{i,j}_{N,M_1})\|_{p,2} &\leq C
  \log(TNM_1+1)^{\frac{3\nu}{2}+\frac{1}{p}} T^{2-\frac{1}{p}-\frac{\nu}{2} } N^{-\frac{\nu}{2} }M_1^{-\frac{\nu}{2}}, \\
   \sum_{\mathbf{i}\in  \mathbf{I}_k}  \|\M(\mathcal{D}^{\mathbf{i}}_{M_2})\|_{p,2}&\leq C
      \log(M_2T+1)^{(k+1)(\frac{\nu}{2}+\frac{1}{p})-2} T^{k-\frac{k}{p}-\frac{\nu}{2} }M_2^{-k \frac{\nu}{2} }.
  \end{align*}
\end{corollary}
\begin{proof}
  We prove the first inequality, the second proof is identical.
  For $i\neq j$, by Lemma \ref{le:firstBound},
  \[
   \Emu [\M(\mathcal{D}^{i,j}_{N,M_1})^2]^{\frac{1}{2}}
   \leq C
  (|\mathcal{D}^{i,j}_{N,M_1}|^{\frac{\nu}{2}}+|\mathcal{D}^{i,j}_{N,M_1}|).\]
  Applying Lemma \ref{le:LAWA:2} with $r=\frac{p\nu}{2}$ and with $r=p$, we obtain
  \begin{align*}
  \EX[ \Emu [\M(\mathcal{D}^{i,j}_{N,M_1})^2]^{\frac{p}{2}}]^{\frac{1}{p}}
  &  \leq C'
  \EX[|\mathcal{D}^{i,j}_{N,M_1}|^{\frac{p\nu}{2}}+|\mathcal{D}^{i,j}_{N,M_1}|^p]^\frac{1}{p}\\
  &\leq C''
  \log(TNM_1+1)^{\frac{3\nu}{2}+\frac{1}{p}}
  \frac{(TNM_1)^{-\frac{\nu}{2} } }{(|j-i|+1)^{\frac{1}{p}}}.
  \end{align*}
  We then sum over $i\neq j$ to get the announced bound.
\end{proof}

The next lemma compares the measures of $\mathcal D_N$ and $\mathcal D_{N'}$ when $N$ and $N'$ are close.

\begin{lemma}
  \label{le:boundMomentP}
  Let $M=M(N)$ be an integer-valued function of $N$ such that $M\to \infty$ and $\frac{M}{N}\to 0$ as $N$ tends to infinity.
  Then, for all $r\in(0,+\infty)$,
  \[
  \EX\big[ \big||\mathcal{D}_N|-|\mathcal{D}_{N+M}|\big|^r \big]^{\frac{1}{r}}
  \leq O(MN^{-2})+M^{\frac{2}{r}}N^{-\frac{3}{2}-\frac{1}{r}+o(1)}.
  \]

  If additionally $\liminf \frac{M^q}{N}>0$ for some $q<2$, then for all $r\in(0,+\infty)$, there exists $C$ such that for all $N\geq 1$,
  \[
  \EX\big[ \big||\mathcal{D}_N|-|\mathcal{D}_{N+M}|\big|^r \big]^{\frac{1}{r}}
  \leq C MN^{-2}.
  \]
\end{lemma}
\begin{proof}
  We use the following convergence, which is the main result of \cite{werner}.
  \begin{equation}
  \label{eq:werner}
  \EX\big[(n^2|\{ z\in \R^2: \theta(z)= n\}| -\tfrac{1}{2\pi})^2\big]\underset{n\to \infty}\longrightarrow 0.
  \end{equation}
  In particular, there exists  $C$ such that for all $n\geq 1$,
  $\EX[|\{ z\in \R^2: \theta(z)= n\}|^2]^{\frac{1}{2}}\leq C n^{-2}$.
  Summing from $n=N$ to $N+M-1$, we deduce that
  \begin{equation}
  \label{eq:casp2}
  \EX[ (|\mathcal{D}_{N}|- |\mathcal{D}_{N+M}|)^2]^{\frac{1}{2}}\leq C MN^{-2}.
  \end{equation}
  This is sufficient to conclude in the case $r=2$.
  The case $r<2$ follows from H\"older inequality.
  For $r>2$, we fix some $p>r$.
  From the triangle inequality,
  \[ \EX\big[\big||\mathcal{D}_{N}|- |\mathcal{D}_{N+M}|\big|^{p} \big]^{\frac{1}{p}}
  \leq \EX\big[\big||\mathcal{D}_{N}|- \tfrac{1}{2\pi N}\big|^{p} \big]^{\frac{1}{p}}
  +\EX\big[\big||\mathcal{D}_{N+M}|- \tfrac{1}{2\pi (N+M)}\big|^{p} \big]^{\frac{1}{p}}
 \tfrac{M}{2\pi N(N+M)}\]
  By Lemma \ref{le:estimate:LAWA}, there is thus $C$ such that for all $N$ and $M$,
  \begin{equation}
  \label{eq:interp}
  \EX[(|\mathcal{D}_{N}|- |\mathcal{D}_{N+M}|)^{p} ]^{\frac{1}{p}}\leq
  C (N^{-\frac{3}{2}+o(1)} +MN^{-2}).
  \end{equation}
  We now interpolate between the inequalities \eqref{eq:casp2} and \eqref{eq:interp}. Setting $\theta=\frac{r^{-1}-p^{-1}}{2^{-1}-p^{-1}}$, the H\"older inequality is written
  \[\|\ {\raisebox{0.05cm}{$\scriptscriptstyle \bullet$} }\  \|_{L^{r}}\leq \|\ {\raisebox{0.05cm}{$\scriptscriptstyle \bullet$}}\ \|_{L^2}^{\theta}\|\ {\raisebox{0.05cm}{$\scriptscriptstyle \bullet$}}\ \|_{L^{p}}^{1-\theta}.
  \]
  Hence,
  \[
  \EX[(|\mathcal{D}_{N}|- |\mathcal{D}_{N+M}|)^{r} ]^{\frac{1}{r}}
  \leq C
  \big( MN^{-2}+  \big(MN^{-2}\big)^{\theta} (N^{-\frac{3}{2}-o(1)})^{1-\theta}\big).\]
  As $p$ goes to infinity, $\theta$ goes to $\frac{2}{r}$, and we end up with
  \[
  \EX[(|\mathcal{D}_{N}|- |\mathcal{D}_{N+M}|)^{r} ]^{\frac{1}{r}}
  \leq C
  \big( MN^{-2}+ M^{\frac{2}{r} }N^{-\frac{3}{2}-\frac{1}{r}+o(1) }\big).\]
  This is the first announced bound.

  For the last bound, it suffices to remark that, for any $r>2$, for $\epsilon$ small enough, $M^{\frac{2}{r}}N^{-\frac{3}{2}-\frac{1}{r}+\epsilon}$ is negligible compared to
  $MN^{-2}$.
\end{proof}
\begin{remark}
  In \cite{werner}, it is stated that the convergence \eqref{eq:werner} can be extended to higher moments, and the proof is sketched. We are convinced that this sketch can indeed, to the price of a lot of effort, be turned into a proof, but to the best of our knowledge, this has not been done.
  If this statement is true, as we think it is, the proof of Lemma \ref{le:boundMomentP} becomes almost trivial, and the additional assumption becomes superflous.
\end{remark}
\begin{corollary}
\label{coro:boundMomentP}
  Let $M=M(N)$ be an integer-valued function of $N$ such that $M^q N^{-1} \to \infty$ for some $q<2$, and $MN^{-1} \to 0$ as $N$ tends to infinity.
  Then, for all $p\in[1,+\infty)$, there exists $C$ such that for all $N\geq 1$,
  \[ \|\M(\mathcal{D}_N)-\M(\mathcal{D}_{N+M})\|_{p,2}
  \leq  C N^{-\nu} M^{\frac{\nu}{2} }.
  \]
\end{corollary}
\begin{proof}
  Remark that $\mathcal{D}_{N+M}\subseteq \mathcal{D}_{N}$.
  We have
  \begin{align*}
  \|\M(\mathcal{D}_N)-\M(\mathcal{D}_{N+M})\|_{p,2}&=
  \EX[ \Emu[(\M(\mathcal{D}_N\setminus\mathcal{D}_{N+M}))^2 ]^{\frac{p}{2}}]^{\frac{1}{p}}\\
  &\leq C \EX[  (|\mathcal{D}_N\setminus \mathcal{D}_{N+M}|^{\nu}+|\mathcal{D}_N\setminus \mathcal{D}_{N+M}|^2 )^{\frac{p}{2}} ]^{\frac{1}{p}} \ \ \mbox{(using Lemma \ref{le:firstBound})}.
  \end{align*}
  We apply Lemma \ref{le:boundMomentP} with $r=\frac{\nu p}{2}$ and with $r=p$, and we obtain
  \[
  \EX[ \Emu[(\M(\mathcal{D}_N)-\M(\mathcal{D}_{N+M}))^2 ]^{\frac{p}{2}}]^{\frac{1}{p}}
  \leq
  C_2 \big( (MN^{-2})^{\frac{\nu}{2}}+MN^{-2}\big)\leq 2C_2M^{\frac{\nu}{2}  }N^{-\nu},
\]
which we square to get the result.
  \end{proof}

We are now finally ready for the comparison between $\M(\mathcal{D}_N)$ and $\M_N$.

\begin{proposition}
  \label{le:main2}
  Let $p\in(1,+\infty)$. 
  Then, for all $t<\min(\tfrac{\nu}{8}\tfrac{p}{p-1},\tfrac{1}{2})$ there exists $C$ such that for all $N\geq 1$,
  \[
  \|\M(\mathcal{D}_N)-\M_N\|_{p,2}
  \leq CN^{ -\frac{3\nu}{4}+2t}.
   \]
\end{proposition}
\begin{proof}
  Let us set  $M_1=\lfloor N^{\frac{1}{2}} \rfloor$ and ${M_2}=\lfloor N^{\frac{1}{4}} \rfloor$. The exponents $\tfrac{1}{2}$ and $\tfrac{1}{4}$ here are chosen in order to optimize some bound later on, and we suggest the reader should think of $M_1$ and $M_2$ as `some powers of $N$, satisfying $1\ll M_2\ll M_1$ and $M_1T \ll N$' during the proof. Actually, the value $\frac{1}{4}$
  can be replaced with any value strictly between $0$ and $\frac{1}{2}$, which even allows to extend the result of the lemma to $t<\min(\tfrac{\nu}{4}\tfrac{p}{p-1},\tfrac{1}{2})$. This, however, is useless for us.

  For any fixed integer $k$, the relations $T{M_2}\leq \frac{N}{k}-T$ and $kM_1+({M_2}+1)T<N$ holds as soon as $N$ is large enough.
  We can then apply the Lemma \ref{le:inclusions}. Since $\mathcal{D}^i_{N+TM_1 }\cap \mathcal{D}^j_{N+TM_1}\subseteq \mathcal{D}^{i,j}_{\frac{N}{k},M_1}$, we deduce the following inequalities, which holds $\mathbb{P}$-almost surely.
  \begin{equation}
  \sum_{i=1}^T \M(\mathcal{D}^i_{N+TM_1}) -\sum_{i\neq j} \M(\mathcal{D}^{i,j}_{\frac{N}{k},M_1})
  \leq \M(\mathcal{D}_N)
  \leq
  \sum_{i=1}^T \M( \mathcal{D}^i_{N-2TM_1}) +\sum_{i\neq j} \M(\mathcal{D}^{i,j}_{\frac{N}{k},M_1})+ \sum_{\mathbf{i}\in \mathbf{I}_k} \M(\mathcal{D}^{\mathbf{i}}_{{M_2}}).
  \label{eq:inclus1}
  \end{equation}
  Remark that we have replaced the sets $\mathcal{D}^i_{N-kM_1-(M_2+1)T}$ that appears in Lemma \ref{le:inclusions} with the larger set $\mathcal{D}^i_{N-2TM_1}$. This is possible because $kM_1+(M_2+1)T\leq 2TM_1$.

  If we try to compare directly the first and last expressions of \eqref{eq:inclus1}, we are lead to compare $ \sum_{i=1}^T \M( \mathcal{D}^i_{N-2TM_1})$ with $ \sum_{i=1}^T \M( \mathcal{D}^i_{N+TM_1})$, which is not very convenient. To circumvent the difficulty,
  we apply \eqref{eq:inclus1} with $N$ replaced by $\tilde{N}=N-3TM_1$, and also with $N$ replaced by $\tilde{N}= N+3TM_1$.

  We then obtain the following inequalities, $\mathbb{P}$-almost surely.
  \begin{multline}
  \M(\mathcal{D}_{N+3TM_1})- 2 \sum_{i\neq j} \M(\mathcal{D}^{i,j}_{\frac{N}{k},M_1})-\sum_{\mathbf{i} } \M(\mathcal{D}^{\mathbf{i}}_{{M_2}})\leq
  \sum_{i=1}^T \M( \mathcal{D}^i_{N+TM_1}) -\sum_{i\neq j} \M(\mathcal{D}^{i,j}_{\frac{N}{k},M_1}) \\
  \leq
  \M(\mathcal{D}_N)
  \leq
  \sum_{i=1}^T \M( \mathcal{D}^i_{N-2TM_1}) +\sum_{i\neq j} \M(\mathcal{D}^{i,j}_{\frac{N}{k},M_1})+ \sum_{\mathbf{i}} \M(\mathcal{D}^{\mathbf{i}}_{{M_2}})\\
  \leq
  \M(\mathcal{D}_{N-3TM_1})+ 2 \sum_{i\neq j} \M(\mathcal{D}^{i,j}_{\frac{N-TM_1}{k},M_1})+ \sum_{\mathbf{i}} \M(\mathcal{D}^{\mathbf{i}}_{{M_2}}).
  \label{eq:inclus2}
  \end{multline}
  The exact same inequalities also hold if the middle term $\M(\mathcal{D}_N)$ is replaced by $\M_N=\sum_{i=1}^T \M( \mathcal{D}^i_{N})$.
  It follows that the difference between $\M(\mathcal{D}_N)$ and  $\M_N$ is less than the difference between the left-most and right-most terms of  \eqref{eq:inclus2}:
  \begin{equation}
  \label{eq:boundmui:1}
  \big| \M(\mathcal{D}_N)-\M_N \big|\leq \big(\M(\mathcal{D}_{N-3TM_1})- \M(\mathcal{D}_{N+3TM_1}) \big)+ 4 \sum_{i\neq j} \M(\mathcal{D}^{i,j}_{\frac{N-TM_1}{k},M_1})+ 2 \sum_{\mathbf{i}} \M(\mathcal{D}^{\mathbf{i}}_{{M_2}}).
  \end{equation}
  The three terms on the right-hand side are  the ones that appears in Corollary \ref{coro:joint} and Lemma \ref{coro:boundMomentP} (applied with $M=TM_1$).
  Applying these lemmas, we obtain
  \begin{align*}
  \| \M(\mathcal{D}_N)-\M_N \|_{p,2}& \leq
  C \big( (TM_1)^{\frac{\nu}{2}  }N^{-\nu}+
  \log(TNM_1+1)^{\frac{3\nu}{2}+\frac{1}{p}} T^{2-\frac{1}{p}-\frac{\nu}{2} } N^{-\frac{\nu}{2} }M_1^{-\frac{\nu}{2}}\\ &\hspace{3cm}+
  \log(M_2T+1)^{(k+1)(\frac{\nu}{2}+\frac{1}{p})-2} T^{k-\frac{k}{p}-\frac{\nu}{2} }M_2^{-k \frac{\nu}{2} }\big)\\
  &\leq C' \big(  N^{-\frac{3\nu}{4}+t \frac{\nu}{2}} +
  \log(N+1)^{\frac{3\nu}{2}+\frac{1}{p}} N^{-\frac{3\nu}{4}+t(2-\frac{1}{p}-\frac{\nu}{2})}\\&\hspace{3cm}+
  \log(N+1)^{(k+1)(\frac{\nu}{2}+\frac{1}{p})-2} N^{-t\frac{\nu}{2}+k(t\frac{p-1}{p}-\frac{\nu}{8} )}\big).
  \end{align*}
  The bound on $t$ is such that the exponent factorized by $k$ is strictly negative. It follows that for $k$ sufficiently large, the last term is negligible compared to the others. We end up with, for some $C,C',d$, for all $N\geq 1$,
\[
 \| \M(\mathcal{D}_N)-\M_N \|_{p,2}\leq
  C \log(N+1)^d N^{-\frac{3\nu}{4}+t \max(\frac{\nu}{2},2-\frac{1}{p}-\frac{\nu}{2} )}\leq C' N^{-\frac{3\nu}{4}+2t}.
  \]
  This concludes the proof.
\end{proof}
\begin{remark}
Let us recall that our ultimate goal is to show that $\| \M(\mathcal{D}_N)-\M(\mathcal{D}_{-N})\|_{p,2}$ is asymptotically less than $N^{-1-\epsilon}$, for some $\epsilon>0$. The condition $\gamma<\sqrt{4/3}$ that we imposed in Theorem \ref{th:1} is equivalent to the condition $\frac{3\nu}{4}>1$, for which we just showed that
$ \| \M(\mathcal{D}_N)-\M_N\|_{p,2}$  is asymptotically less than $N^{-1-\epsilon}$, provided $t$ and $\epsilon$ are small enough. If we were to try to improve our proof to larger $\gamma$, the more important step would be to improve the bound $\frac{3\nu}{4}$ in Proposition \ref{le:main2}.
\end{remark}

\section{Estimating $\M_N$}
\label{sec:estim}

We now take the next step towards the proof of Theorem \ref{th:1}, according to the strategy presented in Section \ref{sec:1} and recalled at the beginning of the previous section. We use the notation introduced in Section \ref{sec:comparison}, as well as that of \eqref{eq:defDiN}.

Our next goal is thus to estimate $\M_N$, and more precisely
\[(\M_N-\M_{-N})^2= \sum_{i,j=1}^T ( \M(\mathcal{D}^i_N)- \M(\mathcal{D}^i_{-N}))( \M(\mathcal{D}^j_N)- \M(\mathcal{D}^j_{-N})).\]
This task will, in fact, only be completed in Section \ref{sec:boot}, because of some `bad' couples $(i,j)$ in this sum. Basically, we expect the expression $( \M(\mathcal{D}^i_N)- \M(\mathcal{D}^i_{-N}))( \M(\mathcal{D}^j_N)- \M(\mathcal{D}^j_{-N}))$ to have a $\Pmu$-expectation very close to $0$, provided that $X^i$ and $X^j$ are far from each other: these are the good couples $(i,j)$. When these Brownian pieces go close to each other (bad couples), it is difficult to have a good bound on the $\Pmu$-expectation.  To prove that these bad couples do not contribute too much to $\M_N$ is a substantial problem on its own, and we will address it in the next section. In the present section, we bound the sum over the good couples only.\footnote{The proofs presented in this section can be substantially simplified in the case where $p=2$, using the symmetry properties of the Brownian motion. The corresponding arguments are less robust but much simpler than the ones that we give below, and we will include them in a forthcoming version of this paper.}

Let us introduce a notation which will make many expressions much shorter than they would otherwise be: we set, for all $i\in \{1,\ldots,T\}$,
\begin{equation}\label{eq:defRi}
R_i=\M(\mathcal{D}^i_N)- \M(\mathcal{D}^i_{-N}).
\end{equation}

We fix some $\epsilon>0$ and set $\beta=\tfrac{1}{2}-\tfrac{\epsilon}{3}$.
For $j>i+1$, we define the following events of large probability  in $\Omega^X$:
\[
E=\{ \|X\|_{\mathcal{C}^\beta}\leq \tfrac{1}{4}T^{\frac{\epsilon}{3}} \},\ \
F_{i,j}=\{ \| X_{(i+1)T^{-1}}-X_{jT^{-1}}\|\geq T^{-\frac{1}{2}+\epsilon } \}.
\]
For $j\in\{i,i+1\}$, we simply set $F_{i,j}=\varnothing $ in order to harmonize some results.
The complement of an event $G$ in $\Omega^X$ is denoted by $G^c$. We first get rid of the event $E^c$. 

\begin{lemma}
  \label{le:eventE}
  For all $p\in[1,+\infty)$, and all $r>0$, there exists $C$ such that for all $N\geq 1$,
  \[ \|\mathbbm{1}_{E^c} (\M_N-\M_{-N})\|_{p,2}  \leq C N^{-r}.\]
\end{lemma}
\begin{proof}

  From the triangle inequality in $L^p(\Omega^X,L^2(\Omega^\M))$ and Cauchy--Schwarz inequality between $\mathbbm{1}_{E^c}$ and
  $\Emu[R_i^2 ]^{\frac{p}{2}}$
  in $\Omega^X$,
  \begin{align*}
\|\mathbbm{1}_{E^c} (\M_N-\M_{-N})\|_{p,2} \leq \sum_{i=1}^T\|\mathbbm{1}_{E^c} R_i\|_{p,2}&=
  \sum_{i=1}^T
  \EX [ \Emu[\mathbbm{1}_{E^c} R_i^2]^{\frac{p}{2}}]^{\frac{1}{p}}\leq  \sum_{i=1}^T\PX(E^c)^{\frac{1}{2p}}\EX[ \Emu[R_i^2 ]^p]^{\frac{1}{2p}}.
  \end{align*}
  By Lemma \ref{le:firstBound}, for some $C$, for all $N$,
  \[\Emu[\M(\mathcal{D}^i_N )^2]\leq C  (|\mathcal{D}^i_N|^{\nu}+|\mathcal{D}^i_N|^{2}).\]
  By Lemma \ref{le:estimate:LAWA} and a scaling argument, for some $C'$, for all $N$,
  \[ \EX[ |\mathcal{D}^i_N|^{\nu p}+|\mathcal{D}^i_N|^{2p}]\leq C' (TN)^{-\nu p} \leq C'.\]
  The same bounds hold for $N$ replaced with $-N$, and we deduce that for some $C^{(2)}$, for all $N$,
  \[
  \EX [ \mathbbm{1}_{E^c} \Emu[ (\M_N-\M_{-N})^2]^{\frac{p}{2}}]^{\frac{1}{p}}\leq C^{(2)} T
\PX(E^c)^\frac{1}{2p}.\]

  From Kolmogorov continuity theorem, for all $\beta<\frac{1}{2}$, $\|X\|_{\mathcal{C}^\beta}$ admits moments of all orders. From Markov's inequality, the tail probability $\PX(E^c)$ decreases more quickly than any polynomial in $T$ (hence in $N$). This concludes the proof.
\end{proof}

\begin{lemma}
  \label{le:outdiagonal}
  For all $p\in\big[2,\frac{4}{\gamma^2})$, for all $\epsilon>0$ and $t>0$,
  \[
  \EX \big[ \big|\Emu\big[
   \sum_{i,j=1}^T \mathbbm{1}_{E\cap F_{i,j} } R_iR_j \big]\big|^{\frac{p}{2}} \big]^{\frac{1}{p}}
  \leq C  T^{\frac{\epsilon}{2}} N^{-1+o(1)} \big(
  N^{-\frac{1}{4}}+T^{ \frac{\gamma^2}{4}-\frac{1}{p} }
  \big).\]
\end{lemma}

\begin{proof}
  Let $i,j\in \{1,\dots, T\}$ be such that $j>i+1$.
  During this proof, in order to try and maintain the length of expressions within reasonable bounds,
  we will write
  \begin{align*}
  D^{i j}\hspace{0.3cm}&= |\mathcal{D}^i_N|+|\mathcal{D}^i_{-N}|+|\mathcal{D}^j_N|+|\mathcal{D}^j_{-N}|,\\
  \triangle D^{i j}&= \big||\mathcal{D}^i_N|-|\mathcal{D}^i_{-N}|\big|+\big||\mathcal{D}^j_N|-|\mathcal{D}^j_{-N}|\big|,\\
  X_{i,j}\hspace{0.3cm}&=  | X_{(i+1)T^{-1}}-X_{jT^{-1}} |.
  \end{align*}
  On $F_{i,j}\cap E$, both
  $4d_{\sup}(\mathcal{D}^i_N, \mathcal{D}^i_{-N})$ and $4 d_{\sup}(\mathcal{D}^j_N, \mathcal{D}^j_{-N})$ are less than  $d_{\inf} (\mathcal{D}^i_N, \mathcal{D}^j_{N})$, so that we can apply Lemma \ref{le:4points}. Since
  $d_{\inf} (\mathcal{D}^i_N, \mathcal{D}^j_{N})\leq X_{i,j}$, we obtain, for some constant $C$,
  \[
  \big|\Emu[R_iR_j]\big|
  \leq C \big( X_{i,j}^{-\gamma^2} (\triangle  D^{i j})^{2}
  + T^{-\beta}\|X\|_{\mathcal{C}^\beta} X_{i,j}^{-1-\gamma^2}
  D^{i j} \triangle D^{i j} + T^{-2 \beta} \|X\|_{\mathcal{C}^\beta}^{2}
  X_{i,j}^{-2-\gamma^2}(D^{i j})^2\big).
  \]
  We raise to the power $\frac{p}{2}$, multiply by $\mathbbm{1}_{E\cap F_{i,j}}$, and take the $\PX$-expectation. We obtain
  \begin{align}
  \EX &\Big[
  \mathbbm{1}_{E\cap F_{i,j}}
  \big|\Emu[
  R_iR_j]\big|^{\frac{p}{2}} \Big]\leq
  C' \big( \EX\big[\mathbbm{1}_{F_{i,j}} X_{i,j}^{-\frac{p}{2}\gamma^2} (\triangle D^{i j})^{p} \big]   \label{eq:sec4,1}\\&
  + T^{-\frac{p\beta}{2}} \EX\big[\mathbbm{1}_{F_{i,j}} \|X\|_{\mathcal{C}^\beta}^{\frac{p}{2}} X_{i,j}^{-\frac{p}{2}(1+\gamma^2)}
  (D^{i j})^{\frac{p}{2}} (\triangle  D^{i j})^{\frac{p}{2}}\big] + T^{-p\beta} \EX\big[ \mathbbm{1}_{F_{i,j}} \|X\|_{\mathcal{C}^\beta}^{p}
  X_{i,j}^{-\frac{p}{2}(2+\gamma^2) } (D^{i j})^p \big]\big). \nonumber
  %
  \end{align}
  The variables $D^{i j}$ and $\triangle  D^{i j}$ are measurable with respect to the $\sigma$-algebra generated by $(X_s)_{s\leq (i+1)T^{-1}}$ and $(X_s-X_{jT^{-1}})_{s\geq jT^{-1}}$, 
  and hence jointly independent from $X_{i,j}$.  
  Thus,
  \begin{align}
  \EX[\mathbbm{1}_{F_{i,j}} X_{i,j}^{-\frac{p}{2}\gamma^2} \hspace{1.88cm}(\triangle  D^{i j})^{p}\hspace{0.07cm}]
  &=\EX[\mathbbm{1}_{F_{i,j}} X_{i,j}^{-\frac{p}{2}\gamma^2}]\hspace{0.747cm}
  \EX[ (\triangle D^{i j})^{p}],\nonumber \\
  \EX[\mathbbm{1}_{F_{i,j}}X_{i,j}^{-\frac{p}{2}(2+\gamma^2) }(D^{i j})^p\hspace{1.66cm}]
  &= \EX[\mathbbm{1}_{F_{i,j}}
  X_{i,j}^{-\frac{p}{2}(2+\gamma^2 )}   ]\
  \EX[(D^{i j})^p], \label{eq:sec4,2}\\
  \EX [\mathbbm{1}_{F_{i,j}}   X_{i,j}^{-\frac{p}{2}(1+\gamma^2) } (D^{i j})^{\frac{p}{2}} (\triangle D^{i j})^{\frac{p}{2}}]
  &= \EX[ \mathbbm{1}_{F_{i,j}} X_{i,j}^{-\frac{p}{2}(1+\gamma^2) } ]\
  \EX[ (D^{i j})^{\frac{p}{2}} (\triangle  D^{i j})^{\frac{p}{2}} ]. \nonumber
  \end{align}

  With  Lemma \ref{le:estimate:LAWA} and the scaling properties of the Brownian motion, we obtain the following bounds. For all $r\in [1,+\infty)$, there exists $C$ such that for all $N$,
  \begin{equation}
  \EX[(D^{i j})^{r}]^{\frac{1}{r}}\leq C T^{-1}N^{-1},\ \
  \EX[ (\triangle D^{i j})^{r} ]^{\frac{1}{r}}\leq T^{-1} N^{-\frac{3}{2}+o(1) }. \label{eq:sec4,3}
  \end{equation}
  We also need to control the expectations that depends on $X_{i,j}$. This variable is distributed according to $2\pi p_{(j-i-1)T^{-1}}(u) u \d u$, where $\d u$ denotes the Lebesgue measure on $\R^+$ and
  $p_t(u)=(2\pi t)^{-1} \exp(-\frac{u^2}{2t})$.

  With elementary computations, we obtain, for any $r>0$,
  \begin{align}
  \EX[\mathbbm{1}_{F_{i,j}} X_{i,j}^{-r} ]&=2\pi
  \int_{T^{-\frac{1}{2}+\epsilon }}^{+\infty} u^{1-r} p_{(j-i-1)T^{-1}}(u)\d u\nonumber\\
  &=
  (T^{\frac{1}{2}}(j-i-1)^{-\frac{1}{2}})^{r}
  \int_{(j-i-1)^{-\frac{1}{2}}T^{\epsilon }}^{+\infty} \rho^{1-r} p_1(\rho) \d \rho\nonumber\\
  &\leq\left\{
  \begin{array}{ll}
  C \phantom{\log(T+1)} \hspace{0.05cm}
  T^{\frac{r}{2}}(j-i-1)^{-\frac{r}{2}} & \mbox{if } 1-r>-1 \\
  C \log(T+1)
  T^{\frac{r}{2}}(j-i-1)^{-\frac{r}{2}} &\mbox{if } 1-r=-1 \\
  C\phantom{\log(T+1)} \hspace{0.05cm}
  T^{\frac{r}{2}}(j-i-1)^{-1}
  &
  \mbox{if } 1-r<-1
  \end{array}
  \right. \nonumber\\
  &\leq C \log(T+1) T^{\frac{r}{2}}(j-i-1)^{-\min(\frac{r}{2},1)}.\label{eq:sec4,4}
  \end{align}
  %
  Combining \eqref{eq:sec4,1}, \eqref{eq:sec4,2}, \eqref{eq:sec4,3} and \eqref{eq:sec4,4} leads to
  \begin{align*}
  \EX &\big[
  \mathbbm{1}_{E\cap F_{i,j}}
  \big|\Emu[R_iR_j]\big|^{\frac{p}{2}} \big] \leq C\log(T+1)\big(
   T^{\frac{p\gamma^2}{4}} (j-i-1)^{-\min(\frac{p\gamma^2}{4},1)}T^{-p}N^{-\frac{3}{2}p+o(1)}\\
  &\hspace{2.5cm}+T^{-\frac{p}{4}+\frac{\epsilon p}{3}} T^{\frac{p}{4}(1+\gamma^2)}(j-i-1)^{-\min(\frac{p}{4}(1+\gamma^2),1)} T^{-p}N^{-\frac{5p}{4}+o(1)}\\
  &\hspace{2.5cm}+ T^{-\frac{2 p}{4}+\frac{\epsilon p}{2} } T^{\frac{p}{4}(2+\gamma^2)}(j-i-1)^{-\min(\frac{p}{4}(2+\gamma^2),1)}T^{-p}N^{-p}
  \big)\\
  &\leq C'  \log(T+1)T^{\frac{\epsilon p}{2}} N^{-p}T^{p(\frac{\gamma^2}{4}-1)} \big(
   (j-i-1)^{-\frac{p\gamma^2}{4} } N^{-\frac{p}{4}+o(1)}+ (j-i-1)^{-1} \big).
  \end{align*}

  From the triangle inequality in $L^{\frac{p}{2}}(\Omega^X)$,
  \begin{align*}
  \EX \big[ \big|\Emu\big[
   \sum_{i,j=1}^T \mathbbm{1}_{E\cap F_{i,j} } R_iR_j \big]\big|^{\frac{p}{2}} \big]^{\frac{1}{p}}
  &\leq\Big( \sum_{i,j=1}^T \EX\big[ \big|\Emu\big[\mathbbm{1}_{E\cap F_{i,j} }  R_iR_j  \big]\big|^{\frac{p}{2}} \big]^{\frac{2}{p}}\ \Big)^{\frac{1}{2}}\\
  &\leq
  \Big( \sum_{i,j=1}^T \big(
  C  \log(T+1)T^{\frac{\epsilon p}{2}} N^{-p}T^{p( \frac{\gamma^2}{4}-1 ) }\\
  & \hspace{-4.5cm}\big(
(j-i-1)^{-\min(\frac{p\gamma^2}{4},1)} N^{-\frac{p}{2}+o(1)}+ (j-i-1)^{-\min(\frac{p}{4}(1+\gamma^2),1)} N^{-\frac{p}{4}+o(1)}+  (j-i-1)^{-1}
  \big)\big)^{\frac{2}{p}}
  \Big)^{\frac{1}{2}}\\
  &\leq
  C' \log(T+1)^\frac{1}{p} T^{\frac{\epsilon}{2}} N^{-1} T^{\frac{\gamma^2}{4}} \big(
  N^{-\frac{1}{4}+o(1)} T^{-\frac{\gamma^2}{4} }  + T^{ -\frac{1}{p} }
  \big).
  \end{align*}
  Since $\log(T+1)\leq N^{o(1)}$, This concludes the proof.
\end{proof}

\section{Bootstrapping the bounds}
\label{sec:boot}
Let us summarize what we did up to here. Our goal is to bound $\|\M(\mathcal{D}_N)-\M(\mathcal{D}_{-N})\|_{p,2}$. In Section \ref{sec:comparison}, we showed that, in this estimation, $\M(\mathcal{D}_N)$ can safely be replaced with $\M_N$, provided $t$ is small and $\gamma<\sqrt{4/3}$. Then, we split $(\M_N-\M_{-N})^2$
into a sum of two terms. Let us call them $\M_{\rm good}$ and $\M_{\rm bad}$. The first one is a sum over `good terms', and we have been able in Section \ref{sec:estim} to show that this term is small, provided $p<\frac{4}{\gamma^2}$.

What remains to be done is thus to control $\M_{\rm bad}$. Here is the place where the bootstrap really starts. In a perfect world, our dearest wish would be to have $\| \M_{\rm bad}\|_{p,2}<CN^{-1-h}$, for some $C$ and $h>0$. This however is not what we will obtain in the first place. We will first give a bound on $\M_{\rm bad}$ which is something like
\begin{equation}
\label{eq:mubad}
\| \M_{\rm bad}\|_{p,2}<CN^\xi \|\M(\mathcal{D}_N)-\M(\mathcal{D}_{-N})\|_{p,2}.
\end{equation}
It might seem at first that this approach is doomed to fail, because it seems that the problem is now to control $\|\M(\mathcal{D}_N)-\M(\mathcal{D}_{-N})\|_{p,2}$, which actually was the problem we started with. The crucial point is to obtain a \emph{negative} exponent $\xi$: indeed, in that case, the relation
\[
\|\M(\mathcal{D}_N)-\M(\mathcal{D}_{-N})\|_{p,2}\leq CN^{-1-\epsilon}+ C
N^\zeta \|\M(\mathcal{D}_N)-\M(\mathcal{D}_{-N})\|_{p,2}\]
does imply
\[
\|\M(\mathcal{D}_N)-\M(\mathcal{D}_{-N})\|_{p,2}\leq C'N^{-1-\epsilon},\]
for a new constant $C'$.

We will show a relation of the form \eqref{eq:mubad}, with $\xi=t(\frac{\gamma^2}{4}-\frac{1}{p})$. Let us recall that $\mathcal{T}(\M)$ is the set of random measures obtained from $\M$ by translation by a random variable independent from $\M$, and possibly  a symmetry with respect to the horizontal axis.

Let us also recall that the parameter $\epsilon>0$ appears in the definition of the event $F_{i,j}$.

\begin{lemma}
  \label{le:diagBoundBootstrap}
  Set $t>0$, $p\in[2,\frac{4}{\gamma^2})$, and $\zeta\in \R$. Assume that
  there exists $C$ such that for all $N\geq 1$,
  \[
  \|\M'(\mathcal{D}_N )-\M'(\mathcal{D}_{-N} )\|_{p,2}\leq C N^\zeta .
  \]
  Then, for all $\epsilon>0$, there exists a constant $C'$ such that for all $N\geq 1$,
  \[\EX\Big[\Emu\big[\hspace{-0.2cm} \sum_{1\leq i\leq j\leq T } \hspace{-0.2cm} \mathbbm{1}_{F_{i,j}^c}  |R_iR_j| \big]^{\frac{p}{2}}\Big]^\frac{1}{p}  \leq C' \log(T+1)^{\frac{1}{p}}  T^{\frac{\gamma^2}{4}-\frac{1}{p}+\epsilon }N^\zeta.
  \]
\end{lemma}
\begin{proof}
Before we dive into the proof, let us look at the behaviour of $\nor{\M}$ under scaling of $\M$.
For all $z\in \R^2$ and $\lambda\geq 1$, let $\M_{z,\lambda}$ be the measure defined by setting
\[\M_{z,\lambda}(A)=\M( \lambda^{-\frac{1}{2}} A+z).
\]
In particular, $\M_{z,1}=\tau_z(\M)$ with the definition of Section \ref{sec:section2}
 For $A$ with $|A|\leq 1$, we have $|\lambda^{-\frac{1}{2}} A+z|\leq 1$, so that
  \begin{align*}
  \mathbb{E}[\M_{z,\lambda}(A)^2]^\frac{1}{2}
  &\leq \nor{\M} |\lambda^{-\frac{1}{2}} A+z|^{\frac{\nu}{2}}\\
  &= \nor{\M} \lambda^{-\frac{\nu}{2}} |A|^{\frac{\nu}{2}}.
  \end{align*}
  Thus, $\nor{\M_{z,\lambda} }\leq \lambda^{-\frac{\nu}{2}} \nor{\M}$. Similarly, we have
  $
  \nord{\M_{z,\lambda} }
  \leq \lambda^{-\frac{\nu}{2}} \nord{\M}$.
Besides, the assumption of the lemma extends automatically to all $\M'\in \mathcal{T}(M)$: for all $M'\in \mathcal{T}(M)$, for all $N\geq 1$,
  \[
  \|\M'(\mathcal{D}_N )-\M'(\mathcal{D}_{-N} )\|_{p,2}\leq C N^\zeta .
  \]
  That being said, let us start the proof.
  Let us fix two indices $i, j\in \{1,\dots, T\}$ with $i\leq j$.  We bound $|R_iR_j|$ by
$\frac{1}{2}(R_i^2+R_j^2)$,
  and we will treat separately the term with $i$ from the term with $j$. The reason why we do not apply the same treatment to these two terms is that the event $F^{c}_{i,j}$ is independent of $R_i$ but \emph{not} of $R_j$.

  For $i\in\{1,\dots, T\}$, let us denote by $J_i$ the random variable
  \[J_i=  \sum_{j\geq i} \mathbbm{1}_{F_{i,j}^c},\]
  which is independent of $R_i$.
  Then, using the triangle inequality in  $L^{\frac{p}{2}}(\Omega^X)$ (recall that $p\geq 2$), we have
  \begin{align*}
  \EX  \Big[ \Emu\big[  \sum_{j\geq i}\mathbbm{1}_{F_{i,j}^c} R_i^2 \big]^{\frac{p}{2}} \Big]^{\frac{1}{p}}
  =  \EX  \Big[ \Emu\big[  \sum_{i=1}^T J_i R_i^2 \big]^{\frac{p}{2}} \Big]^{\frac{1}{p}}
  &\leq \Big( \sum_{i=1}^T \EX \big[ \Emu \big[   J_i R_i^2\big]^{\frac{p}{2}}\big]^{\frac{2}{p}} \Big)^{\frac{1}{2}}\\
  &\leq  \Big( \sum_{i=1}^T \EX[J_i^{\frac{p}{2}}]^{\frac{2}{p}}
  \|R_i\|_{p,2}^2 \Big)^{\frac{1}{2}}.
  \end{align*}
  For $j>i+1$,

  \begin{align}
  \PX(F_{i,j}^c)&=
   \PX(\| X_{(i+1)T^{-1}}-X_{jT^{-1}}\|\leq T^{-\frac{1}{2}+\epsilon }  )
   = \int_{0}^{T^{-\frac{1}{2}+\epsilon } }  \frac{\exp(- \frac{r^2}{2 (j-i-1)T^{-1}} )}{2\pi (j-i-1)T^{-1}} r\d r\nonumber \\
  &  \leq  \frac{T^{2\epsilon}}{2 (j-i-1)}.\label{eq:boundProbaFij}
  \end{align}
  It follows that
  \[ \EX [J_i^{\frac{p}{2}}]\leq T^{\frac{p}{2}-1}\sum_{j=i}^T \PX(F_{i,j}^c)\leq C T^{\frac{p}{2}-1+2\epsilon} \log(T+1).\]

  Let $\tilde{\mathcal{D}}_N$ be a random set which is equal in distribution to $\mathcal{D}_N$ under $\PX$, but which is independent from $X$. Then, the random set $\mathcal{D}^i_N$ is equal in distribution to $T^{-\frac{1}{2}}\tilde{\mathcal{D}}_N+X_i$. It follows that
  \begin{align*}
  \|R_i\|_{p,2}=
  \|\M(\mathcal{D}^i_N )-\M(\mathcal{D}^i_{-N} )\|_{p,2}
  &\leq \sup_{z\in \R^2} \|\M(T^{-\frac{1}{2}}\tilde{\mathcal{D}}_N+z)-\M(T^{-\frac{1}{2}} \tilde{\mathcal{D}}_N+z)\|_{p,2}\\
  &  =\sup_{z\in \R^2} \|\M_{z,T}(\mathcal{D}_N)-\M_{z,T}(\mathcal{D}_{-N})\|_{p,2}\\
  &
  \leq C (\nor{\M_{z,T}}+ \nord{\M_{z,T}} ) N^{\zeta}\\
  &\leq CT^{-\frac{\nu}{2}} (\nor{\M}+ \nord{\M} ) N^{\zeta}.
  \end{align*}

  Putting all together,
  \begin{align*}
  \EX  \Big[ \Emu\big[\hspace{-0.2cm}\sum_{1\leq i\leq j \leq T} \hspace{-0.2cm}\mathbbm{1}_{F_{i,j}^c} R_i^2 \big]^{\frac{p}{2}} \Big]^{\frac{1}{p}}
  &\leq C T^{\frac{-\nu}{2}} N^{\zeta} \big( \sum_{i=1}^T C' \log(T+1)^{\frac{2}{p}} T^{1-\frac{2}{p}-\frac{4\epsilon}{p}} \big)^{\frac{1}{2}}\\
  &\leq C \log(T+1)^{\frac{1}{p}}  T^{1-\frac{1}{p}-\frac{\nu}{2}+\frac{2\epsilon}{p} }  N^{\zeta}.
  \end{align*}

  Having dealt with the terms $i\leq j$,  we now have to deal with the ones for which $j<i$. For this, we replace the Brownian motion $X$ with the time-reversed Brownian motion $\tilde{X}:t\mapsto X_{1-t}-X_1$, which is independent from $(X_1,\M)$.
  We remark that $\M_{X_1,1}(\tilde{\mathcal{D}}_N^{i})={\M}(\mathcal{D}_{-N}^{T-i-1})$, where $\tilde{\mathcal{D}}_N^{i}$ is the set defined as $\mathcal{D}_N^{i}$, but with $\tilde{X}$ replacing $X$. We are now back in the situation `$i\leq j$' since $i'=T-i-1<T-j-1=j'$.  We can then use the same bounds as in the case $i\leq j$, and we end up with the same bound, which concludes the proof.
\end{proof}

We are now ready to prove Proposition \ref{le:1}. We first recall it.
\begin{proposition*}[\ref{le:1}]
Assume that $\gamma\leq \sqrt{4/3}$ and $p\in[2,\frac{4}{\gamma^2})$. Then, there exists $\delta>0$ and $C$ such that for all $N\geq 1$,
\[ \|\M(\mathcal{D}_N) - \M(\mathcal{D}_{-N} )\|_{p,2} \leq C N^{-1-\delta}.\]
\end{proposition*}

\begin{proof}
Set $t  \in\big(0, \frac{-3\nu}{8}\big)$, so that, by Proposition \ref{le:main2}, there exists $\delta_1>0, C_1$ such that for all $N\geq 1$,
\[ \|\M(\mathcal{D}_N)-\M_N\|_{p,2}\leq C_1 N^{-1-\delta_1}.\]

Set also $\epsilon\in \big(0, \min \big(\frac{1}{p} -\frac{\gamma^2}{4}, \frac{1}{2}\big)\big)$, and $\delta_2\in \big(0, t(\frac{1}{p} -\frac{\gamma^2}{4} -\epsilon\big)$,
so that by Lemma \ref{le:outdiagonal}, there exists $C_2$ such that for all $N\geq 1$,
\[ \big\| \Emu\big[\sum_{i,j} \mathbbm{1}_{E\cap F_{i,j}} R_iR_j \big]^{\frac{1}{2}}  \big\|_{p} \leq C_2 N^{-1-\delta_2}.\]


We inductively show that for all $k\in \mathbb{N}$, there exists a constant $C$ such that
\[
\|\M(\mathcal{D}_N)-\M(\mathcal{D}_{-N})\|_{p,2}\leq C ( N^{-1-\min(\delta_1,\delta_2)}+N^{-\frac{\nu}{2}-k \delta_2 } ).
\]
At rank $k=0$, it follows directly from Lemma \ref{le:firstBound}.

The induction hypothesis, together with Lemma \ref{le:diagBoundBootstrap}, ensures that there exists $C$ such that for all $N\geq 1$,
\[
\big\| \Emu\big[\sum_{i,j} \mathbbm{1}_{E\cap F_{i,j}^c} R_iR_j \big]^{\frac{1}{2}} \big\|_{p}\leq C N^{\max (-1-\delta_1, -1-\delta_2, \xi-k\delta_2)}.
\]
To go from rank $k$ to rank $k+1$, let us decompose $\|\M(\mathcal{D}_N)-\M(\mathcal{D}_{-N})\|_{p,2}$ as follows:
\begin{align*}
\|\M'(\mathcal{D}_N)-\M'(\mathcal{D}_{-N})\|_{p,2}
&\leq \|\M(\mathcal{D}_N)-\M_N\|_{p,2}+
\|\M(\mathcal{D}_{-N})-\M_{-N}\|_{p,2}
+ \| \M_{N}-\M_{-N}\|_{p,2}\\
&\leq 2C_1 N^{-1-\delta_1}+ \big\| \Emu\big[\sum_{i,j} \mathbbm{1}_{E^c} R_iR_j \big]^{\frac{1}{2}}  \big\|_{p}+ \big\| \Emu\big[\sum_{i,j} \mathbbm{1}_{E\cap F_{i,j}} R_iR_j \big]^{\frac{1}{2}}  \big\|_{p}\\& \hspace{6cm}
+ \big\| \Emu\big[\sum_{i,j} \mathbbm{1}_{E\cap F_{i,j}^c} R_iR_j \big]^{\frac{1}{2}}  \big\|_{p}\\
&\leq 2C_1 N^{-1-\delta_1}+CN^{-r}+ C_2 N^{-1-\delta_2} +C N^{\max (-1-\delta_1, -1-\delta_2, \xi-k\delta_2)},
\end{align*}
which implies the induction hypothesis at rank $k+1$.
This concludes the induction. For $k$ large enough, we obtain
\[
\|\M(\mathcal{D}_N)-\M(\mathcal{D}_{-N})\|_{p,2}\leq C N^{-1-\min(\delta_1,\delta_2)},
\]
which concludes the proof of the proposition, hence also of Theorem \ref{th:1}.
\end{proof}

The work that we have done so far allowed us to define the `algebraic Liouville area enclosed by the Brownian curve' as the $\P$-almost sure limit
\[
\mathbb{A}_{0,1}= \lim_{K\to +\infty}\int_{\R^2}\max( -K,\min(\theta_{X_{|[0,1]}},K ))\d \M .\]

Our reasoning extends without trouble if we replace $(0,1)$ with any couple $(s,t)\in \Delta=\{(s,t)\in [0,1]^2: s\leq t \}$: for all such couple, $\mathbb{P}$-almost surely, the
limit
\begin{equation}
\label{eq:def:AcalB}
\mathbb{A}_{s,t}= \lim_{K\to +\infty}\int_{\R^2}\max( -K,\min(\theta_{X_{|[s,t]}},K ))\d \M
\end{equation}
exists. The fact that the trajectory now starts from a random point does not necessitate any additional work.

In the next section, we will define a similar quantity when the Brownian motion is replaced with a deterministic smoother curve. A classical analogue, when the measure $\M$ is replaced with the Lebesgue measure, would be to say that after defining the Lévy area, we are now constructing the Young integral.

The reason with why do this before pushing further the analysis of the Brownian case is twofold. First, there are some results that we will apply to both the Brownian and the smoother case. In order to understand these results fully when we state them, it is better to have both definitions in mind. The second reason is that our proof that Chen's relation holds in the Brownian case is rather involved, and we think that it is more easily understood once we have proved the Chen relation for smoother curves.


\section{The case of smoother curves}
\label{sec:smoother}
The results that we present now are mostly independent from the first part of the paper.
In this section, we replace the planar Brownian motion $X$ with a function $Y:[0,1]\to \R^2$ which is $\alpha$-H\"older continuous. We will say that we work under the \emph{relaxed assumptions} if we only assume that $\gamma<2$ and $\alpha>\frac{1}{2}$, and under the \emph{strengthened assumptions} if we assume $\gamma<\sqrt{2}$ and $\alpha> \frac{1}{2(1-\frac{\gamma^2}{4}) }$.

For a real number $q\in[0,\frac{4}{\gamma^2})$, we define
\[
\xi(q)=q(1+\frac{\gamma^2}{4})-q^2\frac{\gamma^2}{4},
\]
the so-called \emph{structure exponent} of $\M$.

Under the relaxed assumptions, we will show that for all $(s,t)\in \Delta=
\{(s,t)\in [0,1]^2: s\leq t \}$, the $\M$-area $\mathbb{A}^Y_{s,t}$ delimited by $Y_{|[s,t]}$ is almost surely defined as
\begin{equation}
\label{eq:def:Acal}
\mathbb{A}^Y_{s,t}= \int_{\R^2} \theta_{Y_{|[s,t]}} \d \M,
\end{equation}
  and lies in $L^1(\Omega^{\M})$.

Under the strengthened assumptions, we show that it actually lies in $L^2(\Omega^{\M})$, and that the map $\mathbb{A}^Y:\Delta\to L^2(\Omega^{\M})$ admits some H\"older regularity. \footnote{The strengthened assumptions are necessary to work in the $L^2$ framework. Nonetheless, it should be possible to extend some of the results of the next sections to the relaxed assumptions, provided that we succeed to work in $L^{1+\epsilon}(\Omega^{\M})$. 
}

To simplify notation, we will write $Y_{s,t}=Y_{|[s,t]}$. Let us recall that $\lambda$ denotes the Lebesgue measure on the plane.

The following result gives us information about the winding function of a H\"older continuous curve.

  \begin{lemma}
  \label{le:six1}
    If $Y$ is $\alpha$-H\"older continuous with $\alpha>\frac{1}{2}$, then $\theta_Y$ is defined Lebesgue-almost everywhere and lies in $L^{r}(\R^2,\lambda)$ for any $r\in[1,2\alpha)$.

  For any $r\in[1,2\alpha)$, there exists a constant $C$ which depends on $\alpha$ and $r$ but not on $Y$ and such that for all $s<t\in[0,1]$,
  \[ \| \theta_{Y_{s,t}}\|_{L^r(\R^2,\lambda)}\leq C (t-s)^{\frac{2\alpha}{r}} \|Y\|_{\mathcal{C}^\alpha}^{\frac{2}{r}}.\]
\end{lemma}
The first part of this lemma was proved in Theorem 0.2 of our previous work \cite{LAWA} (but with H\"older norm instead of $p$-variation norm\footnote{Recall that $\|Y_{s,t}\|_{p-{\rm var}}\leq (t-s)^\alpha \|Y\|_{\mathcal{C}^\alpha}$. }).
The second part is Lemma 7.8 in the same paper.

\begin{corollary}
  \label{coro:six2}
  Under the relaxed assumptions, for all $(s,t)\in\Delta$, almost surely, the function $\theta_{Y_{s,t}}$ lies in $L^1(\R^2,\M)$ and the random variable $\int_{\R^2} |\theta_{Y_{s,t}}|\d \M$ lies in $L^1(\Omega^\M, \mathbb{P}^{\M})$.

  %
\end{corollary}
\begin{proof}
The first point follows from the previous lemma. For the second, it suffices to remark that
  \[
  \Emu\bigg[ \int_{\R^2} |\theta_{Y}| \d \M\bigg]=\int_{\R^2} |\theta_{Y}| \d \lambda <+\infty.
  \]
\end{proof}
We are now interested in higher moments of $\int_{\R^2} \theta_{Y} \d \lambda$.
\begin{lemma}
  \label{le:six3}
  Let $r>(1-\tfrac{\gamma^2}{4})^{-1}$ and $f\in L^{r}(\R^2,\lambda)$ with support in the unit ball.

  Then, $\Pmu$-almost surely, $f\in L^{r}(\R^2,\M)$, and for any $\rho\in [1,r(1-\tfrac{\gamma^2}{4}) )$, the random variable $ \|f\|_{L^\rho(\R^2, \M)}$ lies in $L^{2\rho}(\Omega^\M)$.

  Besides, there exists a constant $C$, which depends only on $r$, $\rho$ and $K$, and such that for all $f \in L^{r}(\R^2,\M)$,
  \[
  \big\| \|f\|_{L^\rho(\R^2,\M) }  \big\|_{L^{2\rho}(\Omega^\M)}
  \leq C \|f\|_{L^{r}(\R^2,\lambda)}.
  \]
\end{lemma}
\begin{proof}
  The fact that $f$ lies $\Pmu$-almost surely in $ L^{r}(\R^2,\M)$ follows directly from
  \[  \mathbb{E}\bigg[ \int_{\R^2} |f|^r \d \M \bigg]=\int_{\R^2} |f|^r\d \lambda <+\infty.\]

  For the remaining part of the lemma, let $C$ be such that $K(z,w)\leq C+ \log(|z-w|^{-1})$ for all $(z,w)\in(\R^2)^2$. We have
  \begin{align*}
  \mathbb{E}\bigg[ \Big( \int |f|^\rho \d \M\Big)^{2} \bigg]&=e^{C\gamma^2}\int_{\R^2\times \R^2} \frac{|f|^\rho(z) |f|^\rho(w)}{|z-w|^{\gamma^2} }  \d z\d w\\
  &=e^{C\gamma^2} \int_{B(0,1)} \frac{|f|^\rho(z) }{|z-w|^{\frac{\gamma^2}{2} } } \bigg( \int_{B(0,1)} \frac{|f|^\rho(w) }{|z-w|^{\frac{\gamma^2}{2} } } \d w\bigg) \d z.
  \end{align*}

  By H\"older inequality applied with $p=\frac{r}{\rho} >\frac{1}{1-\frac{\gamma^2}{4} }$ and $q=\frac{r}{r-\rho}< \frac{4}{\gamma^2}$,
  \[
  \int_{B(0,1)} \frac{|f|^\rho(w) }{|z-w|^{\frac{\gamma^2}{2} } } \d w\leq \|f\|^\rho_{L^{r}(\R^2,\lambda)} \bigg( \int_{B(0,2)} \frac{1}{|w|^{q \frac{\gamma^2}{2}} }\d w\bigg)^{\frac{1}{q}}.
  \]
  The latter integral is finite, and with another identical computation we end up with
  \[
  \int_{\R^2} \frac{|f|^\rho(z) }{|z-w|^{\frac{\gamma^2}{2} } } \bigg( \int_{\R^2} \frac{|f|^\rho(w) }{|z-w|^{\frac{\gamma^2}{2} } } \d w\bigg) \d z\leq C' \|f\|_{L^r(\R^2,\lambda)}^{2\rho},
  \]
  hence
  \[ \mathbb{E}\bigg[ \Big( \int |f|^\rho \d \M\Big)^{2}\bigg]^{\frac{1}{2\rho}}\leq C'' \|f\|_{L^r(\R^2,\lambda)},\]
  which concludes the proof.
\end{proof}

Recall that $\nu=2(1-\frac{\gamma^2}{4})$.

\begin{corollary}
  \label{coro:six4}
  Let $Y:[0,1]\to \R^2$ be an $\alpha$-H\"older continuous function  for some $\alpha>\frac{1}{\nu }$.
  Then, for all $(s,t)\in \Delta$, $\mathbb{A}^Y_{s,t}$ lies in $ L^2(\Omega^\M, \Pmu)$.
  Besides, for all $\epsilon>0$, there exists $C$ such that, for all $s<t$,
  \[\| \mathbb{A}^Y_{s,t}\|_{L^2(\Omega^\M, \Pmu)}\leq C (t-s)^{ \alpha \nu-\epsilon}.
\]
  %
\end{corollary}
\begin{proof}
  Lemma \ref{le:six1} ensures that for all $r\in \big(\frac{2}{\nu}, 2\alpha \big)$,
  there exists $C$ such that
  for all $s<t\in[0,1]$, $\theta_{Y_{s,t}}\in L^{r}(\R^2, \lambda)$ and
  $\|\theta_{Y_{s,t}}\|_{L^r(\R^2,\lambda)}\leq C (t-s)^{\frac{2\alpha}{r}}$.

  Lemma \ref{le:six3} applied with $\rho=1< \frac{\nu r}{2}$
  then ensures that $\Pmu$-almost surely, $\theta_{Y_{s,t}}\in L^{1}(\R^2,\M)$,
  and that
  \[\| \mathbb{A}^Y_{s,t}\|_{L^2(\Omega^\M, \Pmu)}\leq C' (t-s)^{\frac{2\alpha}{r}}.
  \]
  We conclude by taking $r$ arbitrarily close to $\frac{2}{\nu}$.
%
%
\end{proof}

\begin{remark}
The bound $r<2\alpha$ in Lemma \ref{le:six1} is optimal. In Lemma \ref{le:six3}, the bound $\rho<\frac{\nu r}{2}$ is optimal in the sense that there exists $f\in L^r(\mathbb{R}^2,\lambda)$ such that $\| \|f\|_{L^{\rho}(\R^2,\M) } \|_{L^{2\rho}(\Omega^\M)}=+\infty $ for $\rho>\frac{\nu r}{2}$
{\rm (}roughly speaking, $L^r(\mathbb{R}^2,\lambda)$ is not included in the `fibered space' $L^{2\rho}(\Omega^\M, L^\rho(\R^2,\M) )${\rm )}. We do not know
about the cases $r=2\alpha$ and $\rho=\frac{\nu r}{2}$.

The optimality in Lemma \ref{le:six1} is seen by considering the curve
$Y$ which goes once along each of the circles with center $(0,n^{-\alpha'})$ and radius $n^{-\alpha'}$ {\rm (}see Figure \ref{fig:circle} below{\rm )}. This curve has finite $p$-variation for all $p>\frac{1}{\alpha'}$, hence it can be parameterized as an $\alpha$-H\"older continuous function for all $\alpha<\alpha'$. Nonetheless, there exists $C$ such that for all $N>0$, $|\{z: \theta_Y(z)=N\}|=CN^{-2\alpha'-1}$. It follows that
the winding function $\theta_Y$ does not belong to $L^{2\alpha'}(\R^2,\lambda)$.
\begin{figure}[h!]
\begin{center}
\mbox{
\includegraphics[scale=0.7]{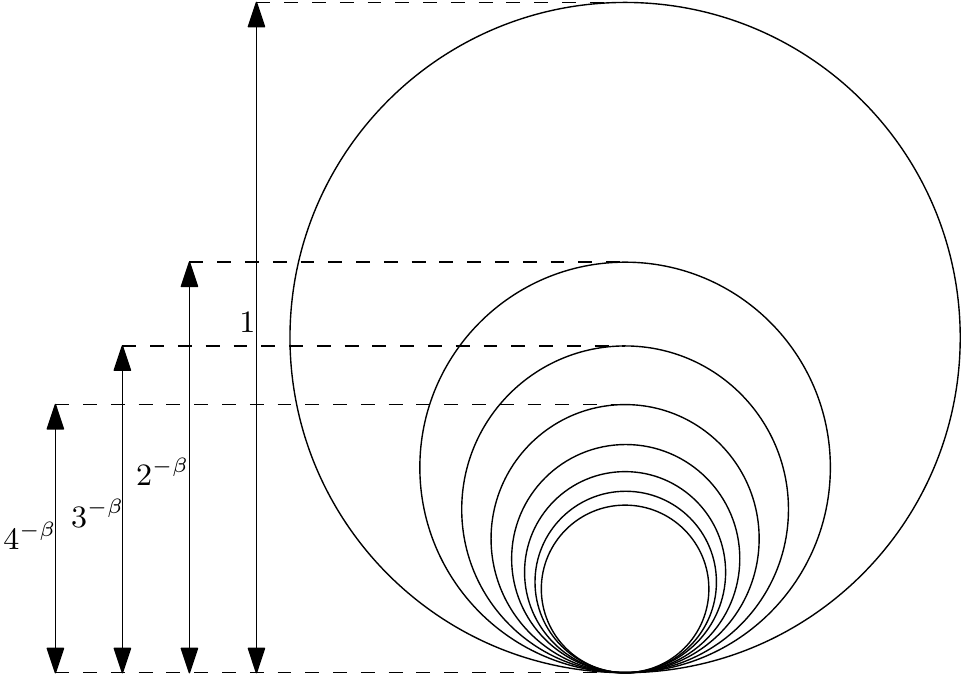}
}
\caption{\label{fig:circle}
The function $Y$, here for $\beta= 2/3$. Inspired from Figure 5, p.11 in \cite{2DM}.
}
\end{center}
\end{figure}

For the optimality in Lemma \ref{le:six3}, one could look precisely at the function $\theta_Y$, but it is slightly simpler to look at the function $f:z\mapsto |z|^{-\alpha}\mathbbm{1}_{B(0,1)}(z)$ with $\alpha<\frac{2}{r}$. This function is easily seen to lie in $L^{r}(\R^2, \lambda)$, but for any $\rho\geq \frac{\nu}{\alpha}$,
\begin{align*}
\Emu\bigg[ \Big(\int_{\R^2} f(z)^\rho \d z\Big)^{2}\bigg]
&=\int_{B(0,1)^2} |z|^{-\alpha\rho}|w|^{-\alpha\rho}|z-w|^{-\gamma^2} \d z\d w\\
&\geq \int_{B(0,1)} \int_{B(0,|w|)} |w|^{-\alpha\rho}|w|^{-\alpha\rho}2^{-\gamma^2} |w|^{-\gamma^2} \d z\d w\\
&=2^{-\gamma^2} \pi \int_{0}^1 r^{-2\alpha\rho -\gamma^2+3} \d r=+\infty.
\end{align*}
\end{remark}

\section{Weak Chen's relation}

We will need the following version of the Chen relation, where the quantifier on $s$, $u$ and $t$ and the \emph{almost sure} has been exchanged. Let us recall from the introduction that $T_{s,u,t}$ is the triangle delimited by $Z_s$, $Z_u$, and $Z_t$ (the continuous function $Z:[0,1]\to\R^2$ is assumed to be fixed), and that $\epsilon_{s,u,t}\in \{\pm 1\}$ depends one the cyclic order between these points on the boundary of $T_{s,u,t}$.
\begin{definition}
Set $\Delta_3=\{(s,t,u): 0\leq s\leq u\leq t\leq 1\}$. We say that a collection $(\mathbb{A}_{s,t})_{(s,t)\in \Delta}$ of random variables satisfies the \emph{weak Chen relation} {\rm (}relative to the continuous function $Z${\rm )} if for all $(s,u,t)\in \Delta_3$, $\mathbb{P}$-almost surely,
\[ \mathbb{A}_{s,t}=\mathbb{A}_{s,u}+\mathbb{A}_{u,t}+\epsilon_{s,u,t}\M(T_{s,u,t}).\]
\end{definition}
Let us remark that, in contrast with the Chen relation, the weak Chen relation is preserved by the replacement of the family $\mathbb{A}$ by one of its modifications (see Definition \ref{def:modif}).
\begin{lemma}
\label{le:chenFaibleY}
Under the relaxed assumptions, that is, if $\alpha>\frac{1}{2}$ and $\gamma<2$, the family of random variables $\mathbb{A}^Y$ defined by \eqref{eq:def:Acal} satisfies the weak Chen relation relative to $Y$.
\end{lemma}
\begin{proof}
  We give two proofs. The first is much simpler, but understanding the second one will help us to understand the proof of the similar result for the Brownian motion.

  \emph{First proof.}
    Let us fix $(s,u,t)\in\Delta_3$.
    The equality
    \[\theta_{s,t}=\theta_{s,u}+\theta_{u,t}+\epsilon_{s,u,t}\mathbbm{1}_{T_{s,u,t}}\] holds $\lambda$-almost everywhere. Hence, $\Pmu$-almost surely, this equality holds $\M$-almost everywhere. Since all of these functions are $\Pmu$-almost surely $\M$-integrable (Corollary \ref{coro:six2}),  $\Pmu$-almost surely,
    \[\int_{\R^2} \theta_{s,t} \d \M =\int_{\R^2} \theta_{s,u} \d \M +\int_{\R^2} \theta_{u,t} \d \M +\int_{\R^2} \epsilon_{s,u,t}\mathbbm{1}_{T_{s,u,t}} \d \M.\]
    This is exactly the announced equality.

    \emph{Second proof.}
    Let us choose $(s,u,t)\in\Delta_3$.
    We decompose the plane according to the values of the two winding functions $\theta_{s,u}$ and $\theta_{u,t}$. Unfortunately, we also have to take the triangle $T_{s,u,t}$ into account, which muddles the proof. We invite the reader to write down the simplified version when $Y_s=Y_u=Y_t$.

    For three relative integers $j$, $k$, and $n$, we define the following sets:
    \begin{align*}
    \mathcal{A}_{j,k}&=\{z\in \R^2: \theta_{s,u}(z)=j,\ \theta_{u,t}(z)=k\},\\
    \mathcal{A}_{n}&=\{z\in \R^2: \theta_{s,t}(z)=n\},\\
    \mathcal{A}^1_{j}&=\{z\in \R^2: \theta_{s,u}(z)=j\},\\
    \mathcal{A}^2_{k}&=\{z\in \R^2: \theta_{u,t}(z)=k\}.
    \end{align*}
    We also define $\mathcal{A}_{j,k}^{\wedge}$, $\mathcal{A}_{n}^{\wedge}$, $\mathcal{A}^{1,\wedge}_{n}$ and $\mathcal{A}^{2,\wedge}_{n}$ the intersection of the triangle $T_{s,u,t}$ with (respectively) $\mathcal{A}_{j,k}$, $\mathcal{A}_{n}$, $\mathcal{A}^1_{n}$ and $\mathcal{A}^2_{n}$.
    %
    %

    For $z\in \mathcal{A}^{\wedge}_{j,k} $, $\theta_{s,t}(z)=j+k+\epsilon_{s,u,t}$.
    Hence, for all integer $n$, \[
\mathcal{A}^{\wedge}_{n}=\bigsqcup_{j,k: j+k+\epsilon_{s,u,t}=n}\mathcal{A}^{\wedge}_{j,k}.
\]
  Besides, for all $j\in \mathbb{N}$,
\[\mathcal{A}^{1,\wedge}_{j}=\bigsqcup_{k\in \mathbb{N}} \mathcal{A}^{\wedge}_{j,k},\]
and for all $k\in \mathbb{N}$,
\[\mathcal{A}^{2,\wedge}_{k}=\bigsqcup_{j\in \mathbb{N}} \mathcal{A}^{\wedge}_{j,k}.\]

  It follows from these relations that
  \begin{align*}
    \int_{T_{s,u,t}} \theta_{s,t} \d \M
    &=\sum_{j,k\in \mathbb{N}^2} (j+k+\epsilon_{s,u,t}) \M(\mathcal{A}^{\wedge}_{j,k})\\
    &=\epsilon_{s,u,t} \M(T_{s,u,t})+ \sum_{j\in \mathbb{N}} j \M(\mathcal{A}^{1,\wedge}_{j})+ \sum_{k\in \mathbb{N}} k \M(\mathcal{A}^{2,\wedge}_{k})\\
    &=\epsilon_{s,u,t} \M(T_{s,u,t})+
    \int_{T_{s,u,t}} \theta_{s,u} \d \M +
    \int_{T_{s,u,t}} \theta_{u,t} \d \M.
  \end{align*}

  We then replace the sets $\mathcal{A}_{j,k}^{\wedge}$, $\mathcal{A}_{n}^{\wedge}$, $\mathcal{A}^{1,\wedge}_{n}$ and $\mathcal{A}^{2,\wedge}_{n}$
with the sets $\mathcal{A}_{j,k}^{\vee}$, $\mathcal{A}_{n}^{\vee}$, $\mathcal{A}^{1,\vee}_{n}$ and $\mathcal{A}^{2,\vee}_{n}$ defined as the intersection of $\R^2 \setminus T_{s,u,t}$ with (respectively) $\mathcal{A}_{j,k}$, $\mathcal{A}_{n}$, $\mathcal{A}^1_{n}$ and $\mathcal{A}^2_{n}$.
The same computations hold, except that for $z\in \mathcal{A}^{\vee}_{j,k} $, $\theta_{s,t}(z)$ is equal to $j+k$ instead of $j+k+\epsilon_{s,u,t}$.
We end up with \[
  \int_{\R^2\setminus T_{s,u,t}} \theta_{s,t} \d \M=
  \int_{\R^2\setminus T_{s,u,t}} \theta_{s,u} \d \M +
  \int_{\R^2\setminus T_{s,u,t}} \theta_{u,t} \d \M, \]
which allows to conclude the proof.
\end{proof}



We now prove the corresponding result for the Brownian motion.
\begin{lemma}
If $\gamma<\sqrt{4/3}$, then $\mathbb{A}^X$ satisfies the weak Chen relation.
\end{lemma}
\begin{proof}
  We invite the reader to skim through the (rather long) proof a first time, and to convince herself or himself that it is merely a question of interchanging the summation order in a double sum,  hence of showing that some residual terms are small.   

  We fix $(s,u,t)\in\Delta_3$.  For simplicity, we assume that $\epsilon(s,u,t)=1$ (that is, the triangle has `positive orientation').

  For a positive integer $N$, we set
  \begin{align*}
  \mathcal{D}_{N}&=\{z: \theta_{X_{|[s,t]}}(z)\geq N\}\\
  \mathcal{D}^1_{N}&=\{z: \theta_{X_{|[s,u]}}(z)\geq N\}\\
  \mathcal{D}^2_{N}&=\{z: \theta_{X_{|[u,t]}}(z)\geq N\},
  \end{align*}
  and as usual the same notation with $N<0$ is used with the inequality reversed.

  For two positive integers $N,M$, we denote by $\mathcal{D}_{N,M}$ the set
  \[ \mathcal{D}_{N,M}=\Big(\big(\mathcal{D}^1_{N}\cup \mathcal{D}^1_{-N} \big)\cap
  \big(\mathcal{D}^2_{M}\cup \mathcal{D}^2_{-M} \big)\Big)\cup \Big(\big(\mathcal{D}_{M}^1\cup \mathcal{D}^1_{-M} \big)\cap
  \big(\mathcal{D}^2_{N}\cup \mathcal{D}^2_{-N} \big)\Big).\]
  This is the set of points $z$ such that among $|\theta_{s,u}(z)|$ and $|\theta_{u,t}(z)|$, one is at least $N$ and the other is at least $M$.
  We also use the notations $\mathcal{A}_{N}$ $\mathcal{A}_{k,j}$,$\mathcal{A}^1_{k}$ and $\mathcal{A}^2_{j}$, of the previous proof.

  We will use the following bounds: for all $\epsilon>0$, there exists $C$ such that for all $N,M\geq 1$,
  \begin{equation}
  \label{eq:momentJoints}
  \mathbb{E}^X[|\mathcal{D}_{N,M}| ]\leq C (NM)^{-1+\epsilon}.
  \end{equation}
  This is an easy consequence of Lemma 2.4 in \cite{LAWA}.

\ref{le:boundMomentP}

  %
  For a point $z\in\R^2 \setminus (T_{s,u,t}\cup \Range(X))$, it is easily seen that
  $z\in \mathcal{A}_{N}$ if and only if
there exists $k\in \mathbb{Z}$ such that
  $z\in \mathcal{A}^1_{k}\cap \mathcal{A}^2_{N-k}$, in which case this $k$ is unique. We let $\M^\vee$ be the (random) measure defined by $\M^\vee(A)=\M(A\setminus T_{s,u,t})$.

  For all $N$, we compute
  \begin{align*}
    \sum_{k=1}^{N} \M^\vee(\mathcal{D}_{k}   )
    &=\sum_{k=1}^{N}\sum_{l=k}^{+\infty} \M^\vee(\mathcal{A}_{l} )= \sum_{k=1}^{N} \sum_{l=k}^{+\infty}  \sum_{j=-\infty}^{+\infty} \M^\vee(\mathcal{A}^1_{j} \cap \mathcal{A}^2_{l-j})= \sum_{k=1}^{N}  \sum_{j=-\infty}^{+\infty} \sum_{m=k-j}^{+\infty}  \M^\vee(\mathcal{A}^1_{j} \cap \mathcal{A}^2_{m})\\
    &=\sum_{j=-\infty}^{+\infty} \sum_{m=1-j}^{+\infty} \min(N, m+j)  \M^\vee(\mathcal{A}^1_{j} \cap \mathcal{A}^2_{m})\\
    &=\sum_{j=-\infty}^{+\infty} \sum_{m=-\infty}^{+\infty} \max(0,\min(N, m+j))  \M^\vee(\mathcal{A}^1_{j} \cap \mathcal{A}^2_{m}).
  \end{align*}

This computation remains true if we replace each set $\mathcal A^{*}_{\ell}$ by $\mathcal A^{*}_{-\ell}$ and leave everything else unchanged. Doing this susbstitution, subtracting the resulting equality from the one that we just obtained, and using the notation $[j]_n=\max(-n,\min(n,j))$, we find
  \begin{equation}
  \label{eq:som1}
  \sum_{k=1}^{N} \big(\M^\vee(\mathcal{D}_{k}   )-\M^\vee(\mathcal{D}_{-k}   )\big)
  =\sum_{j=-\infty}^{+\infty} \sum_{m=-\infty}^{+\infty} [m+j]_N  \M^\vee(\mathcal{A}^1_{j} \cap \mathcal{A}^2_{m}).
  \end{equation}

  On the other hand, using the fact that $(\mathcal{A}^2_{m})_{m\in\mathbb{Z}}$ is a partition of $\R^2$, we find, by a superficially identical, but in fact different computation,
  \begin{align*}
    \sum_{k=1}^{N} \M^\vee(\mathcal{D}^1_{k})
    &=
    \sum_{k=1}^{N} \sum_{j=k}^{+\infty}
    \M^\vee(\mathcal{A}^1_{j})  =
    \sum_{k=1}^{N} \sum_{j=k}^{+\infty} \sum_{m=-\infty}^{+\infty}
    \M^\vee( \mathcal{A}^1_{j}\cap\mathcal{A}^2_{m} )\\
    &=\sum_{j=1}^{+\infty} \sum_{m=-\infty}^{+\infty } \min(N,j)  \M^\vee(\mathcal{A}^1_{j} \cap  \mathcal{A}^2_{m})\\
    &=\sum_{j=-\infty}^{+\infty} \sum_{m=-\infty}^{+\infty } \max(0,\min(N,j))  \M^\vee(\mathcal{A}^1_{j} \cap  \mathcal{A}^2_{m}).
  \end{align*}
Replacing $k$ by $-k$ as we did before and combining the two results, we obtain
  \begin{equation}
    \label{eq:som2}
    \sum_{k=1}^{N} \big(\M^\vee(\mathcal{D}^1_{k})-\M^\vee(\mathcal{D}^1_{-k})\big)
  =\sum_{j=-\infty}^{+\infty} \sum_{m=-\infty}^{+\infty } [j]_N  \M^\vee(\mathcal{A}^1_{j} \cap  \mathcal{A}^2_{m}).
  \end{equation}
  The same equation holds after exchanging the superscript $1$ and $2$, so that
  \begin{align}
    &\sum_{k=1}^{N} \Big(\big(\M^\vee(\mathcal{D}_{k})-\M^\vee(\mathcal{D}_{-k})\big) -\big( \M^\vee(\mathcal{D}^1_{k})-\M^\vee(\mathcal{D}^1_{-k})+\M^\vee(\mathcal{D}^2_{k})-\M^\vee(\mathcal{D}^2_{-k}) \big)\Big)\nonumber\\
&\hspace{5cm}  =\sum_{k,j=-\infty}^{+\infty} \big( [k+j]_{N} -[j]_{N} -[k]_{N}  \big)  \M^\vee(\mathcal{A}^1_{j} \cap  \mathcal{A}^2_{k}).
  \label{eq:bigsum}
  \end{align}

  Our goal is now to show that this sums goes to $0$ as $N$ goes to infinity. To this end, we decompose $\mathbb{N}^2$ as follows. We fix a parameter $m\in(0,1)$ and set $M=\lfloor N^m\rfloor$. We then partition $\mathbb{N}^2$ into five subsets $E_1, E_2,E_3,E_4,E_5$, illustrated on Figure \ref{fig:cases}.
  \begin{align*}
  &E_1=[0,\tfrac{N}{2})^2,\\[-2pt]
  &E_2=([M,+\infty) \times [\tfrac{N}{2},+\infty))\cup ([\tfrac{N}{2},+\infty)  \times[M,+\infty) ), \\
  &E_3= ([0,M)\times [N-M,N+M))\cup( [N-M,N+M)\times [0,M)),
  \\
  &E_4=  ([0,M)\times [N+M,+\infty))\cup ( [N+M,+\infty)\times[0,M) ),\\[-1pt]
  &E_5= ([0,M)\times [\tfrac{N}{2},N-M))\cup ([\tfrac{N}{2},N-M)\times [0,M)).
  \end{align*}

  \begin{figure}[h!]
  \begin{center}
  \mbox{
  \includegraphics[keepaspectratio,height=7cm]{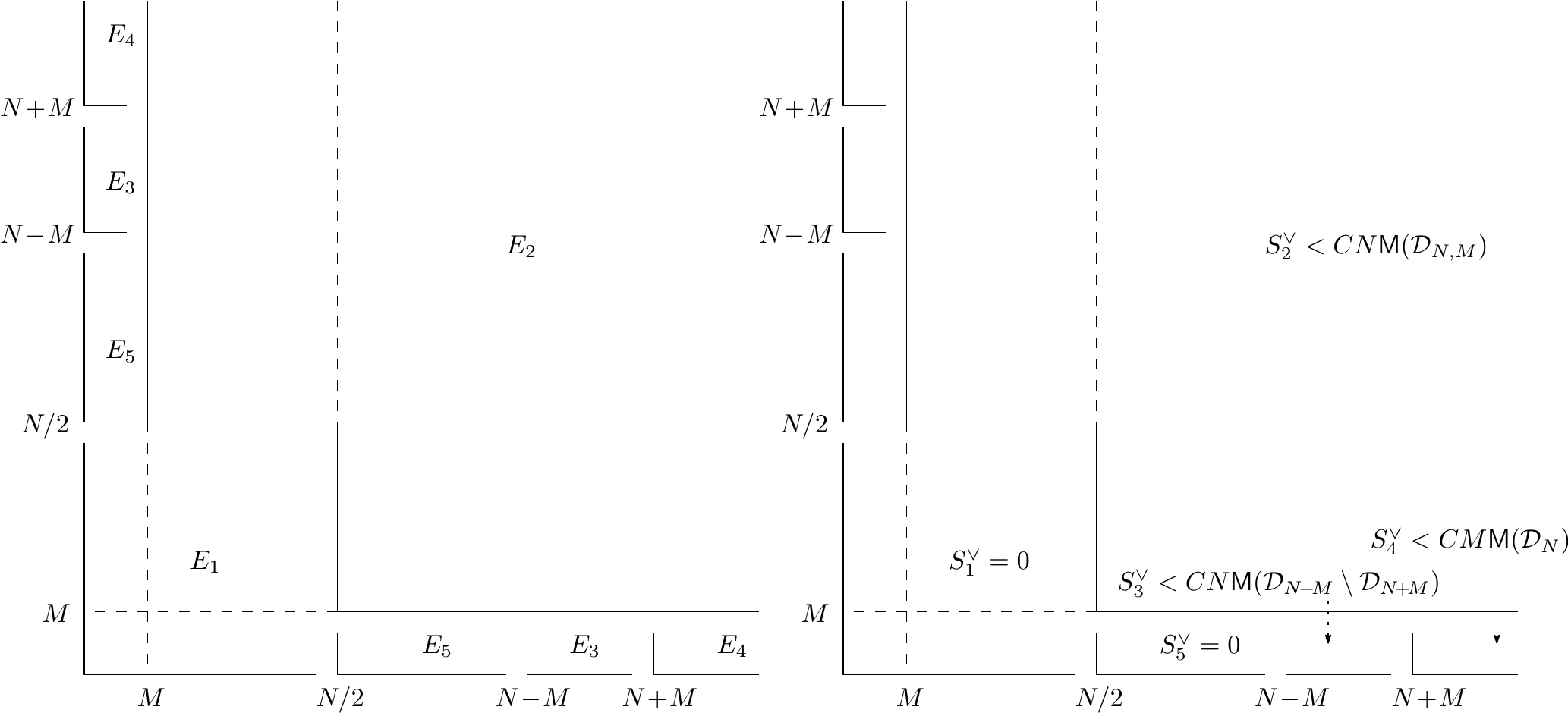}
  }
  \caption{\label{fig:cases}
  On the left:
  decomposition of $\mathbb{N}^2$. On the right: approximate bounds on the corresponding sum.
  }
  \end{center}
  \end{figure}

We now partition $\mathbb{Z}^2$ into the sets $F_i=\{(k,j)\in \mathbb{Z}^2: (|k|,|j|)\in E_i\}$, for $i\in \{1,\ldots,5\}$.
We decompose the sum \eqref{eq:bigsum} accordingly into five sums $S_1^{\vee},\dots, S_5^{\vee}$.

$\bullet$ It is easily seen that $S_1^{\vee}=S_5^{\vee}=0$.

$\bullet$  For $(k,j)\in F_2$, the inequality $|(k+j)_{N} -(j)_{N} -(k)_{N}|\leq N $ holds. Moreover, the $\mathcal{A}^1_{j} \cap  \mathcal{A}^2_{k}$ are disjoint subsets of $\mathcal{D}_{N,M}$.
  Hence,
  \[|S_2^{\vee}|
  \leq N \M^\vee( \mathcal{D}_{N,M}).\]

  $\bullet$
  For $(k,j)\in F_3$, $|(k+j)_{N} -(j)_{N} -(k)_{N}|\leq N $,
  and the $\mathcal{A}^1_{j} \cap  \mathcal{A}^2_{k}$ are disjoint subsets of $(\mathcal{D}^1_{N-M}\setminus \mathcal{D}^1_{N+M+1})\cup (\mathcal{D}^1_{N-M}\setminus \mathcal{D}^2_{N+M+1})$.
  Hence, \[|S_3^{\vee}| \leq N \big( \M^\vee(\mathcal{D}^1_{N-M}\setminus \mathcal{D}^1_{N+M+1}) + \M^\vee(\mathcal{D}^2_{N-M}\setminus \mathcal{D}^2_{N+M+1})\big).
  \]

$\bullet$
  For $(k,j)\in F_4$, $|(k+j)_{N} -(j)_{N} -(k)_{N}|\leq 2 M $, and the $\mathcal{A}^1_{j} \cap  \mathcal{A}^2_{k}$ are disjoint subsets of $\mathcal{D}^1_{N}\cup \mathcal{D}^1_{-N}\cup
  \mathcal{D}^2_{N}\cup
  \mathcal{D}^2_{-N}$. Hence,
  \[ |S_4^{\vee}|
  \leq 2M   \big(\M^\vee( \mathcal{D}^1_{N})+\M^\vee(\mathcal{D}^2_{N})\big).
  \]

  Altogether, we have
  \begin{multline} |S_1^{\vee}+S_2^{\vee}+S_3^{\vee}+S_4^{\vee}+S_5^{\vee}|
  \\
  \leq N\big(\M^\vee( \mathcal{D}_{N,M}) + \M^\vee(\mathcal{D}^1_{N-M}\setminus \mathcal{D}^1_{N+M+1}) + \M^\vee(\mathcal{D}^2_{N-M}\setminus \mathcal{D}^2_{N+M+1})
  \big)+\\
  2M \big(\M^\vee( \mathcal{D}^1_{N})+\M^\vee(\mathcal{D}^2_{N})\big).
  \end{multline}
  We take the expectation under $\Pmu$ on both sides. Using the fact that the intensity of the random measure $\M$ is the Lebesgue measure, we obtain
  \[\Emu\big[|S_1+S_2+S_3+S_4+S_5|\big]\leq N\big(|\mathcal{D}_{N,M}| + |\mathcal{D}^1_{N-M}\setminus \mathcal{D}^1_{N+M+1}|) + |\mathcal{D}^2_{N-M}\setminus \mathcal{D}^2_{N+M+1}|
  \big)+\\
  2M \big(|\mathcal{D}^1_{N}|+|\mathcal{D}^2_{N}|)\big).\]
  We take the expectation under $\PX$ on both sides. Using Equation \eqref{eq:momentJoints}, as well as the fact that $|\mathcal{A}_N|$ is equivalent in $L^2$ (hence in $L^1$) to $\frac{|t-s|}{2\pi N^2}$ (recall Equation \ref{eq:werner}),  we obtain
  \[\mathbb{E}\big[|S_1+S_2+S_3+S_4+S_5|\big]\leq
  C' N ( (NM)^{-1+\epsilon}+ MN^{-2})+MN^{-1}\underset{N\to +\infty}{\longrightarrow} 0,
\]
  where the last convergence holds for an arbitrary choice of $m\in(0,1)$ and $\epsilon\in (0,\frac{1}{2})$.

  From this long discussion, it follows that, $\mathbb P$-almost surely,
  \[
     \sum_{k=1}^{+\infty} \big(\M^\vee(\mathcal{D}_{k}   )-\M^\vee(\mathcal{D}_{-k}   )\big)
    =\sum_{k=1}^{+\infty} \big(\M^\vee(\mathcal{D}^1_{k}   )-\M^\vee(\mathcal{D}^1_{-k}   )\big)
    + \sum_{k=1}^{+\infty} \big(\M^\vee(\mathcal{D}^2_{k}   )-\M^\vee(\mathcal{D}^2_{-k}   )\big).
  \]

  For a point $z$ in the interior of $T_{s,u,t}$ (and outside the range of $X$), the relation between $\mathcal{A}_{N}$ and the $\mathcal{A}^1_{k}, \mathcal{A}^2_{j}$ has to be shifted by $1$:
  \[z\in \mathcal{A}^1_{N} \Longleftrightarrow
  \exists k\in \mathbb{Z}: z\in \mathcal{A}^1_{k}\cap \mathcal{A}^2_{N-k-1}.\]
  This explains the apparition of the additional term $\M(T_{s,u,t})$ in the Chen relation. Computations similar to the previous ones lead to the equality
  \begin{multline}
     \sum_{k=1}^{+\infty} \big(\M(\mathcal{D}_{k}  \cap T_{s,u,t} )-\M(\mathcal{D}_{-k}  \cap T_{s,u,t}  )\big)
    =\sum_{k=1}^{+\infty} \big(\M(\mathcal{D}^1_{k}  \cap T_{s,u,t}   )-\M(\mathcal{D}^1_{-k}   \cap T_{s,u,t} )\big)
  \\  + \sum_{k=1}^{+\infty} \big(\M(\mathcal{D}^2_{k}   \cap T_{s,u,t}  )-\M(\mathcal{D}^2_{-k}   \cap T_{s,u,t} )\big)+\M(  T_{s,u,t}).
  \end{multline}
  Combining the two equalities gives the desired result.
\end{proof}

\section{From weak Chen's relation to pathwise Chen's relation}

The goal in this section is to show that, up to modification, a map $\mathbb{A}$ that satisfies the weak Chen relation does satisfies the Chen relation. For this, we need the $\M$-measure of a triangle to be a continuous function of its vertices, and we are able to do this only under the assumption that $\gamma<2(\sqrt{2}-1)$.

Let us first recall some terminology for functions of two parameters. For functions of three parameters, we use the same definitions with $\Delta$ replaced by $\Delta_3=\{(s,u,t)\in[0,1]^3: s\leq u\leq t\}$.

\begin{definition}
\label{def:modif}
Let $X$ and $\tilde{X}$ be two collections of random variables on the same probability space, both indexed by $\Delta$. We say that they are \emph{modifications} of each other, or that one is a modification of the other, if for all $w\in \Delta$, almost surely, $X_w=\tilde{X}_w$.
\end{definition}

For instance, the collections $ \mathbb{A}$ defined by \eqref{eq:def:AcalB} and \eqref{eq:def:Acal} are defined only up to modification.
\begin{definition}
A collection $X$ of random variables indexed by $\Delta$ and with values in $\R^d$ is said to be \emph{separable} {\rm (}with respect to the class of closed sets{\rm)} if there exist a countable set $I$ {\rm(}called the separability set{\rm)} and a negligible event $\mathcal{N}$ such that for all open set $U$ of $\Delta$ and all closed set $F$ of $\mathbb{R}^d$, the following inclusion holds:
\[ \{\forall w\in U\cap I : X_w\in F\}
\setminus \{\forall w\in U : X_w\in F\}
\subseteq  \mathcal{N}.
\]
\end{definition}
Our impression is that the terminology of separability, which might have been very commonly used in the past, has gone lost with time. It seems to us that modern introductions to stochastic processes tend to forget about it, to the profit of stronger properties such as continuity or càdlàg property.
We will use some results which can be found in  \cite{doob} and \cite{gikhman}. A discussion of these questions can also be found in \cite[Chapter IV, 24-30]{dellacherie}. In the first cited text, things are stated for functions from $\R$ to $\R$, but the results that we use extend without any technical complications to our situation. On the contrary, the framework is much more general in the second cited text. We state these results in the form which is adapted to our framework.
The reason why we use separability is that, in order to prove some regularity result, we need to \emph{first} prove the pathwise Chen relation. This relation, since it can rewritten as
\[
\forall (s,u,t)\in \Delta_3, \ \delta \mathbb{A}_{s,u,t}= \mathbb{A}_{s,t}-\mathbb{A}_{s,u}-\mathbb{A}_{u,t}-\epsilon_{s,u,t}\M(T_{s,u,t})\in \{0\},
\]
clearly follows from the weak one, \emph{provided the family $\delta \mathbb{A}$ is separable}. The following result, due to J.L. Doob, states that any family has a separable modification. Though we will use this result, it is not directly sufficient to us:
what we want is not a modification of $\delta \mathbb{A}$ equal to $0$,
but a modification $\tilde{\mathbb{A}}$ of $\mathbb{A}$ such that the corresponding family $\delta \tilde{\mathbb{A}}$ is separable.
To show the existence of such a family is the main purpose of this section.
\begin{lemma}[{\cite[Theorem 2.4]{doob}, \cite[Theorem 1 in Section III.2]{gikhman}}]
Let $X$ be a collection of random variables in the same probability space, indexed by $\Delta$ {\rm (}resp. $\Delta_3${\rm)} and with values in $\R^d$. Then, there exists a separable modification of $X$.
\end{lemma}
We will also need the following characterization of the separability condition.
\begin{lemma}[{\cite[Lemma 1 in Section III.2]{gikhman}}]
\label{le:gikhman}
  A collection $(X_\mathbf{t})_{\mathbf{t}\in \Delta_3}$ is separable if and only if there exists a negligible set $\mathcal{N}$ and a countable set $I\subseteq \Delta_3$ such that for all $\omega\in\Omega\setminus   \mathcal{N}$ and $\mathbf{t}\in \Delta_3$, the value $X_\mathbf{t}(\omega)$ lies in
  \[ \bigcap_{U} \overline{ \{X_s(\omega): s\in U\cap I \} },\]
  where $U$ ranges over the open sets in $\Delta_3$ containing $\mathbf{t}$.
\end{lemma}
This characterization allows us to prove the following result, which we were unable to find in the literature.
\begin{corollary}
  Assume that $(X_\mathbf{t})_{\mathbf{t}\in \Delta_3}$ and $(Y_\mathbf{t})_{\mathbf{t}\in \Delta_3}$ are separable. Then, $(X_\mathbf{t},Y_\mathbf{t})_{\mathbf{t}\in \Delta_3}$ and $(X_\mathbf{t}+Y_\mathbf{t})_{\mathbf{t}\in \Delta_3}$ are separable.
\end{corollary}
\begin{proof}
  Let $I_1,\mathcal{N}_1$ (resp. $I_2,\mathcal{N}_2$) be the sets that appear in the characterization of the separability of $X$ (resp. $Y$).
  Let $I=I_1\cup I_2$, and $\mathcal{N}=\mathcal{N}_1\cup \mathcal{N}_2$.
  For all $\mathbf{t}\in \Delta_3$, and $U$ open set in $\Delta_3$ containing $\mathbf{t}$, for all $\omega\in\Omega\setminus   \mathcal{N}$, we know that
  $X_\mathbf{t}(\omega)$ lies in $\overline{ \{X_\mathbf{s}(\omega): \mathbf{s}\in U\cap I \} }$ and that $Y_\mathbf{t}(\omega)$ lies in $\overline{ \{Y_\mathbf{s}(\omega): \mathbf{s}\in U\cap I \} }$.
  Hence $(X_\mathbf{t},Y_\mathbf{t})$ lies in \[\overline{ \{X_\mathbf{s}(\omega): \mathbf{s}\in U\cap I \} }\times \overline{ \{Y_\mathbf{s}(\omega): \mathbf{s}\in U\cap I \} }=\overline{ \{(X_\mathbf{s}(\omega),Y_\mathbf{s}(\omega)): \mathbf{s}\in U\cap I \} }.\]
  This allows us to conclude to the first point.

  For the second, we use the definition rather than the characterization. Let $I,\mathcal{N}$ be the sets that appear in the definition of the separability of $(X,Y)$.
  Let $U\subseteq \Delta_3$ be an open set, and $F\subset \R^d$ be a closed set. Let $\pi:(\R^d)^2\to \R^d$ be the map $(x,y)\mapsto x+y$.
  Then, $\pi^{-1}(F)$ is closed, so that
  \begin{align*}
  \{\forall \mathbf{t}\in U, (X+Y)_\mathbf{t}(\omega)\in F\}&= \{\forall \mathbf{t}\in U, (X_\mathbf{t}(\omega), Y_\mathbf{t}(\omega))\in \pi^{-1}(F)\}\\
  &\subseteq  \{\forall \mathbf{t}\in U\cap I , (X_\mathbf{t}(\omega), Y_\mathbf{t}(\omega))\in \pi^{-1}(F)\}\cup \mathcal{N}\\
  &=\{\forall \mathbf{t}\in U\cap I , (X+Y)_\mathbf{t}(\omega)\in F\}\cup \mathcal{N}.
  \end{align*}
  Hence $X+Y$ is separable.
\end{proof}

In particular, for
 $\delta \mathbb{A}=(\mathbb{A}_{s,t}-\mathbb{A}_{s,u}-\mathbb{A}_{u,t}-\epsilon_{s,u,t} \M(T_{s,u,t}))_{(s,u,t)\in \Delta_3}$
to be separable, it suffices that
\begin{itemize}
\item $(\mathbb{A}_{s,t})_{(s,u,t)\in \Delta_3}$ be separable (or equivalently, that $(\mathbb{A}_{s,t})_{(s,t)\in \Delta}$ be separable), and
\item $(\M(T_{s,u,t}))_{(s,u,t)\in \Delta_3}$ be separable.
\end{itemize}
For the first point, we know from Doob's lemma that  $(\mathbb{A}_{s,t})_{(s,t)\in \Delta}$ admits a separable modification. For the second point, the problem is posed in a slightly different way, because $(\M(T_{s,u,t}))_{(s,u,t)\in \Delta_3}$ is defined not as a collection of random variables indexed by $\Delta_3$, but really as a random function on $\Delta_3$. Taking a modification of it to ensure its separability would possibly destroy the structure given by the fact that $\M$ is a measure: there is no reason why a modification $(\tilde{m}_{s,u,t})_{(s,u,t)\in \Delta_3}$ of $(\M(T_{s,u,t}))_{(s,u,t)\in \Delta_3}$ would be of the form
$\tilde{m}_{s,u,t}=\tilde{\M}(T_{s,u,t})$ for a random measure $\tilde{\M}$.\footnote{Actually, random measures are entirely characterized by their finite dimensional marginals $\M(A_i)_{i\in I}$ (see \cite{daley}), so that the only modifications of $\M$ which are random measures are indistinguishable from $\M$.}

In order to show that the map $(s,u,t)\mapsto \M(T_{s,u,t})$ is separable (not up to modification), the only way that we found is to prove a much stronger result, for which we need $\gamma$ to be smaller than $2(\sqrt{2}-1)$. For $\mathbf{z}=(z_1,z_2,z_3)\in (\R^2)^3$, we set $T^{\mathbf{z}}$ the convex hull of the three points $z_1,z_2,z_3$. We then define $a:(\R^2)^3\to \R^+$ the map given by $a(\mathbf{z})=\M(T^\mathbf{z})$.

\begin{lemma}
For all $\gamma<2$, the random map $a$ is separable. 
\end{lemma}
\begin{proof}
Let $\mathcal{N}$ be the negligible event on which there exists a compact subset of $\R^2$ with infinite $\M$-area. Any triangle $T$ is the intersection of a decreasing sequence of triangles with rational vertices, and on the complement of $\mathcal{N}$, the measure of $T$ is the decreasing limit of the measures of these rational triangles.
Hence, $a$ satisfies the separability criterion of Lemma \ref{le:gikhman}.
\end{proof}

We will now improve this result and show that the map $a$ is actually continuous, and in fact H\"older continuous, under the additional assumption that $\gamma<2(\sqrt{2}-1)$.

\begin{lemma}
  \label{le:muContinue}
  Let $\gamma<2(\sqrt{2}-1)$. Then, almost surely, the map $a$ is continuous, and locally $\beta$-H\"older continous for any $\beta<1-\sqrt{2}\gamma+\frac{\gamma^2}{4}$.
\end{lemma}
\begin{proof}
  We show the result for the restriction of $a$ on the set of triangles contained in the box $[-1,1]^2$. The global continuity can be deduced by scaling or by a covering argument.

  For $x$ on the boundary of $[-1,1]^2$, $\epsilon>0$ and $\theta \in [-\frac{\pi}{2},\frac{\pi}{2}]$, we denote by $R^{\epsilon}_{x,\theta}$ the rectangle with length $4\sqrt{2}$ and width $\epsilon$, which is centered at $x$ and with angle $\theta$ with respect to the $x$-axis (see Figure \ref{fig:cover} below).
  \begin{figure}[h!]
  \begin{center}
  \mbox{
  \includegraphics[keepaspectratio,height=7cm]{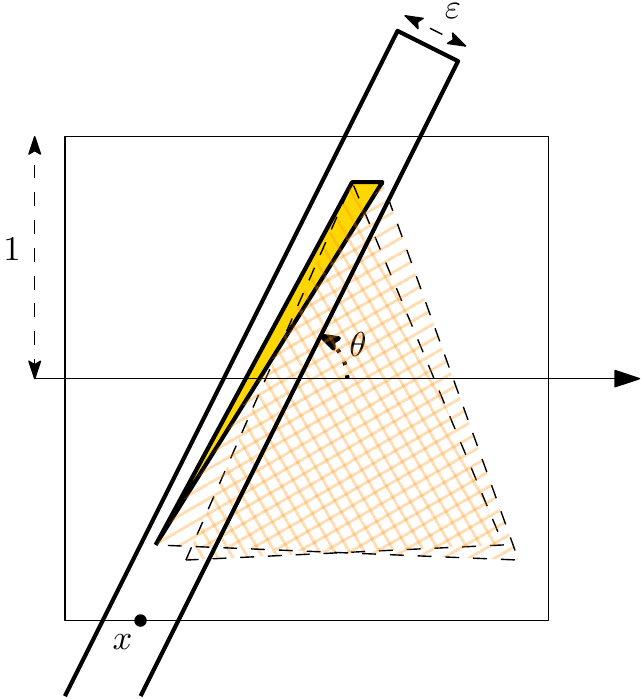}
  }
  \caption{\label{fig:cover}
  Two triangles with close vertices, one of the six triangles that cover there symmetric difference, and a rectangle $R^\epsilon_{x,\theta}$ that contains this triangle.}
  \end{center}
  \end{figure}

  Let $n$ be an integer such that $n\epsilon>1$. Set
  \[I=\{(1,\tfrac{k}{n}) : |k|\leq n\}\cup \{(\tfrac{k}{n},1) : |k|\leq n\} \ \text{ and } \ J=\{\tfrac{k \pi}{2n}: |k|\leq n\}.
  \]


  Consider a triangle $T$ with vertices $z_1,z_2,z_3$ in $[0,1]^2$ such that $|z_2-z_3|\leq \epsilon$. The line through $z_1$
 and the middle of $z_2$ and $z_3$ crosses the boundary of $[-1,1]^2$ at a point close of a point of $I$, with an angle close to an element of $J$. Thus, there exists $(i,j)\in I\times J$
  such that $T$ is included on $R^{4\epsilon}_{i , j}$.

Let us call {\em $\epsilon$-thin} a triangle such as the one that we just considered, that is, a triangle of which two vertices are $\epsilon$-close. The symmetric difference between two triangles, the vertices of which are pairwise $\epsilon$-close, is contained in the union of six $\epsilon$-thin triangles (see Figure \ref{fig:cover} again). Thus, if $\mathbf{z}=(z_1,z_2,z_3)$ and $\mathbf{z}'=(z'_1,z'_2,z'_3)$ are such that $|\mathbf{z}-\mathbf{z}'|\leq \epsilon$, then \[|a(\mathbf{z})-a(\mathbf{z}')| \leq 6 \max_{(i,j)\in I\times J} R^{4\epsilon}_{i,j}.\]


  We bound this supremum as in the proof of the Kolmogorov criterion, using the fact that the cardinal of $I\times J$ is of order $\epsilon^{-2}$. For any $\beta<1-\sqrt{2}\gamma+\frac{\gamma^2}{4}$ and $q\in [0,\frac{4}{\gamma^2})$,
  \begin{align*}
  \mathbb{P}( \max_{(i,j)\in I\times J} R^{4\epsilon}_{i,j}\geq \epsilon^\beta )
 &\leq  \epsilon^{-\beta q} \mathbb{E} \sum_{(i,j)\in I\times J}(R^{4\epsilon}_{i,j})^q  \leq C \epsilon^{-2} \epsilon^{-\beta q}  \mathbb{E}[ (R^{4\epsilon}_{i,j})^{q}]
  \leq C  \epsilon^{-2-\beta q  +\frac{\xi(q)}{2} }.
  \end{align*}
  The exponent is minimized by the choice of $q=\frac{2\sqrt{2}}{\gamma}$ (which is less strictly than $\frac{4}{\gamma^2}$), and the bound on $\beta$ is such that the exponent is then strictly positive.
  To conclude, we take $\epsilon=2^{-n}$ and we apply the Borel--Cantelli lemma.
\end{proof}

Let us summarize. For $\gamma<2$, for any map $\mathbb{A}:\Delta\to \R$ which satisfies the weak Chen relation relative to a continuous path $Z$, we know that there exists a separable modification of $\tilde{\mathbb{A}}$ of $\mathbb{A}$, which is easily seen to also satisfy the weak Chen relation.
We also know
that the map $(s,u,t)\mapsto \M( T^{Z_s,Z_u,Z_t})$ is separable.

Hence, we have obtained the following result.

\begin{proposition}
\label{le:chenFort}
For all $\gamma<2$, for all map $\mathbb{A}:\Delta\to \R$ which satisfies the weak Chen relation relative to a continuous path $Z$,
there exists a modification of $\mathbb{A}$ which is separable and satisfies the (strong) Chen relation relative to $Z$.
\end{proposition}

We will now adress the question of the regularity of such a map $\mathbb{A}$. Before that, let us remark that the Chen relation, together with Lemma \ref{le:muContinue}, allows us to deduce H\"older continuity from regularity. Recall from \eqref{eq:defbetareg} the definition of $\beta$-regularity.

\begin{lemma}
  \label{le:RegToHold}
  Let $\gamma<2(\sqrt{2}-1)$. Let $\mathbb{A}:\Delta\to\R$ be a map that satisfies the Chen relation relative to a function $Z$ which is $\alpha$-H\"older continuous, for some $\alpha>0$. Assume that $\mathbb{A}$ is $\beta$-regular.
  Then, for all $\beta'$ such that $\beta'\leq \beta$ and $\beta'< ( 1-\sqrt{2}\gamma+\frac{\gamma^2}{4} )\alpha$, the map $\mathbb{A}$ is $\beta'$-H\"older continuous. 
\end{lemma}



\section{A Kolmogorov type criterion}
\label{sec:kolmo}
The goal of this section is to obtain a Kolmogorov type criterion that applies to our situation, that is a result that allows us to deduce some pathwise regularity (in particular, continuity) of a map $\mathbb{A}(\omega):\Delta\to \R$ from regularity of the map $\mathbb{A}:\Delta\to L^q(\Omega)$.

\begin{proposition}
\label{le:six:preExten}
  Assume that $\gamma<\sqrt{2}$, $\alpha\in \big(\gamma^2\big(1+\frac{\gamma^2}{4}\big)^{-2},1\big]$, and $\xi>\frac{1}{2}$.
Set
\[\beta_1=\left\{\begin{array}{ll}
\min(  \alpha\nu-1, \xi-\frac{1}{2}) & \text{if } \alpha\geq \gamma^{-2},\\
\min(2\alpha(1+\frac{\gamma^2}{4})-2\gamma\sqrt{\alpha},  \alpha\nu-\frac{1}{2}, \xi-\frac{1}{2}) & \text{if }
\alpha \leq  \gamma^{-2}.
\end{array}\right.\]


  Let $Z\in \mathcal{C}^\alpha$ be a continuous function from $[0,1]$ to $\R^2$, possibly random but independent from $\M$ (say, defined in a probability space $\Omega^X$).  
  Assume that $\mathbb{A}$ is separable and satisfies the weak Chen relation relative to $Z$, and assume that there exists a positive random variable $C_0$ on $\Omega^X$ such that for all $(s,t)\in\Delta$, $\PX$-almost surely,
  $\| \mathbb{A}_{s,t}\|_{L^2(\Omega^\M,\Pmu) }\leq C_0 (t-s)^\xi$.

  Then, $\mathbb{A}$ is almost surely $\beta_1$-regular: $\mathbb{P}$-almost surely, for all $\beta<\beta_1$, there exists $C$ such that for all $(s,t)\in \Delta$,
  \[ |\mathbb{A}_{s,t}| \leq C (t-s)^{\beta}.  \]
\end{proposition}

\begin{proof}
We write \[\mathbb{D}=\{s\in[0,1): \exists i,j\in \mathbb{N}, s=i2^{-n}\}\]
the set of dyadic numbers. We set
$\beta_0=\alpha\nu-1$ if $\alpha\geq \gamma^{-2}$, and $\beta_0=2\alpha(1+\frac{\gamma^2}{4})-2\gamma\sqrt{\alpha}$ if
$\alpha \leq  \gamma^{-2}$. In particular, $\beta_1=\min( \beta_0, \alpha\nu-\frac{1}{2},\xi-\frac{1}{2})$.
  We rely on the classical proof of Kolmogorov criterion for rough paths.  We follow in  particular the proof of Theorem 3.1 in \cite{hairerFriz} (with $q=2$), but with some modification.
  By localization, we can assume that both $C_0$ and $ \|Z\|_{\mathcal{C}^\alpha}$ admit poynomial moments of all order.

 For all $n\geq 0$, we denote by
  $\mathbb{D}_n$ the set of integer multiples of $2^{-n}$ in $[0,1)$. 

  We now introduce a few notations for some rectangles and triangles. Figure \ref{fig:rect} should help understand the notations.
  For $s\leq t$ and $Z$ equal to either $X$ (a Brownian motion) or $Y$ (an $\alpha$-H\"older continuous curve), we denote by $T_{s,t}$ the triangle with vertices $Z_s$ ,$Z_t$, and $(Z_t^1, Z_s^2)$.
  For $s\leq u\leq  t$, we denote (as before) $T_{s,u,t}$ the triangle with vertices $Z_s$, $Z_u$ and $Z_t$, and  we denote $R_{s,u,t}$ the rectangle $ [Z^1_u,Z^1_t]\times [Z^2_s, Z^2_u]$. 
  Finally, we define $R_{s,t,u,v}$ as the rectangle $[Z^1_s,Z^1_t]\times [Z^2_u, Z^2_v]$.
  \begin{figure}[h!]
  \begin{center}
    \mbox{
    \includegraphics[scale=0.7]{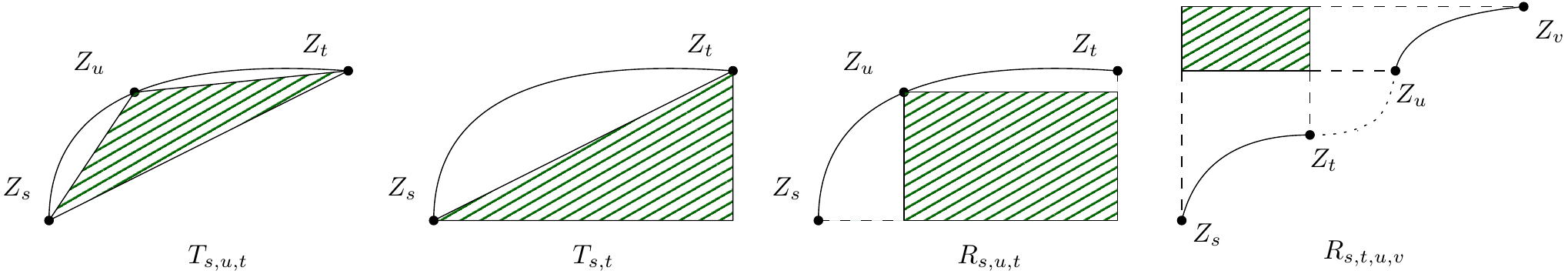}   
      }
    \caption{\label{fig:rect}
      The triangles $T_{s,u,t}$ and$T_{s,t}$, and the rectangles $R_{s,u,t}$ and $R_{s,t,u,v}$.
    }
  \end{center}
  \end{figure}
  Let us recall that $\epsilon_{s,u,t}\in\{\pm 1\} $ is equal to $1$ if the points $Z_s$, $Z_t$ and $Z_u$ appears in trigonometric order along the boundary of $T_{s,u,t}$. We similarily define $\epsilon_{s,t}$ as equal to  $1$ if the points $Z_s$, $(Z^1_t,Z^2_s)$ and $Z_t$ appears in trigonometric order along the boundary of $T_{s,t}$, and equal to $-1$ otherwise. 
  In Figure \ref{fig:rect}, $\epsilon_{s,t}=\epsilon_{s,u,t}=1$. These values are such that for all $s\leq u\leq t$,
  \[\epsilon_{s,u,t}\M(T_{s,u,t})+ \epsilon_{s,t}\M(T_{s,t}) =\epsilon_{s,u,t}\M(R_{s,u,t})+\epsilon_{s,u}\M(T_{s,u})+\epsilon_{u,t}\M(T_{u,t}).\]
  Besides, $R_{s,u,t}=R_{s,u,u,t}$.
  We denote
  \begin{equation}
  \label{eq:def:Abar}
  \mathcal{A}^Y_{s,t}= \mathbb{A}^Y_{s,t}+ \epsilon_{s,t} \M(T_{s,t}).
  \end{equation}

  We also set
  \begin{equation}
\label{eq:defJnn}
 J_{n,n'}= \max_{\substack{i\in \mathbb{D}_n \\ j\in \mathbb{D}_{n'} } }   \M(R_{i,i+2^{-n},j,j+2^{-n'}}).
  \end{equation}
  The second moment of this variable can be estimated as follows:
  \begin{equation}
  \label{eq:secondMoment}
  \mathbb{E} [ J_{n,n'}^2 ]
   \leq \sum_{ \substack{i\in \mathbb{D}_n \\ j\in \mathbb{D}_{n'} } } \mathbb{E}[ \M(R_{i,i+2^{-n},j,j+2^{-{n'}}})^2]\leq C \sum_{ \substack{i\in \mathbb{D}_n \\ j\in\mathbb{D}_{n'} } } |R_{i,i+2^{-n},j,j+2^{-{n'}}}|^{ \nu }\leq C 2^{(1-\alpha \nu)(n+n')}
  \end{equation}
  with $C=\mathbb{E}[\|Z\|_{\mathcal{C}^\alpha}^2]$.

  Actually, we obtain a better estimation by looking at the moment of order $q$ of $J_{n,n'}$. We obtain the following bound, the proof of which is postponed to Section \ref{app:A} (Lemma \ref{le:ap:bound}):
  \begin{equation}
  \mathbb{E} [ J_{n,n'}^2 ]\leq C 2^{-\beta_0(n+n')}.
  \label{eq:boundJ}
  \end{equation}

  For $s<t\in \mathbb{D}=\bigcup_{n\geq 0} \mathbb{D}_n$, we define $s=\tau_0<\dots <\tau_N=t$ as follows\footnote{This sequence $\tau_0,\dots, \tau_N$ plays the same role as the one defined in  \cite{hairerFriz}, though the construction is not exactly the same.}. Let $m$ be the integer part of $-\log_2(t-s)$ and $r$ be the unique element of $\mathbb{D}_m\cap[s,t)$. Write
  \[s=r - \sum_{i=1}^{M} 2^{-\phi(i) },\qquad t=r + \sum_{i=1}^{P}  2^{-\psi(i) },\]
  with $\phi,\psi$ two stricly increasing functions from $\mathbb{N}$ to $\mathbb{N}$ with $\phi(1)\geq m+1, \psi(1)\geq m+1 $.
  Such a decomposition always exists. We set $N=M+P$, $\tau_{M}=r$. For $j\in\{0,\dots, M-1\}$, we set
  $\tau_j =r - \sum_{i=1}^{M-j}  2^{-\phi(i)} $. For $j\in\{M+1,\dots, M+P\}$, we set $\tau_j =r + \sum_{i=1}^{j-M }  2^{-\psi(i)} $. For each $j\in \{0,\dots, M+P\}\setminus\{M\}$, set $n_j$ such that $\tau_j\in \mathbb{D}_{n_j}\setminus\mathbb{D}_{n_j-1}$.
  Set also $n_M=m$.
  Then, it is easily seen that $(n_j)_{j\in\{0,M\}}$ is stricly decreasing, whilst
  $(n_j)_{j\in\{M,M+P\}}$ is stricly increasing. In particular, the sequence $(n_j)_{j\in\{0,M+P\}}$ takes each value at most twice, and takes no value smaller than $m$.
  In particular, for any sequence $(a_n)_{n\in \mathbb{N} }$ of positive terms,
  \[ \sum_{i=0}^N a_{n_i}\leq 2 \sum_{n=m}^{+\infty} a_n.\]

  For $u<v\in \mathbb{D}$, we define $u=\sigma_0<\dots <\sigma_{N'}=t$ and $(n'_j)_{j\in \{0,\dots, N'\} }$ in an identical way, and we set $m'$ the integer part of $-\log_2(v-u)$.

  Then, for any $\delta<\frac{\beta_0}{2}$, we claim that  $\mathbb{E} \big[ \max_{s<t,u<v} \big((t-s)^{-\delta} (v-u)^{-\delta} \M(R_{s,t,u,v})\big)^2 \big]$ is finite. Indeed, it is bounded above by
  \[\mathbb{E} \bigg[ \max_{s<t,u<v } \Big( \sum_{i=1}^{N}\sum_{j=1}^{N'}  (t-s)^{-\delta} (v-u)^{-\delta} \M(R_{\tau_{i-1},\tau_i,\sigma_{i-1},\sigma_i}) \Big)^2\bigg],\]
  which in turn is not greater than

  \begin{align*}
  \mathbb{E} \bigg[ \max_{s<t,u<v } \Big( \sum_{i=1}^{N}\sum_{j=1}^{N'}  2^{-\delta m} 2^{-\delta m'} J_{n_i, n'_j}\Big)^2\bigg] &\leq
  \mathbb{E} \big[ \max_{s<t,u<v  } 4 \big(\sum_{\substack{n\geq m\\ n'\geq m'}} 2^{  \delta (m+m')} J_{n,n'}  \big)^2 \big]\\
  &\hspace{-2cm}\leq 4 \Big(\sum_{\substack{n\geq m\\ n'\geq m'}} 2^{\delta (n+n')}    \mathbb{E} [J_{n,n'}^2]^{\frac{1}{2}}       \Big)^{2}
  \leq 4 \Big(\sum_{\substack{n\geq m\\ n'\geq m'}}  2^{(\delta -\frac{\beta_0}{2} ) (n+n') }  \Big)^{2} <+\infty.
  \end{align*}
  We denote by $C_{\M}$ the expectation of which we just proved that it is finite.

  Let us now look at $\mathbb{A}_{s,t}^Z$. 
  Remark that
  \[ \mathbb{E}[ | \M(T_{s,t})|^2]\leq C\mathbb{E}^Z[|T_{s,t}|^\nu]=C \frac{\mathbb{E}[|Z_t^1-Z_s^1|^\nu|Z_t^2-Z_s^2|^\nu] }{2^\nu}\leq C' (t-s)^{2\alpha\nu},\]
  with $C'=\frac{C}{2^\nu} \mathbb{E}[ \|Z\|_{\mathcal{C}^\alpha}^{2\nu} ]$, and $C$ the constant of Lemma \ref{le:firstBound}.   Since we assumed $\mathbb{E}[(\mathbb{A}^{Z}_{s,t})^2]  \leq  C (t-s)^{ 2 \xi}$, this implies
  \[ \Emu[\mathcal{A}^{Z}_{s,t}]  \leq  C  (t-s)^{2\min(\xi,\alpha\nu )}.\]
  Let us define $\mathbb{K}_n= \max_{t\in \mathbb{D}_n} |\mathcal{A}^{Z}_{t,t+2^{-n}}|$.
  Then 
  \[\Emu[ \mathbb{K}_n^2 ]\leq \sum_{t\in \mathbb{D}_n} \Emu[|\mathcal{A}^{Z}_{t,t+2^{-n}}|^2]\leq C 2^{(1-2\min(\xi,\alpha\nu ))n}.
  \]
  We set  $s=\tau_0<\dots <\tau_n=t$ as before.
  Then, for any $\beta'\in(0,\beta)$ and $\delta \in (\frac{\beta'}{2}, \frac{\beta}{2})$,
  \begin{align*}
  \Emu [ &\max_{s<t\in [0,1] } \big((t-s)^{-\beta'} \mathbb{A}_{s,t}^Z\big)^2 ]
  \leq
  \Emu [ \max_{s<t\in [0,1] }  \big( (t-s)^{-\beta'} \sum_{i=0}^N (\mathbb{A}^Z_{\tau_i,\tau_{i+1}}+ \M( R_{s,\tau_i,\tau_{i+1}})  )    \big)^2 ]\\
  &\leq 2 \Emu [ \max_{s<t\in [0,1] }  \big( (t-s)^{-\beta'} \sum_{i=0}^N \mathbb{A}^Z_{\tau_i,\tau_{i+1}} \big)^2 ]
  +2
  \Emu [ \max_{s<t\in [0,1] }  \big( \sum_{i=0}^N (t-s)^{-\beta'} \M( R_{s,\tau_i,\tau_{i+1}})      \big)^2 ]\\
  &\leq
  8 \Emu \big[\big( \sum_{n=0}^{+\infty}  2^{\beta' n}  \mathbb{K}_n\big)^2\big] \\
  &\hspace{1.5cm}+
  8 \Emu \Big[\max_{s<t\in [0,1] }\big((t-s)^{-\beta'}   \sum_{n=0}^{+\infty} \max_{u\in [s,t]} (s-u)^{\delta}
   2^{-\delta n}  (s-u)^{-\delta}
    2^{\delta n} \M( R_{s,u,u+2^{-n} })      \big)^2 \Big]\\
  &\leq
  8 \Big(\sum_{n=0}^{+\infty} 2^{\beta' n}  \Emu \big[\big(    \mathbb{K}_n\big)^2\big]^{\frac{1}{2}} \Big)^2\\
  &\hspace{3cm} +
  8 \Emu \Big[\max_{s<t\in [0,1] }\big(  \sum_{n=0}^{+\infty} 2^{(\beta'-\delta)n }
   2^{-\delta n} \max_{u\in [s,t]}
   (s-u)^{-\delta}
    2^{\delta n} \M( R_{s,u,u+2^{-n} })      \big)^2 \Big]\\
  &\leq
    C \Big(\sum_{n=0}^{+\infty} 2^{(\beta'+ \frac{1}{2}-\max(\xi,\alpha\nu))n } \Big)^2+
    8 \Big(  \sum_{n=0}^{+\infty} 2^{(\beta'-2\delta)n } C_{\M}  \big)^2 ]
  \Big)^2 <+\infty.
  %
  %
  %
  %
  \end{align*}
  This proves that, for all $\beta<\beta_1$, $\mathbb{P}$-almost surely, there exists a constant $C$ such that for all $(s,t)\in \Delta \cap \mathbb{D}^2 $, $|\mathbb{A}_{s,t}|\leq C (t-s)^\beta$.

 Our assumption is that $\mathbb A$ is separable, but with respect to a countable subset $I$ of $\Delta$ which we do not know. Let us explain how this general situation can be reduced to the dyadic situation that we treated above. Firstly, $I$ can be replaced by a countable set with a product structure, namely the set of all points of $\Delta$ which share each of their coordinates with a point of $I$. Thus, $I$ is the intersection with $\Delta$ of a set of the form $\tilde{\mathbb D}^2$, for some countable dense subset $\tilde{\mathbb D}$ of $[0,1]$.
 Then, we can write $\tilde{\mathbb{D}}$ as an infinite union $\bigcup_{N\in \mathbb{N}} \tilde{\mathbb{D}}_N$, in a way that mimics the decomposition of $\mathbb{D}$ that we used above, namely in such a way that the points of $\tilde{\mathbb{D}}_N$ are close enough to being evenly spaced for our arguments to work.


This being taken into account, the separability of $\mathbb{A}$ allows us to conclude that for all $(s,t)\in \Delta $, $|\mathbb{A}_{s,t}|\leq C (t-s)^\beta$ and the proposition is proved.
\end{proof}

\begin{corollary}
\label{coro:casYoung}
  Assume that $\gamma<\sqrt{2}$ and $Y$ is $\alpha$-H\"older continuous for $\alpha>\frac{1}{2}(1-\frac{\gamma^2}{4})^{-1}$.
  Then $\mathbb{A}^Y$ admits a separable modification which satisfies the Chen relation and which is $\beta$-regular for all
   \[
  \beta<\beta_0=\left\{
  \begin{array}{ll}
  \alpha\nu-1 &\mbox{if } \alpha\geq \gamma^{-2},\\
  2\alpha(1+\frac{\gamma^2}{4}) -2\gamma\sqrt{\alpha} &\mbox{if } \alpha\in [\frac{3-2\sqrt{2}}{2} \gamma^{-2}, \gamma^{-2}]\\
  \alpha\nu-\frac{1}{2} &\mbox{if } \alpha\leq\frac{3-2\sqrt{2}}{2} \gamma^{-2}.
  \end{array}
  \right.
  \]
 If $\gamma< 2(\sqrt{2}-1)$, this modification is $\beta'$-H\"older continuous for  $\beta'=\min((1-\sqrt{2}\gamma+\frac{\gamma^2}{4})\alpha , \beta)$.
\end{corollary}

\begin{proof}
  The condition  $\alpha>\frac{1}{2}(1-\frac{\gamma^2}{4})^{-1}$ implies $\alpha>\gamma^2(1+\frac{\gamma^2}{4})^{-2}$.
  Corollary \ref{coro:six4} and Lemma \ref{le:chenFaibleY} ensure that we can take $\xi=\alpha \nu-\epsilon$ for any $\epsilon>0$ in the conditions of Proposition \ref{le:six:preExten}. This suffices to conclude to the first family of properties.  For the H\"older regularity when $\gamma<2(\sqrt{2}-1)$, we also use Proposition \ref{le:chenFort} and then Lemma \ref{le:RegToHold}.
\end{proof}

For the equivalent result in the Brownian situation, we need to obtain a scaling relation for $\mathbb{A}^X$. It is obtained from the scaling properties of the Brownian motion and the measure $\M$.

\begin{lemma}
  \label{le:BMscale}
  Assume that $\gamma<\sqrt{4/3}$. There exists an increasing family of events $(E_n)_{n\geq 1}$ on $\Omega^X$, with $\PX(E_n)\to 1$, and a family of constants $(C_n)_{n\geq 1}$, such that for all $n\geq 1$,
  for all $(s,t)\in \Delta$,
  \[\mathbb{E}[\mathbbm{1}_{E_n} \mathbb{A}_{s,t}^2 ]\leq C_n (t-s)^\nu.\]
\end{lemma}
\begin{proof}
  For $N$ a positive integer and $(s,t)\in \Delta$, we set
  \[\mathcal{D}_{N,s,t}=\{z\in \R^2: \theta_{X_{|[s,t]}}(z)\geq N\}, \ \
  \mathcal{D}_{-N,s,t}=\{z\in \R^2: \theta_{X_{|[s,t]}}(z)\leq -N\}.\]

Assume first that $K$ is given by $K(z,w)=\log_+(|z-w|^{-1})$ and write $\M_0$ the associated measure.
  Let $E_n$ be the event $E_n=\{ \forall (s,t)\in \Delta,\ t-s<\frac{1}{n} \implies |X_t-X_s|\leq \frac{1}{2}\}$. Clearly, $\mathbb{P}(E_n)\to 1$.
  For $n\geq 1$ and $(s,t)\in \Delta$ with $t-s<\frac{1}{n}$,
  using the exact scale invariance of $\M_0$ (see Section \ref{app:A}), then Cauchy--Schwarz inequality, then scaling properties of the Brownian motion, we get
  \begin{align*}
  \EX[ \mathbbm{1}_{E_n}\Emu[ \mathbb{A}_{s,t}^2 ]]
  &=\EX\big[\mathbbm{1}_{E_n}\Emu \big[ \big( \sum_{N=1}^{+\infty} (\M_0(\mathcal{D}_{N,s,t})-\M_0(\mathcal{D}_{-N,s,t}))\big)^2\big]\\
  &=(t-s)^{\nu} \EX\big[\mathbbm{1}_{E_n}\Emu\big[\big( \sum_{N=1}^{+\infty} (\M_0((t-s)^{-\frac{1}{2}} \mathcal{D}_{N,s,t})-\M_0((t-s)^{-\frac{1}{2}}\mathcal{D}_{-N,s,t}))\big)^2\big]\big]\\
  &\leq (t-s)^{\nu} \mathbb{E}\big[\big( \sum_{N=1}^{+\infty} (\M_0((t-s)^{-\frac{1}{2}} \mathcal{D}_{N,s,t})-\M_0((t-s)^{-\frac{1}{2}}\mathcal{D}_{-N,s,t}))\big)^2\big]\big]\\
  &=(t-s)^{\nu} \mathbb{E}\big[\big( \sum_{N=1}^{+\infty} (\M_0( \mathcal{D}_{N,0,1})-\M_0(\mathcal{D}_{-N,0,1}))\big)^2\big]\big],
  \end{align*}
  and we know the latter sum to be convergent. For $t-s\geq \frac{1}{n} $, we can simply say that
  \[
  \EX[ \mathbbm{1}_{E_n}\Emu[ \mathbb{A}_{s,t}^2 ]]\leq (t-s)^\nu n^{\nu} \EX[ \mathbbm{1}_{E_n}\Emu[ \mathbb{A}_{s,t}^2 ]],
  \]
  so that the constant \[C_n=\max\Big(\mathbb{E}\big[\big( \sum_{N=1}^{+\infty} (\M_0( \mathcal{D}_{N,0,1})-\M_0(\mathcal{D}_{-N,0,1}))\big)^2\big], n^{\nu} \EX[ \mathbbm{1}_{E_n}\Emu[ \mathbb{A}_{s,t}^2 ]] \Big)\]
  works.

  If $K$ is now given by $K(s,t)=\log_+(|t-s|^{-1})+C$ for a constant $C$, the measure $\M_C$ associated is given by $\M_C(A)=e^{\Omega- \frac{1}{2}\mathbb{E}[\Omega^2] } \M_0(A)$, for $\Omega$ a centered Gaussian variable independent from $\M_0$ and with variance $C$. We conclude to this case from the previous one.

  Finally, for the general case $K(s,t)=\log_+(|t-s|^{-1})+g(x,y)$, let us recall that we assumed $g$ to be bounded. Let $C$ be its supremum. Then, the Kahane convexity inequalities (see for example \cite{garban2}, Appendix A) implies that $\mathbb{A}_{s,t}$ computed with $\M$ has a second moment which is less than the one computed with $\M_C$. This concludes the proof.
\end{proof}

\begin{corollary}
\label{coro:casMB}
  Assume $\gamma<2(\sqrt{2}-1)$. Let $X$ be a Brownian motion independent from $\M$. Then,
  $\mathbb{A}^X$ admits a modification which satisfies the Chen relation, is $\beta$-regular
for all $\beta<\min (  \frac{1}{2}-\frac{\gamma^2}{4}, 1+\frac{\gamma^2}{4}-\sqrt{2}\gamma)$, and $\beta$-H\"older continuous for all $\beta<\frac{1}{2}( 1+\frac{\gamma^2}{4}-\sqrt{2}\gamma)$.

\end{corollary}

\begin{proof}
  The condition $\alpha>\gamma^2(1+\frac{\gamma^2}{4})^{-2}$ in Proposition \ref{le:six:preExten} is satisfied, for $\gamma<2(\sqrt{2}-1)$, provided the H\"older exponent $\alpha$ is chosen sufficiently close to $\frac{1}{2}$.
  Lemmas \ref{le:chenFaibleY} and \ref{le:BMscale} ensure that the hypothesis of Proposition \ref{le:six:preExten} are fulfilled with $\xi=\frac{\nu}{2}$, and for all $\alpha<\frac{1}{2}$. We are on the case $\alpha<\gamma^{-2}$. The conclusion of Proposition \ref{le:six:preExten}, together with Proposition \ref{le:chenFort} and Lemma \ref{le:RegToHold}, gives the corollary.
\end{proof}


\begin{remark}
  For $\gamma<2(\sqrt{2}-1)$, the almost surely defined continuous extension of $\mathbb{A}^{X}_{s,t}$ allows us to define $\mathbb{A}^{X}_{\sigma, \tau}$ for any random  time $\sigma,\tau$. It thus allows us to define  $\mathbb{A}^{\mathcal{X}}_{s,t}$ for any process $\mathcal{X}$ obtained as a reparametrization of $X$, hence in particular for the Liouville Brownian motion, defined in \cite{garban2} or in \cite{Berestycki}.
\end{remark}

\begin{remark}
  It is plausible that the H\"older regularity is actually higher when one considers the Liouville Brownian motion $\mathcal{X}$ instead of $X$. Also, it is possible that a H\"older continuous curve $Y$ admits a reparametrization $\mathcal{Y}$ such that $\mathbb{A}^{\mathcal{Y}}$ has a higher regularity than $\mathbb{A}^Y$. Such a reparametrization should be obtained by `freezing' $Y$ when it lies on the set where $\M$  is large and by `speeding it up' when it is far from it.
\end{remark}


\section{Uniform estimates in $L^q$}
\label{app:A}

During the proof of Proposition \ref{le:six:preExten}, there is a point that was left aside, about the estimation of $\mathbb{E}[J_{n,n'}^2]$, where $J_{n,n'}$ is defined by \eqref{eq:defJnn}. The goal in this section is to prove the estimation \eqref{eq:boundJ}. Our starting point is the following combination of an elementary comparison of a maximum and a sum, and H\"older inequality: for all $q\in [2,\frac{4}{\gamma^2})$,
  \[
  \mathbb{E} [ J_{n,n'}^2 ]\leq \mathbb{E} [ J_{n,n'}^q ]^{\frac{2}{q}}\leq 2^{n+n'} \max\big\{  \mathbb{E}[\M(R_{i,i+2^{-n},j,j+2^{-n'}})^q]^{\frac{2}{q}}: i\in {\mathbb D}_n , j\in {\mathbb D}_{n'} \big\}
  .\]

What we need is thus a good uniform bound on the $q$-th moment of the $\M$-measure of a small rectangle. To do this, we study the way in which the $\M$-measure of a rectangle is affected by a smooth transformation of the plane, and prove that a rectangle of given area can be nicely and smoothly sent into a fixed square. Once this is done, we know that the $\M$-measure is not too different from the $\M$-measure of a subset of the fixed square, which gives us what we needed.

In order to understand how the measure $\M$ is affected by a smooth transformation of the plane, it is useful to think about its informal definition \eqref{eq:defmuheuri}. Pushing this expression forward by a diffeomorphism affects it in two ways: it changes the correlation structure of the field in the exponential, and introduces a Jacobian. In order to control the change in the correlation structure of the field, we will use Kahane's convexity inequalities, which compare the multiplicative Gaussian chaoses associated to two kernels which do not differ too much. The crucial point for us is that our transformation of the plane does not bring too far apart two points that were initially close. We must therefore control something like its Lipschitz norm. On the other hand, we also need the (inverse) Jacobian term that appears not to explode. How to map a possibly long and thin rectangle into a square in a way that satisfies these constraints is explained by Lemma \ref{le:six5}.

\begin{lemma}
  \label{le:six5}
  Let $n\geq 1$ be an integer. There exists a function $\phi:[0,2^{n}]\times [0,2^{-n}]\to[0,10]^2$
  that is injective, $10$-Lipschitz continuous,  piecewise $\mathcal{C}^1$, with Jacobian bounded below by $\frac{1}{10}$.
\end{lemma}
\begin{proof}


Let us split the rectangle $R=[0,2^{n}]\times [0,2^{-n}]$ into $2^{n}$ rectangles $R_1,\dots R_{2^n}$, each of which have sides of length
$1$ and $2^{-n}$. We map each rectangle to a domain shaped like an integral symbol (see Figure \ref{fig:int} below).
\begin{figure}[h!]
\begin{center}
\includegraphics[scale=0.8] {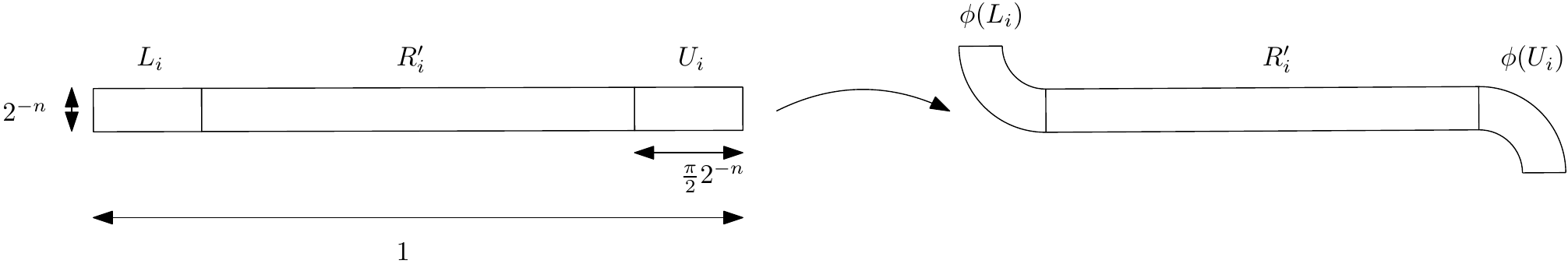}
\caption{\label{fig:int}Mapping of one rectangle.}
\end{center}
\end{figure}
For this, we decompose each $R_i$ into $L_i\sqcup R'_i\sqcup U_i$, where $L_i$ and $U_i$ are the two rectangles of width $2^{-n}\frac{\pi}{2}$ at the extremities. We  parametrize each of them linearly by $[0,2^{-n}\frac{\pi}{2}]\times[0,2^{-n}]$, and
we map them to a quarter of an annulus by $\phi: (2^{-n}x,2^{-n}y)\mapsto (2^{-n}(y+1)\sin(x), 2^{-n}(y+1)\cos(x)) $. This map (defined on the given rectangle) is $10$-Lipschitz, and its Jacobian determinant is uniformly bounded below by $\frac{1}{10}$.
We finally glue the pieces together, as shown by the following figure (Figure \ref{fig:spritz}).

\begin{figure}[h!]
\begin{center}
\includegraphics[scale=0.5]{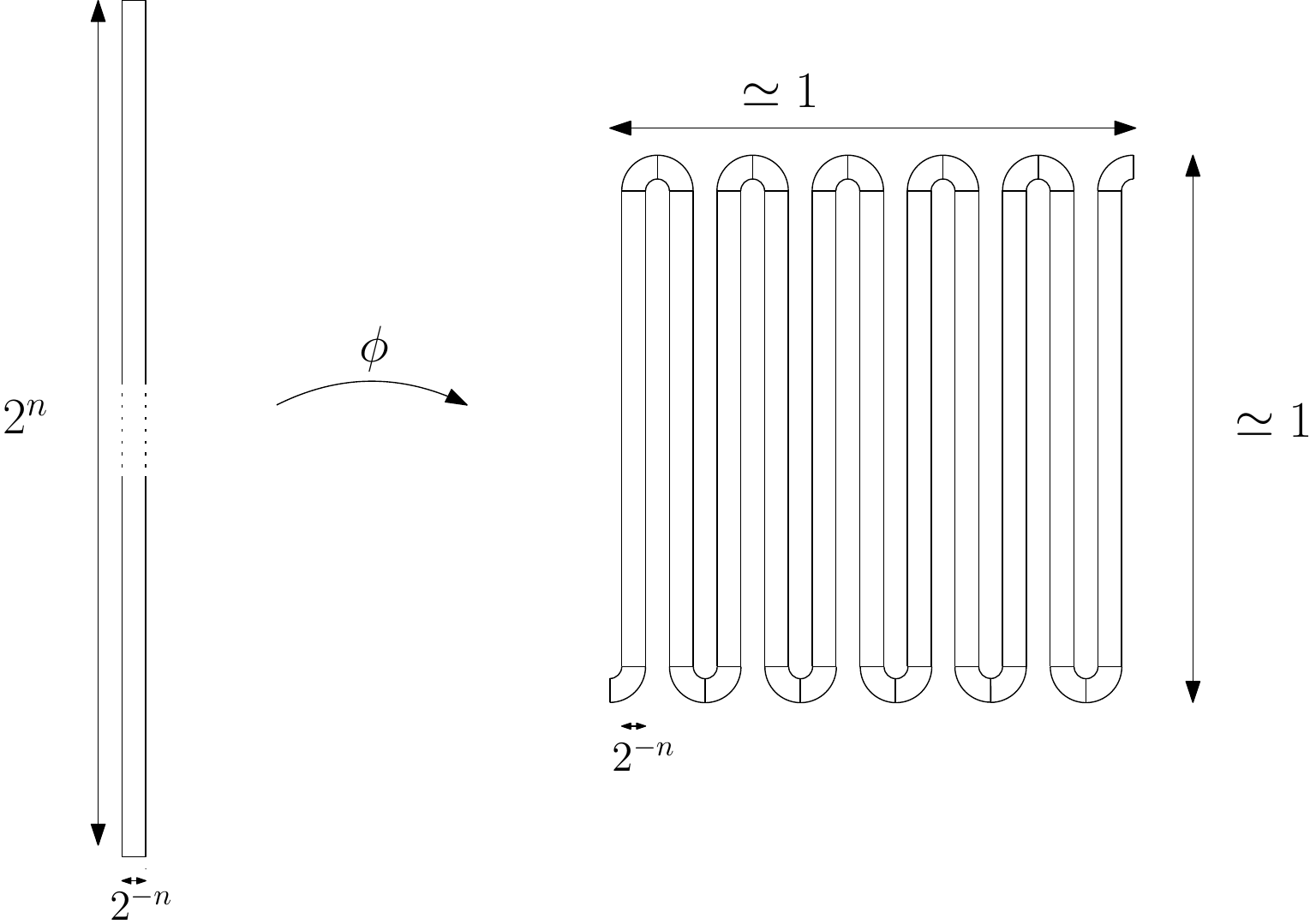}
\caption{\label{fig:spritz} gluing the pieces.}
\end{center}
\end{figure}
%
%
%
%
This concludes the proof.
\end{proof}

As explained at the beginning of the section, we will make use of Kahane's convexity inequality, which we state in a version adapted to our framework. The proof, which is a notoriously hard one, can be found in \cite{Kahane}. We assume that $\M$ and $\M'$ are two multiplicative gaussian chaos with kernels $K$ and $K'$ (and with the same intermittency parameter $\gamma$).

\begin{proposition}
Let $F:\R_+\to \R$ be some convex function such that
\[ \forall x\in \R_+,\ |F(x)|\leq M(1+|x|^\beta),\]
for some positive constants $M,\beta$.
Assume that $K(z,w)\leq K'(z,w)+C$ for some $C\geq 0$. Then, for all compact set $A$,
\[
\mathbb{E}\big[ F (\M(A))\big]\leq \mathbb{E}\big[ F ( e^{\sqrt{C} -\frac{C}{2}}  \M'(A))\big].
\]
\end{proposition}

We can now prove the following lemma.

\begin{lemma}
  \label{le:anB:2}
  For all $q\in [1,\frac{4}{\gamma^2})$, there exists a constant $C$ such that for all rectangle $R$ with sides of length $l,L\leq 1$,
  \[ \mathbb{E}[ \M(R)^q ]<C |R|^{\frac{\xi(q)}{2}},\]
  where $\xi$ is the so-called structure exponent $\xi(q)=(2+\frac{\gamma^2}{2})q- \frac{\gamma^2}{2}q^2$.
\end{lemma}
\begin{proof}

  If we restrict ourselves to squares, the result is standard and follows simply from scaling relations (see for example Theorem 2.14 in \cite{Rhodes}). We will reduce the more general case of rectangles to the case of squares thanks to the previous lemma.

  First, we fix $m$
   and $m'$ such that $l\in [2^{-m-1},2^{-m}], L\in [2^{-m'-1},2^{-m'}]$ where $l$ and $L$ are the length of the sides of $R$ (with $L>l$). Let $\phi$ be the map resulting from an application of Lemma \ref{le:six5} with $n=m'-m$, properly rotated, translated, and conjugated by a homothecy, so that it maps $R$ to a square $S$ of area $100 |R|$.


 Let $K_k$ be an increasing sequence of continuous covariance kernels converging pointwise towards $K$. Let also $(X_k)_{k\geq 1}$ be a sequence a continuous centered Gaussian field with covariance kernels $(K_k)_{k\geq 1}$, and $\M_k$ the associated measure, given by
 \[\M_k=e^{\gamma X_k -\frac{\gamma^2}{2} \Emu[X_k^2] } \d \lambda.\]
  Let also $\phi_*(X_k)$ be the centered Gaussian field defined on $\phi(R)$ by $\phi_*(X_k)_u= (X_k)_{\phi^{-1}(u)}$. 
  As explained at the beginning of this section, the push-forward of $\M_k$ by $\phi$ and the exponential of the Gaussian field $\phi_*(X_k)$ differ by a Jacobian term.

  We denote by $J:R\to \R^*_+$ the absolute value of the Jacobian determinant of $\phi$, and by $ \M_k^\phi$ the random measure on $\phi(R)$ given by
  \[\d \M_k^\phi(u)= \exp\big(\gamma (X_k)_{\phi^{-1}(u)}-\frac{\gamma^2}{2} \mathbb{E}[(X_k)_{\phi^{-1}(u)}^2] \big) \d\lambda (u).\]
  Observe that the push-forward of the measure $\M_k$ by $\phi$, that we denote by $\phi_*(\M_k)$, is related to $ \M_k^\phi$ by the relation
  \[\d \phi_*(\M_k) (u)=J(\phi^{-1}(u))^{-1} \d \M^\phi_k(u).\]
  Then
  \[\M_k(R)=\phi_*(M_k)(\phi(R))\leq 10 \M^\phi_k(\phi(R)),\]
  from which it follows that
  \[ \Emu[ \M_k(R)^q]\leq 10^q \Emu[ \M_k^{\phi}(\phi(R))^q].\]
  Let us now check that we can apply Kahane's inequality.

  For any pair of points $z,w\in R$, $|\phi(z)-\phi(w)|\leq 10 |z-w|$. Hence,
  $\log( |z-w|^{-1}) \leq \log( |\phi(z)-\phi(w)|^{-1}) +\log(10)$. It follows that there a constant $C$ such that for any $z,w\in R$, $K(z,w)\leq K(\phi(z),\phi(w))+C$.
  We denote by $\tilde{K}_k$ the kernel of $\phi_*(X_k)$. For any $u,v\in \phi(R)$,
  \[ \tilde{K}_k(u,v)= K_k(\phi^{-1}(u),\phi^{-1}(v))\leq K(\phi^{-1}(u),\phi^{-1}(v))\leq K(u,v)+C,\]
  for a constant $C$ which we allow to vary from line to line.  We now apply Kahane convexity inequality, and we deduce that
  \[
  \Emu[ \M_N(R)^q]\leq 10^q \Emu[ \phi_*(\M_N)(\phi(R))^q]\leq C \Emu[ \M(\phi(R))^q]\leq  C \Emu[ \M(S)^q].
  \]
  for some constant $C$. The theory of Gaussian multiplicative chaos ensures that $\Emu[ \M_N(R)^q]$ converges toward $\Emu[ \M(R)^q]$, so that

  \[\Emu[ \M(R)^q]\leq C  \Emu[ \M(S)^q]\leq C |S|^{\frac{\xi(q)}{2}}\leq C |R|^{\frac{\xi(q)}{2}}.\]
  This concludes the proof.
\end{proof}
\begin{remark}
We think that Lemma \ref{le:six5} can be extended to general measurable sets, from which Lemma \ref{le:ap:bound} would also extend to general set. We are even more strongly convinced that Lemma \ref{le:ap:bound} holds in such a generality, but we miserably failed to prove it despite a tremendous quantity of effort put into it.

The author discovered the paper \cite{wong} after writing the solution presented here. It is possible that the general result can be deduced from the estimates found in this paper.
\end{remark}

Finally, we can prove the following, with the notation of Proposition \ref{le:six:preExten}.

\begin{lemma}
\label{le:ap:bound}
There exists a constant $C$ such that for all integers $n,n'$,
\[ \mathbb{E} [ J_{n,n'}^2 ]\leq C 2^{-\beta_0(n+n')}.\]
\end{lemma}
\begin{proof}
    From the discussion at the beginning of the section, we know that it suffices to show that,
 for some $q\in [2,\frac{4}{\gamma^2})$, and $C>0$, for all $n\in \mathbb{N}$,
  \[ \max\big\{  \mathbb{E}[\M(R_{i,i+2^{-n},j,j+2^{-n'}})^q]^{\frac{2}{q}}: i\in {\mathbb D}_n , j\in {\mathbb D}_{n'} \big\} \leq C 2^{-(n+n')} 2^{-\beta_0(n+n')}
  .\]

  From Lemma \ref{le:anB:2}, for all $q\in [1,\frac{4}{\gamma^2}$, there exists $C$ such that
  \[
  \mathbb{E}[\M(R_{i,i+2^{-n},j,j+2^{-n'}})^q]^{\frac{2}{q}}\leq C |R_{i,i+2^{-n},j,j+2^{-n'}}|^{\frac{\xi(q)}{q}}= C 2^{-(n+n')\frac{\xi(q)}{q} }.
  \]

  For $\alpha\geq\frac{1}{\gamma^2}$, the bound is optimal at $q=2$. For $\alpha\in (\frac{\gamma^2}{4}, \frac{1}{\gamma^2}]$,  the bound is optimal at $q=\frac{2}{\gamma\sqrt{\alpha}}\in [2, \frac{4}{\gamma^2})$, and we get
  \begin{equation}
  \label{eq:secondMomentBis}
  \Emu [ J_{n,n'}^2 ]\leq C 2^{(2\gamma\sqrt{\alpha} -2\alpha -\frac{\alpha \gamma^2}{2})(n+n')}= C 2^{-\beta_0(n+n')}.
  \end{equation}
  This concludes the proof.
\end{proof}

\begin{remark}
  As opposed to a more classical situation, the optimal bound is not obtained by taking $q$ `as large as possible'. This is due to the non-linearity if the map $q\mapsto \xi(q)$. It would be interesting to know if the bound given by Lemma \ref{le:ap:bound} can be improved.
\end{remark}

\bibliographystyle{plain}
\bibliography{bib.bib}

\begin{thebibliography}{10}

\bibitem{Berestycki}
Nathana\"{e}l Berestycki.
\newblock Diffusion in planar {L}iouville quantum gravity.
\newblock {\em Ann. Inst. Henri Poincar\'{e} Probab. Stat.}, 51(3):947--964,
  2015.

\bibitem{Berestycki2}
Nathana\"{e}l Berestycki.
\newblock Introduction to the gaussian free field and liouville quantum
  gravity.
\newblock 2016.

\bibitem{daley}
Daryl~J. Daley and David Vere-Jones.
\newblock {\em An introduction to the theory of point processes. {V}ol. {II}}.
\newblock Probability and its Applications (New York). Springer, New York,
  second edition, 2008.
\newblock General theory and structure.

\bibitem{dellacherie}
Claude Dellacherie and Paul-Andr\'{e} Meyer.
\newblock {\em Probabilities and potential}, volume~29 of {\em North-Holland
  Mathematics Studies}.
\newblock North-Holland Publishing Co., Amsterdam-New York; North-Holland
  Publishing Co., Amsterdam-New York, 1978.

\bibitem{doob}
Joseph~L. Doob.
\newblock {\em Stochastic processes}.
\newblock Wiley Classics Library. John Wiley \& Sons, Inc., New York, 1990.
\newblock Reprint of the 1953 original, A Wiley-Interscience Publication.

\bibitem{hairerFriz}
Peter~K. Friz and Martin Hairer.
\newblock {\em A course on rough paths}.
\newblock Universitext. Springer, Cham, second edition, [2020] \copyright 2020.
\newblock With an introduction to regularity structures.

\bibitem{garban2}
Christophe Garban, R\'{e}mi Rhodes, and Vincent Vargas.
\newblock Liouville {B}rownian motion.
\newblock {\em Ann. Probab.}, 44(4):3076--3110, 2016.

\bibitem{gikhman}
Iosif~I. Gikhman and Anatoli~V. Skorokhod.
\newblock {\em The theory of stochastic processes. {I}}.
\newblock Classics in Mathematics. Springer-Verlag, Berlin, 2004.
\newblock Translated from the Russian by S. Kotz, Reprint of the 1974 edition.

\bibitem{Kahane}
Jean-Pierre Kahane.
\newblock Sur le chaos multiplicatif.
\newblock {\em Ann. Sci. Math. Qu\'{e}bec}, 9(2):105--150, 1985.

\bibitem{2DM}
Thierry L\'{e}vy.
\newblock Two-dimensional {M}arkovian holonomy fields.
\newblock {\em Ast\'{e}risque}, (329):172, 2010.

\bibitem{Rhodes}
R\'{e}mi Rhodes and Vincent Vargas.
\newblock Gaussian multiplicative chaos and applications: a review.
\newblock {\em Probab. Surv.}, 11:315--392, 2014.

\bibitem{houches}
R\'{e}mi Rhodes and Vincent Vargas.
\newblock {Gaussian multiplicative chaos and Liouville quantum gravity}.
\newblock In {\em {Stochastic processes and random matrices. Lecture notes of
  the Les Houches summer school. Volume 104, Les Houches, France, July 6--31,
  2015}}, pages 548--577. Oxford: Oxford University Press, 2017.

\bibitem{LAWA}
Isao Sauzedde.
\newblock L{\'e}vy area without approximation.
\newblock 2021.
\newblock \arXiv{2101.03992}.

\bibitem{bwe}
Isao Sauzedde.
\newblock Planar brownian motion winds evenly along its trajectory, 2021.
\newblock \arXiv{2102.12372}.

\bibitem{werner}
Wendelin Werner.
\newblock Sur les points autour desquels le mouvement brownien plan tourne
  beaucoup.
\newblock {\em Probab. Theory Related Fields}, 99(1):111--144, 1994.

\bibitem{wong}
Mo~Dick Wong.
\newblock Universal tail profile of {G}aussian multiplicative chaos.
\newblock {\em Probab. Theory Related Fields}, 177(3-4):711--746, 2020.

\end{thebibliography}

\end{document}